\documentclass[11pt, letterpaper, reqno]{amsart}

\usepackage{amsmath,amssymb,amscd,amsthm,amsxtra,amsfonts, amsbsy, mathrsfs}
\usepackage[dvips]{graphics,epsfig}
\usepackage[T1]{fontenc}
\usepackage[all]{xy}
\usepackage{url}

\usepackage{color}

\allowdisplaybreaks[2]

\newtheorem{theorem}{Theorem} [section]

\newtheorem{lemma}[theorem]{Lemma}
\newtheorem{proposition}[theorem]{Proposition}
\newtheorem{remark}{Remark}[section]

\newtheorem{definition}{Definition}[section]
\newtheorem{corollary}[theorem]{Corollary}

\newtheorem*{ackno}{Acknowledgements}

\DeclareMathOperator*{\supp}{supp}

\newcommand{\noi}{\noindent}
\newcommand{\N}{\mathbb{N}}

\newcommand{\M}{\mathcal{M}}
\newcommand{\Z}{\mathbb{Z}}
\newcommand{\R}{\mathbb{R}}
\newcommand{\T}{\mathbb{T}}
\newcommand{\C}{\mathbb{C}}

\def\l{\lambda}
\def\e{\varepsilon}

\newcommand{\dl}{\delta}

\newcommand{\eps}{\varepsilon}
\newcommand{\E}{\mathbb E}

\newcommand{\ld}{\lambda}

\newcommand{\pa}{\partial}
\newcommand{\pP}{\mathbf{P}}

\newcommand{\ft}{\widehat}
\newcommand{\wt}{\widetilde}

\newcommand{\dt}{\partial_t}

\newcommand{\jb}[1]
{\langle #1 \rangle}

\numberwithin{equation}{section}
\numberwithin{theorem}{section}

\let\Re=\undefined\DeclareMathOperator*{\Re}{Re}
\let\Im=\undefined\DeclareMathOperator*{\Im}{Im}

\begin{document}
\baselineskip = 14pt

\title[Almost sure global well-posedness for energy-critical NLW on $\R^d$, $d=4, 5$]{Almost sure global well-posedness for the energy-critical defocusing nonlinear wave equation on $\R^d$, $d=4$ and $5$}

\author[O.~Pocovnicu]
{Oana Pocovnicu}

\address{
Oana Pocovnicu\\
School of Mathematics\\
Institute for Advanced Study\\
Einstein Drive, Princeton\\ NJ 08540\\ USA}

\curraddr{
Department of Mathematics\\
Princeton University\\
Washington Road,
Princeton\\ NJ 08544\\ USA}

\email{opocovnicu@math.princeton.edu}



\thanks{This material is based upon work supported by the National Science Foundation under agreement No. DMS-1128155. Any opinions, findings, and conclusions or recommendations expressed in this material are those of the author and do not necessarily reflect the views of the National Science Foundation.}

\begin{abstract}
We consider the energy-critical defocusing nonlinear wave equation (NLW)
on $\R^d$, $d=4$ and $5$. 
We prove almost sure global existence and uniqueness for NLW with
rough random initial data in $H^s(\R^d)\times H^{s-1}(\R^d)$, with $0< s\leq 1$
if $d=4$, and $0\leq s\leq 1$ if $d=5$.
The randomization we consider is 
naturally associated with the Wiener decomposition
and with modulation spaces. The proof is based on  a probabilistic 
perturbation theory.
Under some additional assumptions, 
for $d=4$, we also prove the probabilistic continuous dependence of the flow 
with respect to the initial data
(in the sense proposed by Burq and Tzvetkov in \cite{BT3}).
\end{abstract}

\maketitle

\section{Introduction}

\subsection{Energy-critical nonlinear wave equations}

We consider the Cauchy problem for the energy-critical defocusing nonlinear wave equation (NLW) on $\R^d$, $d=4$ or $5$:
\begin{equation}\label{NLW}
\begin{cases}
\pa_{t}^2 u-\Delta u+F(u)=0 
\\
(u,   \pa_t u)\big|_{t = 0} = (\phi_0, \phi_1), 
\end{cases}
\quad \quad (t,x)\in\R\times\R^d, 
\end{equation} 

\noi
where 
$F(u)=|u|^{\frac{4}{d-2}}u$ and
$u$ is a real-valued function on $\R\times\R^d$.
Our main focus in this paper is to 
study the global-in-time behavior of solutions with {\it random} and {\it rough} initial data.

The flow of the equation  \eqref{NLW}  formally conserves  the  energy $E(u)$ defined by 
\begin{equation}
E(u) = E(u, \pa_t u) :=\int_{\R^d} \frac 12(\pa_t u)^2+\frac 12|\nabla u|^2+\frac{d-2}{2d}
|u|^{\frac{2d}{d-2}} dx.
\label{Zenergy}
\end{equation} 

\noi
We define 
the energy space $\mathcal{E}(\R^d)$ associated to \eqref{NLW} 
to be the space of pairs $(f, g)$ of real-valued functions of finite energy
\[\mathcal{E}(\R^d) : = \left\{(f,g): E(f,g):=\int_{\R^d} \frac 12g^2+ \frac 12|\nabla f|^2+ 
\frac{d-2}{2d}|f|^{\frac{2d}{d-2}}dx<\infty\right\}.\]

\noi
Then, in view of the Sobolev embedding
$\dot{H}^1(\R^d) \subset L^{\frac{2d}{d-2}}(\R^d)$, we immediately see that 
\[\mathcal{E}(\R^d)=\dot{H}^1(\R^d)\times L^2(\R^d).\]

It is well known that the NLW equation \eqref{NLW} on $\R^d$ enjoys several symmetries.
Of particular importance is the following 
 scaling invariance:
\begin{align}
u(t,x)\mapsto u_{\lambda}(t,x):=\lambda^{\frac{d-2}{2}} u(\lambda t, \lambda x).
\label{Zscaling}
\end{align}

\noi
Namely, if $u$ is a solution of equation \eqref{NLW},
then $u_{\lambda}$ is also a solution of \eqref{NLW} with rescaled initial data.
Notice that the following equality hods:
\begin{equation}
\|(u_\lambda (0), \pa_t u_\lambda(0))
\|_{\dot{H}^{s}(\R^d)\times \dot{H}^{s-1}(\R^d)}=
\ld^{s - 1}\|(u(0),\pa_t u(0))\|_{\dot{H}^{s}(\R^d)\times\dot{H}^{s-1}(\R^d)}.
\label{Zscaling2}
\end{equation}
One then defines the so-called scaling critical Sobolev index $s_c : = 1$
to be the index $s$ for which
the homogeneous $\dot H^{s}(\R^d)\times \dot H^{s-1}(\R^d)$-norm of $(u(0), \pa_t u(0))$ is invariant under  
the scaling \eqref{Zscaling}.
We notice that the critical space $\dot{H}^1(\R^d)\times L^2(\R^d)$ under the scaling
coincides with the energy space $\mathcal{E}(\R^d)$.
Moreover, 
the energy $E(u)$ defined in \eqref{Zenergy}  is also invariant under the scaling.
Therefore, we refer to the NLW \eqref{NLW} on $\R^d$
as {\it energy-critical}. 

Heuristically speaking, 
for an energy-critical NLW,
there is a delicate balance between the 
linear and nonlinear parts of the equation, 
and so this has made the analysis of such equations rather intricate.
Nonetheless, 
after an intensive effort which materialized in many articles on the subject, 
it is now known that the energy-critical defocusing nonlinear wave equations on $\R^d$, $d\geq 3$, 
are globally well-posed in the energy space 
and that all solutions in the energy space scatter.
The small energy data theory goes back to Strauss \cite{Strauss}, Rauch \cite{Rauch}, and Pecher \cite{Pecher}. 
The global regularity 
(referring to the fact that smooth initial data lead to smooth global solutions)
was proved in the works of Struwe \cite{Struwe}, Grillakis \cite{Grillakis90, Grillakis92}, and Shatah and Struwe \cite{Shatah_Struwe93}.
Regarding global well-posedness in the energy space, scattering, and global space-time bounds, we cite Shatah and Struwe \cite{Shatah_Struwe}, Kapitanski \cite{Kapitanski},
Ginibre, Soffer, and Velo \cite{Ginibre}, Bahouri and Shatah \cite{Bahouri_Shatah},
Bahouri and G\'erard \cite{Bahouri_Gerard}, Nakanishi \cite{Nakanishi1999, Nakanishi_scattering}, and Tao \cite{Tao}.

On the other hand, 
there are ill-posedness results for the energy-critical NLW 
below the scaling critical regularity $s_c = 1$.
When 
$d=3$ and $4$, Christ, Colliander, and Tao \cite{Christ_Colliander_Tao_main} 
proved that
the solution map of the energy-critical NLW fails to be continuous at zero in 
the $H^s(\R^d)\times H^{s-1}(\R^d)$-topology, for $0<s<1$.
See also \cite{Christ_Colliander_Tao, Alazard_Carles, Lebeau1, Lebeau, BK, Ibrahim} for 
other ill-posedness results for  nonlinear wave and Schr\"odinger equations.

In spite of the above deterministic ill-posedness results,
in this paper we consider the Cauchy problem \eqref{NLW} with
general initial data $(\phi_0,\phi_1)$ that do not belong to the energy space,
in a probabilistic manner.
More precisely, for $d=4$ or $5$, given 
\[(u_0,u_1)
\in H^s(\R^d)\times H^{s-1}(\R^d)\setminus
H^1(\R^d)\times L^2(\R^d)\]

\noi 
for some $s \in (0, 1)$, 
we consider its randomization 
$(u_0^\omega,u_1^\omega)$ 
defined in \eqref{R1} below.
This randomization $(u_0^\omega,u_1^\omega)$
 has the same regularity as $(u_0, u_1)$  and is not smoother 
in terms of 
differentiability, almost surely.
See Lemma \ref{lemma:Hs}.
Our first task is to construct local-in-time solutions 
of \eqref{NLW} with initial data $(\phi_0,\phi_1)=(u_0^\omega, u_1^\omega)$
in a probabilistic manner.
As in the work of Burq and Tzvetkov \cite{BTI}, 
the key point is the improved integrability properties
of the randomization
$(u_0^\omega,u_1^\omega)$.
This allows us to obtain improved  probabilistic local-in-time Strichartz estimates (Proposition \ref{proba_S}).
Then, we prove  almost sure local existence (with uniqueness)
by a simple fixed point argument in Strichartz spaces.
See Theorem \ref{ASLWP} below.
This almost sure local existence result is accompanied
by a probabilistic small data global result (Theorem \ref{scattering}).
In \cite{BOP1, BOP2}, the author with B\'enyi and  Oh
considered the same problem for
the cubic nonlinear Schr\"odinger equation (NLS) on $\R^d$, $d\geq 3$.

Once we construct local-in-time solutions almost surely, 
our next task is to extend them globally in time. 
Indeed, this probabilistic global-in-time argument is the main goal and novelty of this paper.
Informally speaking, 
there have been two kinds of
globalization arguments in the probabilistic setting:
(i) invariant measure argument
and (ii) certain probabilistic adaptations
of deterministic globalization arguments.

In \cite{Bourgain94}, 
Bourgain 
proved global existence
of NLS on $\T$ 
almost surely
with respect to the associated Gibbs measure.
The basic idea behind his argument
is to use (formal) {\it invariance} of this Gibbs measure
in the place of conservation laws.
He made this rigorous by exploiting the  invariance
of the finite dimensional Gibbs measures
for the associated finite dimensional approximations
of NLS.
This approach 
has been used in  subsequent works; see for example 
\cite{Bourgain96, Bourgain97, Tzv2006, BTIMRN, BTII, Oh09, DNLS, Suzzoni2013, BTT, Deng, R, BB3}.
In our context, 
however, there is no apparent (formally) 
invariant measure
and thus this approach is not appropriate.

Recently, 
there have been several 
probabilistic global-in-time arguments
in the absence of an invariant measure.
In \cite{Colliand_Oh}, 
Colliander and Oh 
introduced a 
{\it probabilistic high-low decomposition method} 
to prove almost sure global well-posedness 
of cubic NLS on $\T$ below $L^2(\T)$.
This is an adaptation
of Bourgain's high-low decomposition method \cite{Bourgain98}
in the probabilistic setting.
This approach was also used 
by L\"uhrmann and Mendelson
in the context of energy-subcritical NLW on $\R^3$.

In \cite{BT3}, Burq and Tzvetkov considered
(energy-subcritical) cubic NLW on $\T^3$ and 
proved probabilistic global well-posedness
for rough random data
in $H^s(\T^3)\times H^{s-1}(\T^3)$, $0\leq s<1$.
The main ingredient is to establish
a {\it probabilistic a priori bound on the energy}.\footnote{
The argument for obtaining almost sure global existence
when $s = 0$ is more involved.}
This is in some sense a probabilistic analogue
of the fact that a conservation law yields
(subcritical) global well-posedness in the
deterministic theory.

Probabilistic a priori bounds have also been combined
with a {\it probabilistic compactness method}.
Using this strategy,   
Burq, Thomann, and Tzvetkov obtained in \cite{BTT2}
almost sure existence 
of global solutions of 
the (energy-critical and energy-supercritical) 
cubic wave equation on $\T^d$, $d\geq 4$,
with rough random data.
As in the deterministic setting, 
the compactness method 
does not yield the uniqueness of such solutions. 
Note that this work was inspired by an earlier work 
of Nahmod, Pavlovi\'c, and Staffilani \cite{NPS}.
They considered  
the Navier-Stokes equations on $\T^2$ and $\T^3$
and proved almost sure existence of global weak solutions 
with rough random data. 
These weak solutions were shown 
to be unique on
 $\T^2$.

In the following, 
we use
a {\it probabilistic perturbation theory}
to prove 
almost sure
global well-posedness\footnote{Here, almost sure global well-posedness 
means almost sure global existence, uniqueness, and
a weak form of continuous dependence (see \cite{Bourgain96} and \cite{Colliand_Oh}).
This is not to be confused with probabilistic Hadamard global well-posedness that we describe below.}
for 
the energy-critical NLW \eqref{NLW} on $\R^d$
with rough random data in $H^s(\R^d)\times H^{s-1}(\R^d)$, 
where $0 < s\leq  1$ if $d=4$, and $0\leq s\leq 1$ if $d=5$.  
See Theorem \ref{main} below. 
This is the first instance when a perturbation theory 
is successfully applied  in the probabilistic setting 
to yield almost sure global existence.
See also \cite{BOP2}.
Perturbation theory has played an important
 role in the study of deterministic energy-critical NLW and NLS.
 Moreover, 
perturbation theory has been previously used 
to prove global well-posedness
for various other equations
in the deterministic setting.
See for example \cite{CKSTT, KM, TVZ, KOPV}.  
In \cite{TVZ}, Tao, Vi\c{s}an, and Zhang used
perturbation theory to prove (among other results) global well-posedness
of NLS with a combined power nonlinearity,
 one of
the powers being energy-critical, while the other is energy-subcritical.
Such an equation can be viewed as a perturbation of the energy-critical NLS,
with the smallness of the error coming
from the subcritical nature of the other power nonlinearity.
Now, let us turn our attention to equation
\eqref{NLW} with random initial data 
$(\phi_0,\phi_1)=(u_0^{\omega}, u_1^{\omega})$.
By denoting the linear\footnote{See \eqref{Zlinear} for the precise definition of the linear propagator $S(t)$.}
 and nonlinear parts of the solution 
$u^\omega$
of \eqref{NLW}
by $z^{\omega}(t)=S(t)(u_0^{\omega}, u_1^{\omega})$
and $v^\omega(t):=u^\omega(t)-z^\omega(t)$,
the equation \eqref{NLW} reduces to 
\begin{equation}\label{v}
\begin{cases}
\pa_{t}^2 v^\omega-\Delta v^\omega+F(v^\omega+z^\omega)=0\\
(v^\omega,\pa_t v^\omega)\big|_{t=0}=(0,0).
\end{cases}
\end{equation} 

\noi
Namely, the nonlinear part $v^\omega$ satisfies the energy-critical NLW
with a perturbation.
The crucial point in our approach is the fact that
the  error $F(v^\omega+z^\omega)-F(v^\omega)$ 
can be made small (on short time intervals)
thanks to the improved local-in-time Strichartz estimates
satisfied by 
 the {\it random} linear  part $z^\omega$.

Another essential  ingredient 
for an actual implementation of 
the probabilistic perturbation theory
is 
a probabilistic a priori bound on the energy of the nonlinear part $v^\omega$ of the solution $u^\omega$
on each finite time interval.  See  Proposition \ref{prop:energy}.
The 
probabilistic energy bound that we use here 
is the analogue of that obtained by Burq and Tzvetkov in  
\cite{BT3} for the cubic NLW on $\T^3$.
In \cite{BOP2}, the author with B\'enyi and Oh
applied 
a similar probabilistic perturbation theory
in the context of the defocusing
cubic NLS on $\R^d$, $d\geq 3$.
The global-in-time result in \cite{BOP2}, however, 
is conditional, 
even for the energy-critical cubic NLS on $\R^4$.
On $\R^4$, this failure is due to the fact 
that we do not have a probabilistic 
 a priori energy bound 
on the nonlinear part of a solution.
The key advantage of NLW, in comparison to NLS, 
is the presence of the term $\int \frac 12 (\pa_t v)^2dx$
in the energy.
We also emphasize the importance of the specific power nonlinearity 
in obtaining the probabilistic energy bound. 
It would be interesting to extend 
our almost sure global well-posedness result
to the energy-critical quintic NLW on $\R^3$
due to its physical relevance.
In view of Remark \ref{HypA} below,
such 
almost sure global existence would 
follow 
once we establish an analogous
probabilistic energy bound on $\R^3$.

The almost sure global well-posedness described in Theorem \ref{main}
refers to almost sure global existence, uniqueness,
and a weak form of continuous dependence
(as in \cite{Bourgain96} and \cite{Colliand_Oh}. See Remark \ref{REM:Zconti} (ii) below.)
Recently, 
Burq and Tzvetkov introduced in \cite{BT3}
the stronger notion of {\it probabilistic Hadamard global well-posedness}.
This  refers to almost sure global existence and uniqueness,
accompanied by {\it probabilistic continuity} of the flow with respect to the random initial data.
In Theorem \ref{proba_cont}, we 
prove, for $d=4$, that the flow of equation \eqref{NLW} 
is continuous in probability, 
under some extra assumptions.
This 
allows us to establish probabilistic Hadamard global well-posedness 
of the energy-critical defocusing cubic NLW \eqref{NLW} on $\R^4$ in $H^s(\R^4)\times H^{s-1}(\R^4)$, $0<s\leq 1$, in the sense of \cite{BT3}.


\medskip

\subsection{Randomization adapted to the Wiener decomposition and modulation spaces}

Starting with the works of Bourgain \cite{Bourgain96} and Burq and Tzvetkov \cite{BTI}, 
there have been 
many results on 
probabilistic constructions of solutions
of evolution equations 
with random initial data.
Many of the probabilistic results in the literature are on compact manifolds $M$,
where there is a countable basis $\{e_n\}_{n\in\N}$ of $L^2(M)$ consisting of
eigenfunctions of the Laplace-Beltrami operator.
This gives a natural way to introduce a randomization.
Given $u_0(x)=\sum_{n=1}^\infty c_n e_n(x)\in H^s(M)$,
one defines its randomization by
\begin{align}
u_0^\omega (x)=\sum_{n=1}^\infty g_n(\omega) c_n e_n(x),
\label{Ziv}
\end{align}
where $\{g_n\}_{n\in\N}$
is a sequence of independent random variables.
On $\R^d$, however,
 there is no countable
basis of $L^2(\R^d)$ consisting of eigenfunctions of the Laplacian.
In order to overcome this difficulty, several approaches have been introduced. 
For example, randomizations analogous to \eqref{Ziv} have been considered
with respect to a countable basis of $L^2(\R^d)$
consisting of eigenfunctions of the Laplacian with a confining potential, 
such as the harmonic oscillator $-\Delta + |x|^2$. 
Some results for the NLS with a harmonic potential were then transferred to the usual NLS on $\R^d$
via the lens transform. See \cite{Thomann, BTT, Deng, Poiret}.
Another approach consisted in working on the unit sphere $\mathbb{S}^3$,
and then transferring results to $\R^3$ via the Penrose transform. See 
\cite{Suzzoni2011, Suzzoni2013, Suzzoni2012}.

In this paper,  
we use a simple randomization for functions on $\R^d$,
naturally associated to the Wiener decomposition and modulation spaces.
This 
seems to be quite canonical
from the point of view of 
the time-frequency analysis \cite{Gr}.
Such a randomization was previously used in \cite{LM, BOP1, BOP2}. See also \cite{ZF}.

Given  $d\geq 1$, 
let $Q_n$ be  the unit cube $Q_n:=n+\left[-\frac 12, \frac 12 \right)^d$ centered at $n\in \Z^d$.
Then, 
 consider the uniform partition of the frequency space 
 $\R^d = \bigcup_{n \in \Z^d} Q_n$, 
 commonly referred to as 
the~{\it Wiener decomposition} \cite{W}.
Noting that $\sum_{n \in \Z^d} \chi_{Q_n}(\xi) \equiv 1$, we have 
\[ u = \sum_{n \in \Z^d} \chi_{Q_n} (D) u 
=  \sum_{n \in \Z^d} \chi_{Q_0} (D - n) u, \]

\noi
where $ \chi_{Q_n} (D) u  : = \mathcal{F}^{-1}\big[ \chi_{Q_n} \ft u\big]$.
In the following, we consider a smoothed version of this decomposition.

Let $\tilde{\psi}\in\mathcal{S}(\R^d)$
be a real-valued function such that
$\tilde\psi (\xi)\geq 0$,
$\tilde{\psi}(-\xi)=\tilde{\psi}(\xi)$ for all $\xi\in\R^d$ and
\begin{align*}
\tilde{\psi}(\xi)=
\begin{cases}
1, \textup{ if } \xi \in Q_0, \\
0, \textup{ if } \xi\notin [-1,1]^d.
\end{cases}
\end{align*}

\noindent
Set
\[\psi(\xi):=\frac{\tilde{\psi}(\xi)}{\sum_{n\in\Z^d}\tilde{\psi}(\xi-n)}.\]
Then, $\psi \in \mathcal{S}(\R^d)$ is a real-valued function, 
$0\leq \psi(\xi)\leq 1$ for all $\xi\in\R^d$, $\supp \psi \subset [-1,1]^d$, 
\begin{equation}
\psi(-\xi)=\psi(\xi)
\label{Zsym}
\end{equation}

\noi
 for all $\xi\in\R^d$, and $\sum_{n \in \Z^d} \psi(\xi -n) \equiv 1$.
Now, we define the Fourier multiplier
\begin{equation}\label{psi_multiplier}
\psi(D-n)u(x)=\int_{\R^d}\psi(\xi-n)\widehat u(\xi)e^{2\pi ix\cdot \xi}\,d\xi.
\end{equation}
Then, any function $u$ on $\R^d$ can be written as
\begin{equation}
 u = \sum_{n \in \Z^d} \psi(D-n) u.
\label{Ziv2}
 \end{equation}

\noi
From the symmetry condition \eqref{Zsym} imposed on $\psi$, we have 
\begin{equation}\label{psi}
\overline{\psi (D+n)u}=\psi (D-n)u \quad \text{ for real-valued } u,
\end{equation}
since $\overline{\ft u(\xi)}=\ft{u}(-\xi)$.
In particular, $\psi (D)u$
is real-valued for real-valued $u$.

\begin{remark}\rm

The modulation spaces  introduced by Feichtinger \cite{Fei} 
are naturally associated to the uniform decomposition \eqref{Ziv2}.
Indeed,  the modulation space $M^{p, q}_s$, 
 $0<p,q\leq \infty$, $s\in\mathbb R$, 
 consists of all tempered distributions $u\in\mathcal S'(\R^d)$ for which the (quasi) norm
\begin{equation}
\|u\|_{M_s^{p, q}(\R^d)} := \big\| \jb{n}^s \|\psi(D-n) u
\|_{L_x^p(\R^d)} \big\|_{\ell^q_n(\mathbb{Z}^d)}
\label{Zmod}
\end{equation}

\noi
is finite. 
Compare \eqref{Zmod} with the definition
of the Besov spaces which are naturally associated 
to the decomposition of the frequency space into dyadic annuli.
\end{remark}

In the following, we introduce the randomization 
adapted to the uniform decomposition \eqref{Ziv2}.
Let $\{g_{n,j}\}_{n \in \Z^d,  j=0,1}$ be a sequence of mean zero complex-valued random variables
on a probability space $(\Omega, \mathcal{F}, P)$
such that $g_{-n,j}=\overline{g_{n,j}}$
for all $n\in\Z^d$, $j=0,1$.
Assume also that $\left\{g_{0,j}, \text{Re}\, g_{n,j}, \text{Im}\, g_{n,j}\right\}_{n\in\mathcal I, j=0,1}$ 
are independent
and endowed
with probability distributions\footnote{The probability distribution $\mu$
of a real
random variable $g$ is an induced probability measure defined by
$\mu(A):=P(\{\omega \in \Omega: g(\omega)\in A\})$
for all measurable sets $A\subset \R$.}
 $\mu_{0,j}$, $\mu_{n,j}^{(1)}$,  and $\mu_{n,j}^{(2)}$.
Here,  the index set $\mathcal I$ is given by
\begin{equation}\label{index}
\mathcal I=\bigcup_{k=0}^{d-1}\Z^k\times \Z_+^\ast \times\{0\}^{d-k-1}
\end{equation}
and is such that $\Z^d=\mathcal I\cup (-\mathcal I)\cup \{0\}$.

For functions $u_0$, $u_1$ on $\R^d$,
we  define the \emph{Wiener randomization} $(u_0^\omega, u_1^\omega)$
of $(u_0,u_1)$ by
\begin{equation}
(u_0^\omega, u_1^\omega) : = 
\bigg(\sum_{n \in \Z^d} g_{n,0} (\omega) \psi(D-n) u_0,
\sum_{n \in \Z^d} g_{n,1} (\omega) \psi(D-n) u_1\bigg).
\label{R1}
\end{equation}

\noi
Note that, if $u_0$ and $u_1$ are real-valued,
then their randomizations $u_0^\omega$ and $u_1^\omega$
are also real-valued.
More precisely, by \eqref{psi}, we have
\begin{align}\label{real_randomization}
u_j^\omega&=g_{0,j} \psi(D)u_j+\sum_{n\in \mathcal I}g_{n,j}\psi (D-n)u_j + g_{-n,j}\psi (D+n)u_j\notag\\
&=g_{0,j} \psi(D)u_j+2{\rm Re} \sum_{n\in \mathcal I}g_{n,j} \psi(D-n)u_j, \quad j=0,1.
\end{align}

In the sequel, we make the following assumption: there exists $c>0$ such that
\begin{equation}
\int_{\R}e^{\gamma x}d\mu_{0,j}\leq e^{c\gamma^2}
\text{ and }
\int_{\R} e^{\gamma x } d \mu_{n,j}^{(k)}(x) \leq e^{c\gamma^2}
\label{cond}
\end{equation}
	
\noindent
for all $\gamma \in \R$, $n \in \Z^d$, $j=0,1$, and $k = 1, 2$.
Note that \eqref{cond} is satisfied by
standard complex-valued Gaussian random variables,
standard Bernoulli random variables,
and any random variables with compactly supported distributions.

It is easy to see that, if $(u_0,u_1) \in H^s(\R^d)\times H^{s-1}(\R^d)$
for some $s \in \R$,
then  the randomized function $(u_0^\omega, u_1^\omega)$ is
almost surely  in $H^s(\R^d)\times H^{s-1}(\R^d)$. 
See Lemma \ref{lemma:Hs} below.
One can also show that there is no smoothing upon randomization
in terms of differentiability 
(see, for example, Lemma B.1 in \cite{BTI}).
Instead, the main point of this randomization is
its improved integrability:
 if $u_j \in L^2(\R^d)$, $j=0,1$,
then  the randomized function $u_j^\omega$ is
almost surely  in $L^p(\R^d)$ for any finite $p \geq 2$.
Such results for random Fourier series
are known as  Paley-Zygmund's theorem \cite{PZ}.
See also \cite{Kahane, AT, BOP1}.


\medskip

\subsection{Main results}

In this subsection,  we state the main results of the paper.
In the following, 
we denote the linear propagator
of the linear wave equation
by $S(t)$.
Namely,  
the solution of the linear wave equation 
with initial data $(u, \pa_t u)\big|_{t = 0}=(u_0, u_1)$
is denoted by
\begin{equation}\label{Zlinear}
S(t)\left(u_0, u_1\right):=\cos(t|\nabla|)u_0+\frac{\sin (t|\nabla|)}{|\nabla|}u_1.
\end{equation}

We first state an almost sure local well-posedness result.
\begin{theorem}[Almost sure local well-posedness]
\label{ASLWP}
Let $d=4$ or $5$ and $0\leq s\leq 1$. 
Given  $(u_0, u_1) \in H^s(\R^d)\times H^{s-1}(\R^d)$,
let $(u_0^\omega, u_1^\omega)$
be the randomization defined in \eqref{R1},
satisfying \eqref{cond}. 

\medskip

\noindent
{\rm (a)} The energy-critical defocusing NLW \eqref{NLW}
on $\R^d$ admits almost surely a unique local solution
with initial data $(u_0^\omega, u_1^\omega)$ at $t=0$.
More precisely, there exist $C,c,\gamma>0$ such that for each $T$ sufficiently small,
there exists a set $\Omega_T\subset \Omega$ with the following properties:

\medskip
\noindent
{\rm(i)} $P(\Omega_T^c)<C\exp(-\frac{c}{T^\gamma})$.

\medskip
\noindent
{\rm(ii)} For each $\omega\in \Omega_T$, there exists a unique solution $u^\omega$ of equation \eqref{NLW}
with $(u^\omega, \pa_t u^\omega)\big|_{t = 0}
=(u_0^\omega, u_1^\omega)$ in the class:
\begin{align*}
\big(S(t)(u_0^\omega, u_1^\omega), \pa_tS(t)(u_0^\omega, u_1^\omega)\big)&+C\left([-T,T]; H^1(\R^d)\times L^2(\R^d)\right)\\
&\subset C\big([-T,T]; H^s(\R^d)\times H^{s-1}(\R^d)\big).
\end{align*}

\noi
Here,  uniqueness  holds
in a ball centered at 
$S(\cdot)(u_0^\omega, u_1^\omega)$ 
in 
$C([-T,T]; \dot{H}^1(\R^d))\cap 
L^{\frac{d+2}{d-2}}\Big([-T,T]; L^{\frac{2(d+2)}{d-2}}(\R^d)\Big)$.

\medskip

\noindent
{\rm(b)} Let $(w_0,w_1)\in H^1(\R^d)\times L^2(\R^d)$.
Then, there exists positive $T' = T' \big(T,w_0,w_1\big)<T$ such that for all $\omega\in\Omega_{T}$,
the energy-critical defocusing NLW on $\R^d$ admits a unique solution
with initial data
\[(u,\pa_t u)\big|_{t=t_\ast}=\big(S(t_\ast)(u_0^\omega, u_1^\omega), \pa_t S(t_\ast)(u_0^\omega, u_1^\omega)\big)+(w_0,w_1),\]
in the class
\[(u^\omega,\pa_t u^\omega) \in C\big([t_\ast-T',t_\ast+T']; H^s(\R^d)\times H^{s-1}(\R^d)\big),\]

\noindent
as long as $[t_\ast-T',t_\ast + T']\subset [-T,T]$.
Here,  uniqueness holds in a ball centered at 
\[S(\cdot-t_\ast)\Big(\big(S(t_\ast)(u_0^\omega, u_1^\omega), \pa_t S(t_\ast)(u_0^\omega, u_1^\omega)\big)+ (w_0,w_1)\Big)\]

\noindent
in $C\big([t_\ast-T',t_\ast+T'];\dot{H}^1(\R^d) \big)\cap L^{\frac{d+2}{d-2}}
\Big([t_\ast-T',t_\ast+T'];L^{\frac{2(d+2)}{d-2}}(\R^d) \Big)$.
\end{theorem}

We prove
Theorem \ref{ASLWP} (a)
 by considering the equation \eqref{v} 
 satisfied by the nonlinear part $v^\omega$ of a solution $u^\omega$,
and by viewing the linear part $z^{\omega}(t)=S(t)(u_0^{\omega}, u_1^{\omega})$
as a random forcing term.
We then run a simple fixed point argument in Strichartz spaces.
The  improved local-in-time Strichartz estimates in Proposition \ref{proba_S}
play an essential role.
Part (b) is essentially a corollary of part (a),
stating that the almost sure local existence and uniqueness still 
hold for more general initial data.

We note that the proofs 
of
the analogues of Theorem \ref{ASLWP} 
for NLS
on $\R^d$ and $\T^d$
in \cite{NS, BOP1, BOP2} are 
more intricate.
Lastly, notice that Theorem \ref{ASLWP}
is of a local-in-time nature and hence it also 
holds for the energy-critical focusing NLW on $\R^d$, $d=4$ and $5$.
The same comment applies to 
the probabilistic small data global theory (Theorem \ref{scattering} below).

\begin{remark} \label{REM:Zconti}\rm
While Theorem \ref{ASLWP} does not yield the continuous dependence of the flow 
of equation \eqref{NLW} on $\R^d$, $d=4$ and $5$, in 
$H^s(\R^d) \times H^{s-1}(\R^d)$ with $0\leq s< 1$,
we can modify the proof of Theorem \ref{ASLWP}
to obtain a mild form of continuous dependence. 
More precisely, 
we first fix ``good'' initial data 
$(u_0^\omega, u_1^\omega)$
such that Theorem \ref{ASLWP} yields
the corresponding solution of equation \eqref{NLW}
on the time interval $[-T, T]$ for some $T>0$.
Next, we consider $(\wt u_0, \wt u_1) \in H^s(\R^d) \times H^{s-1}(\R^d)$
such that $\|(\wt u_0, \wt u_1)- (u_0^\omega, u_1^\omega)\|_{H^1\times L^2 }\ll 1$.
Then, by modifying the proof of Theorem \ref{ASLWP}, 
we can construct a solution $\wt u$
of \eqref{NLW} on $[-cT, cT]$ with $(\wt u, \dt \wt u)\big|_{t = 0} 
= (\wt u_0, \wt u_1)$.
Moreover, we have 
\[\big\| (\wt u (t) , \dt \wt u(t)) - (u^\omega(t), \dt u^\omega(t))\big\|_{H^1(\R^d)\times L^2(\R^d)}
\leq
C\big\| (\wt u_0 ,  \wt u_1) - (u_0^\omega,  u_1^\omega)\big\|_{H^1(\R^d)\times L^2(\R^d)}\]

\noi
for $|t| \leq cT$.
See the works of Bourgain \cite{Bourgain96} and Colliander and Oh \cite{Colliand_Oh}
for related discussions.

For $d=4$, if we replace the smooth cutoff $\psi$ in the definition of the 
Wiener randomization \eqref{R1}
by the characteristic function $\chi_{Q_0}$, we obtain more.
Namely, Theorem \ref{proba_cont} below
yields the probabilistic continuous dependence
of the flow of equation \eqref{NLW} on $\R^4$ in $H^s(\R^4)\times H^{s-1}(\R^4)$,
$0<s\leq 1$.
\end{remark}

We now turn to the global-in-time behavior
of the local solutions constructed above. 
The same nonlinear estimates as in the proof of Theorem \ref{ASLWP}
together with the improved global-in-time Strichartz estimate (Lemma \ref{globalS})
 yield the following small data global result.

\begin{theorem}[Probabilistic small data global theory]\label{scattering}
Let $d=4$ or $5$ and $\frac{(d+1)(d-2)}{(d-1)(d+2)}\leq s\leq 1$. Given $(u_0, u_1) \in \dot{H}^s(\R^d) \times \dot{H}^{s-1}(\R^d)$, 
let $(u_0^{\omega}, u_1^\omega)$ be the randomization defined in \eqref{R1}, satisfying \eqref{cond}.
Then, there exist $C,c>0$ such that
for each $0<\eps\ll 1$
there exists a set $\Omega_\eps$
with the following properties:

\noindent
{\rm (i)} $P(\Omega_\eps^c)\leq C\exp
\bigg(-\frac{c}{\eps^2\big(\|u_0\|_{\dot{H}^s}+\|u_1\|_{\dot{H}^{s-1}}\big)^2}\bigg)\to 0$ as $\eps\to 0$.

\noindent
{\rm (ii)} For each $\omega\in \Omega_\eps$,
there exists a unique global solution $u^\omega$
of the energy-critical defocusing NLW \eqref{NLW} on $\R^d$ with initial data
\[(u^\omega, \pa_t u^\omega)\big|_{t=0}=(\eps u_0^\omega, \eps u_1^\omega)\]
in the class
\[\big(S(t)(\eps u_0^\omega,\eps u_1^\omega), \pa_t S(t)(\eps u_0^\omega,\eps u_1^\omega)\big)+C(\R; \dot{H}^1(\R^d)\times L^2(\R^d)).\]

\noi

\noindent
{\rm (iii)} Scattering holds for each $\omega\in \Omega_\eps$.
More precisely,
for each $\omega \in \Omega_\eps$,
there exists a pair
 $(v^\omega_{0,\pm}, v^\omega_{1,\pm}) \in \dot{H}^1(\R^d)\times L^2(\R^d)$
such that 
\[\Big\|\big(u^\omega(t),\pa_t u^\omega(t)\big)-
\Big(S(t)\big(\eps u_0^\omega+v^\omega_{0,\pm}, \eps u_1^\omega +v^\omega_{1,\pm}\big), \pa_tS(t)\big(\eps u_0^\omega+v^\omega_{0,\pm}, \eps u_1^\omega +v^\omega_{1,\pm}\big)\Big)\Big\|_{\dot{H}^1\times L^2}\to 0,\]

\noi
as $t\to\pm\infty$.
\end{theorem}

We refer to \cite{LM, BOP2}
for analogous probabilistic small data global results in 
the context of the quintic NLW on $\R^3$
and the cubic NLS on $\R^d$, $d \geq 3$, respectively.

\begin{remark}\label{rem1.3}
\rm
In proving scattering in Theorem \ref{scattering}, we 
exploit the finiteness of the 
global-in-time Strichartz $L^{\frac{d+2}{d-2}}_t\big(\R; L^{\frac{2(d+2)}{d-2}}_x\big)$-norm of the nonlinear part of a solution $v^\omega$.
This space-time norm is finite almost surely due to the
smallness of the initial data.
In the case of {\it large
data}, the best global space-time bound one could expect
(guided by the literature on the deterministic energy-critical NLW
\cite{Bahouri_Gerard, Bahouri_Shatah, Tao}) is 
\begin{equation}\label{stbound}
\|v\|_{L^{\frac{d+2}{d-2}}_t\Big(\R; L^{\frac{2(d+2)}{d-2}}_x(\R^d)\Big)}\leq C\Big(\|(v,\pa_t v)\|_{L_t^\infty(\R; \dot{H}_x^1(\R^d)\times L_x^2(\R^d))}\Big),
\end{equation}

\noi
where $C(\cdot)$ is a positive non-decreasing function.
For large data, however, we do not have a uniform in time bound on $\|(v(t),\pa_t v(t))\|_{\dot{H}^1\times L^2}$. 
More precisely, the probabilistic energy bound  in Proposition \ref{prop:energy} below
grows in time, diverging as $t\to\infty$.
Therefore, 
even if one could prove \eqref{stbound},
we would still not have sufficient information on $v$ for almost sure large data scattering.

\end{remark}

Before stating the main result of this paper,
we first recall the definition of the set $\mathcal{M}_s$ of measures on $H^s(\R^d)\times H^{s-1}(\R^d)$.
See \cite{BT3}.

\begin{definition}
{\rm
Let $0 \leq s\leq 1$ and $(u_0, u_1)\in H^s(\R^d)\times H^{s-1}(\R^d)$.
The map defined by
\[\omega \mapsto (u_0^\omega,u_1^\omega)\]
is a measurable map from $(\Omega, \mathcal{F}, P)$
into $H^s(\R^d) \times H^{s-1}(\R^d)$ endowed with the Borel $\sigma$-algebra,
since its partial sums 
\[\bigg(\sum_{|n|\leq N}g_{n,0} (\omega)\psi (D-n)u_0, \sum_{|n|\leq N}g_{n,1} (\omega)\psi (D-n)u_1\bigg)\]

\noi
form a Cauchy sequence in $L^2\left(\Omega; H^s(\R^d)\times H^{s-1}(\R^d)\right)$.
Therefore, we define the induced probability measure
\[\mu_{(u_0,u_1)} (A): =P\big((u_0^\omega,u_1^\omega)\in A\big),\]

\noi
 for all measurable $A\subset H^s(\R^d)\times H^{s-1}(\R^d)$.
We then set 
\[\mathcal{M}_s:= \bigcup_{(u_0,u_1)\in H^s\times H^{s-1}} \big\{\mu_{(u_0,u_1)}\big\}.\]
}
\end{definition}

For any $(u_0,u_1)\in H^s(\R^d)\times H^{s-1}(\R^d)$,
the measure $\mu_{(u_0,u_1)}$ is supported on $H^s(\R^d)\times H^{s-1}(\R^d)$
and moreover, $\mu_{(u_0,u_1)}(H^{s'}(\R^d)\times H^{s'-1}(\R^d))=0$ if $s'>s$
and $(u_0,u_1)\notin H^{s'}(\R^d)\times H^{s'-1}(\R^d)$. This is another way of 
saying that the randomization is not regularizing in terms of differentiability.

\medskip

We are now ready to state the main result of this paper.

\begin{theorem}[Almost sure global well-posedness]
\label{main}
Let $d=4$ or $5$, $0 < s\leq 1$ if $d=4$, and $0\leq s\leq 1$ if $d=5$.
There exists a set of full measure $\Sigma \subset H^s(\R^d)\times H^{s-1}(\R^d)$,
i.e. $\mu(\Sigma)=1$ for all $\mu\in\M_s$,  such that 
for any $(\phi_0,\phi_1)\in \Sigma$,
the energy-critical defocusing NLW \eqref{NLW} on $\R^d$ with initial data 
$\left(u,\pa_t u\right)\big|_{t=0}=(\phi_0, \phi_1)$
admits 
a unique global solution 
in the class
\begin{align*}
(u,\pa_tu)&\in \big(S(t)(\phi_0,\phi_1), \pa_t S(t)(\phi_0,\phi_1)\big)+C\left(\R; H^1(\R^d)\times L^2(\R^d)\right)\\
&\subset C\left(\R; H^{s}(\R^d)\times H^{s-1}(\R^d)\right).
\end{align*}

\noindent
Moreover, 
if $\Phi(t): (\phi_0,\phi_1)\mapsto u(t)$
denotes the solution map of equation \eqref{NLW}, 
then $\Phi(t)(\Sigma)$ is a set of full measure for all $t\in\R$.
\end{theorem}

The proof of Theorem \ref{main}
is based on a {\it probabilistic perturbation theory}. 
More precisely, we combine the global space-time bounds of solutions of the 
energy-critical defocusing NLW from \cite{Bahouri_Gerard, Tao}
with a perturbation lemma,
and 
design a
``good'' deterministic local well-posedness theory
(Proposition \ref{prop:main})
for a perturbed NLW
of the form: 
\begin{equation}\label{detv}
\begin{cases}
\pa_{t}^2 v-\Delta v+F(v+f)=0\\
(v,\pa_t v)\big|_{t=0}=(v_0,v_1)\in \dot{H}^1(\R^d)\times L^2(\R^d),
\end{cases}
\end{equation}  
where $f$ satisfies 
some smallness assumption on small time intervals.
Recall that the usual local well-posedness 
argument for the energy-critical NLW \eqref{NLW}
with initial data $(v_0, v_1)\in \dot{H}^1(\R^d)\times L^2(\R^d)$
yields a local time of existence depending on 
the profile of $(v_0, v_1)$.
The term ``good'' local well-posedness
refers to the fact that
the local time of existence depends only on 
the $\dot H^1\times L^2$-norm of the initial data $(v_0,v_1)$
and on the perturbation $f$.

Given randomized initial data $(u_0^\omega,u_1^\omega)\in H^s(\R^d)\times H^{s-1}(\R^d)$, with $0< s<1$ if $d=4$, and $0\leq s< 1$ if $d=5$, 
let $u^\omega$ be  the corresponding solution of equation \eqref{NLW}.
Also, we denote the linear and nonlinear parts of $u^\omega$
by 
$z^{\omega}(t)=S(t)(u_0^{\omega}, u_1^{\omega})$
and  $v^\omega(t):=u^\omega(t)-z^\omega(t)$
as before.
Then, 
a crucial ingredient of the proof of Theorem \ref{main}
is 
a probabilistic energy 
bound for $v^\omega$ (Proposition \ref{prop:energy}).
The probabilistic energy bound follows
from the improved local-in-time Strichartz estimates
(Proposition \ref{proba_S}) 
and (nonlinear) Gronwall's inequality. 
These Strichartz estimates are also the key 
in showing  that $z^\omega$ 
satisfies almost surely the smallness assumption in the ``good'' local well-posedness theory.
Finally, noting that $v^\omega$ satisfies \eqref{detv}
with $f=z^\omega$, 
the almost sure global existence of $v^\omega$  follows
by iterating the ``good'' local well-posedness.
We remark that the nonlinear part $v^\omega$
satisfies $(v^\omega(t),\pa_t v^\omega(t))\in (H^1(\R^d),L^2(\R^d))$ for all $t\in\R$,
and, in particular, it has improved regularity in comparison to the linear part $z^\omega$,
which is merely in $H^s(\R^d)$ with $0< s<1$ if $d=4$, and $0\leq s<1$ if $d=5$.
In \cite{BOP2}, the author with B\'enyi and Oh considered
the same problem for the energy-critical cubic NLS on $\R^4$.
In this case, they could only prove `conditional'
almost sure global well-posedness, 
assuming a probabilistic energy bound on the nonlinear part of a solution.
See Remark \ref{HypA} below.

Theorem \ref{main}
does not cover the case $s=0$ in dimension $d=4$. 
This is due to the use of the Sobolev embedding $W^{s,r}(\R^4)\subset L^\infty(\R^4)$
for $sr>4$, which requires $s>0$. See Proposition \ref{proba_S} (iii) and Proposition  \ref{prop:energy}. For $d=5$, one does not need to use such a Sobolev embedding,
and thus we can include the case $s=0$.

Notice that the set of full measure $\Sigma$ of initial data in Theorem \ref{main}
is constructed in such a way that $\Phi(t)(\Sigma)$ remains of full measure 
for all $t\in\R$, in other words, the measure does not become smaller
under the evolution of the flow of \eqref{NLW}. 
See \cite{BT3, OQ} for related results.

In the definitions of the uniform decomposition \eqref{Ziv2}
and of the Wiener randomization \eqref{R1}, 
we used a smooth cutoff function $\psi$.
Theorems \ref{ASLWP},  \ref{scattering}, and \ref{main} still hold
even when we replace $\psi$ by the characteristic  function
$\chi_{Q_0}$.

We present the proof of Theorem \ref{main} in Section 5.
In particular, 
Theorem \ref{main}
is a consequence of Theorem \ref{GWP}, Corollary \ref{GWPcor},
and Proposition \ref{invariant_Sigma}.
In addition, in Proposition \ref{almost_almost}, we prove space-time bounds for
the nonlinear part of the solution $v^\omega$ on any given finite time interval $[0,T]$.
Unfortunately, these bounds diverge as $T\to\infty$,
and hence the present paper does not provide
global space-time bounds for $v^\omega$. 
As a consequence, our arguments are not sufficient to obtain scattering 
of the global solution $u$ in Theorem \ref{main}
to a linear solution (which is standardly proved using global space-time bounds).
Therefore, almost sure large data scattering remains a challenging open problem
and new ideas seem to be needed to tackle it. (See also Remark \ref{rem1.3} above.)

\medskip
Our last result concerns the probabilistic continuous dependence
of the flow of equation \eqref{NLW} on $\R^4$.
As mentioned above,
Christ, Colliander, and Tao \cite{Christ_Colliander_Tao_main}
proved ill-posedness of 
\eqref{NLW} 
in $H^s(\R^4)\times H^{s-1}(\R^4)$ with $0<s<1$ in the deterministic theory,
by showing that
the solution map is not continuous at zero.
In the following, we show
that the solution map is, however,  continuous {\it in probability}
 in $H^s(\R^4)\times H^{s-1}(\R^4)$ with $0< s<1$.
The notion of probabilistic continuous dependence of the flow with respect to the the initial data
used here was first introduced by Burq and Tzevtkov in \cite{BT3}.

\begin{theorem}[Probabilistic continuous dependence]
\label{proba_cont}
Let  $d=4$ and $0 < s\leq 1$. 
Assume that in the definition of the randomization
\eqref{R1}, the smooth cutoff $\psi$
is replaced
by the characteristic function $\chi_{Q_0}$ of the unit cube centered at the origin
$Q_0=\left[-\frac 12, \frac 12\right)^4$.
Assume also  that the probability distributions $\mu_{0,j}$, $\mu_{n,j}^{(1)}$, $\mu_{n,j}^{(2)}$,
$n\in\mathcal I$, $j=0,1$ are symmetric\footnote{A probability measure $\theta$ on $\R$ is called symmetric if $\int_\R f(x)d \theta (x)=\int_\R f(-x)d\theta (x)$ for all $f\in L^1(d\theta)$.}. 

Let $T>0$, $\mu\in\M_s$, $R>0$, and 
\[B_R:=\left\{(w_0,w_1)\in H^{s}(\R^4)\times H^{s-1}(\R^4) : \|(w_0,w_1)\|_{H^{s}(\R^4)\times H^{s-1}(\R^4)}\leq R\right\}\]
be the closed ball of radius $R$ centered at the origin
in $H^{s}(\R^4)\times H^{s-1}(\R^4)$.

If $\Phi(t): (\phi_0,\phi_1)\mapsto u(t)$
is the solution map of the energy-critical defocusing cubic NLW \eqref{NLW} on $\R^4$,
defined $\mu$-almost everywhere in Theorem \ref{main},
then for any $\delta,\eta>0$ it follows that
\begin{align}\label{prob_cont_eq}
\mu \otimes &\mu \bigg(\left((w_0, w_1), (w_0',w_1')\right)\in \left(H^{s}(\R^4)\times H^{s-1}(\R^4)\right)^2:\\
&
\left\|\left(\Phi(t)(w_0,w_1)-\Phi(t)(w_0',w_1'), \pa_t \left(\Phi(t)(w_0,w_1)-\Phi(t)(w_0',w_1')\right)\right)\right\|_{L^{\infty}([0,T]; H^{s}\times H^{s-1})}>\delta \notag\\
&\Big| \quad (w_0,w_1), (w_0',w_1')\in B_R \text{ and } \|(w_0,w_1)-(w_0',w_1')\|_{H^{s}(\R^4)\times H^{s-1}(\R^4)}<\eta \bigg)\notag\\
&\leq g(\delta,\eta),\notag
\end{align}
where $g(\delta,\eta)$ satisfies
\[\lim_{\eta\to 0} g(\delta,\eta)=0,\]

\noi
for all $\dl > 0$.
\end{theorem}

In stating Theorem \ref{proba_cont}, 
we assumed that
the randomization
\eqref{R1}
was implemented with the characteristic function
 $\chi_{Q_0}$, instead of the smooth cutoff function $\psi$.
In particular, the support of $\chi_{Q_0}(\cdot -n)$ and that of $\chi_{Q_0}(\cdot -m)$,
namely the cubes $n+Q_0$ and $m+Q_0$, are disjoint
for $n,m\in\Z^4$, $n\neq m$.
This plays an important role in the proof.
We believe that a statement analogous to Theorem \ref{proba_cont} is true
for a general smooth cutoff function $\psi$.
Unfortunately, we do not know how to prove such a claim at this point.

One of the several key points in the proof of Theorem \ref{proba_cont}
is the fact that the power of the nonlinearity needs to be larger or equal to three 
(see Step 3 of the proof, where we used the Strichartz norm $L^2_tL^8_x(\R^4)$).
Therefore, our method does not apply to the energy-critical defocusing NLW on $\R^5$,
for which the power of the nonlinearity is $\frac 73$.

In order to prove Theorem \ref{proba_cont}, 
we first control the linear parts of the solutions $\Phi(t)(w_0,w_1)$ and $\Phi(t)(w_0',w_1')$, as well as their difference. 
The key element here is represented by the improved local-in-time Strichartz estimates,
while the context in which they are used is analogous to that in \cite{BT3}.
The novel element of the proof
is the control of the Stricharz norms of the nonlinear parts 
of the solutions, which allows us
to control the difference of these nonlinear parts.
In \cite{BT3}, in the case of the energy-subcritical cubic NLW on $\T^3$,
a mere probabilistic energy bound was sufficient to control the 
difference of the nonlinear parts.

\begin{remark}\label{HypA}
\rm
In the spirit of \cite{BOP2},
we can also prove `conditional'
almost sure global well-posedness
for the energy-critical defocusing quintic NLW on $\R^3$, 
provided we assume the following Hypothesis\footnote{Very recently,
the author with Tadahiro Oh
proved in \cite{ OPNLW3D} the almost sure
global well-posedness of the energy-critical defocusing
quintic nonlinear wave equation on $\R^3$
below the energy space.
The main new ingredient in the proof
is a uniform probabilistic energy bound for 
approximating random solutions.
In particular, we proved almost sure global well-posedness
without directly establishing
the Hypothesis 
above.
Nevertheless, the Hypothesis follows as a byproduct of
the construction in \cite{OPNLW3D}.}, i.e.~
a probabilistic energy bound on the nonlinear part  $v^\omega$
of the solution:

\smallskip
\noindent
{\bf Hypothesis:}
 Given $T,\eps>0$, there exists $C(T,\eps)$ non-decreasing in $T$
 and non-increasing in $\eps$ and there exists $\tilde \Omega_{T,\eps}$
 with $P(\tilde\Omega_{T,\eps}^c)<\eps$ such that, for all $\omega\in\tilde \Omega_{T,\eps}$,
 the solution $v^\omega$ of \eqref{v} satisfies
 \[\|v^\omega(t)\|_{L^\infty([0,T],\dot{H}^1(\R^3)\times L^2(\R^3))}\leq C(T,\eps).\]

\noi
Furthermore, under this Hypothesis, 
one can also prove probabilistic continuous dependence 
of the flow of the energy-critical defocusing quintic NLW on $\R^3$.
The proofs of these results
follow the same lines as those of Theorem \ref{main} and Theorem \ref{proba_cont}.\end{remark}


\subsection{Probabilistic well-posedness results regarding NLW}
 To the best of the author's knowledge,
Theorem \ref{main} is the first result
of almost sure global well-posedness
for an {\it energy-critical} hyperbolic/dispersive PDE
with large data below the energy space. 
In the following, we briefly mention 
some of the references in the literature regarding almost sure global well-posedness
of NLW. All the results below concern defocusing NLW and we 
do not explicitly mention this in the following.

In what concerns 
NLW on $\T^d$, Burq and Tzvetkov \cite{BT3} proved probabilistic global well-posedness 
for (energy-subcritical) cubic NLW on $\T^3$ for data 
in $H^s(\T^3)\times H^{s-1}(\T^3)$, $0\leq s<1$, while Burq, Thomann, and Tzvetkov \cite{BTT2} considered
(energy-critical and supercritical) cubic NLW on $\T^d$, $d\geq 4$, and proved almost sure global existence, without uniqueness, for data in  $H^s(\T^d)\times H^{s-1}(\T^d)$, $0<s<1$.

Burq and Tzvetkov \cite{BTIMRN} 
also considered subcubic NLW 
on the unit ball in $\R^3$
with the power nonlinearity $|u|^{p-1} u$, $1<p<3$,
and proved almost sure global well-posedness of NLW
for a large set of radially symmetric data in $\bigcap_{s<\frac 12}H^s$.
More precisely, they constructed the Gibbs measure for NLW
on the unit ball and the global-in-time flow on the support of the Gibbs measure,
proving also the invariance of the Gibbs measure under the flow. 
The almost sure global well-posedness part
 was extended by the same authors for $1<p<4$ in \cite{BTII},
and the invariance of the measure for $p=3$ was proved by de Suzzoni in \cite{Suzzoni2011}.
Finally, Bourgain and Bulut extended the above results for all $1<p<5$. 
Notice that for $3\leq p<5$, the above almost sure global existence results are below the  
scaling critical regularity.

Let us now turn our attention to known results on $\R^d$.
De Suzzoni \cite{Suzzoni2013} considered
subquartic NLW on $\R^3$, namely  $3\leq p <4$,
and proved global existence, uniqueness, and scattering
for a set of full measure of radially symmetric data of low regularity, 
that do not belong to $L^2 (\R^3)$.
In \cite{Suzzoni2012}, de Suzzoni proved 
an almost sure global existence, uniqueness, and scattering result
for cubic NLW on $\R^3$, without the radial symmetry assumption. 
The Penrose transform played an essential role
in both articles.
Recently, L\"uhrmann and Mendelson \cite{LM} 
proved almost sure global well-posedness
of 
energy-subcritical NLW on $\R^3$ with a power nonlinearity
$|u|^{p-1} u$, $3\leq p <5$, 
with 
random initial data in $H^s(\R^3) \times H^{s-1}(\R^3)$, for some $s<1$.
For $p \in \big(\frac{7+\sqrt{73}}{4},5\big)$, 
the regularity of initial data can be taken below the critical regularity dictated by the scaling
invariance. 
For the energy-critical NLW on $\R^3$ $(p=5)$,
they obtained small data almost sure global well-posedness and scattering.

Finally, 
there are other classes of almost sure global well-posedness results on $\R^d$
and other unbounded domains.
They involve the construction of invariant Gibbs measures on such domains.
In particular, on $\R^d$, the typical functions in the support of Gibbs measures
do not decay at spatial infinity and thus do not belong to the Lebesgue spaces $L^p(\R^d)$ or Sobolev spaces $H^s(\R^d)$.
See, for example, the work of McKean and Vaninsky \cite{MV} on $\R$
and the recent work of Xu \cite{Xu} concerning cubic NLW on $\R^3$ with radial symmetry.


\subsection{Notations}

If $u$ satisfies the wave equation
\[\pa_{t}^2u-\Delta u+F(u)=0,\]
on the interval $I$ containing $t_0$ and $t$, 
then the following Duhamel's formula holds:
\begin{equation}
\label{Duhamel}
u(t)=S(t-t_0)(u(t_0), \pa_tu(t_0))
-\int_{t_0}^t \frac{\sin ((t-t')|\nabla|)}{|\nabla|}F(u(t'))dt'.
\end{equation}

If $F(u)=|u|^{\frac{4}{d-2}}u$, we use the fundamental theorem of calculus to write
\[F(u)-F(\tilde{u})=(u-\tilde u)\int_0^1 F'\left(\lambda u+(1-\lambda)\tilde u\right)d\lambda,\]
and we immediately deduce that
\begin{equation}\label{Fbasic}
|F(u)-F(\tilde u)|\lesssim |u-\tilde u|\left(|u|^{\frac{4}{d-2}}+|\tilde u|^{\frac{4}{d-2}}\right).
\end{equation}

\begin{definition}
Let $\gamma\in\R$ and $d\geq 2$. We say that $(q,r)$ is a $\dot{H}^\gamma(\R^d)$-wave admissible pair
if $q\geq 2$, $2\leq r<\infty$,
\[\frac{1}{q}+\frac{d-1}{2r}\leq \frac{d-1}{4}\]
and
\[\frac 1q+\frac dr=\frac d2-\gamma.\]
\end{definition}

We recall some Strichartz estimates for wave equations on $\R^d$. 
For more details as well as other Strichartz estimates, see \cite{Ginibre, Lindblad, Keel}.
\begin{proposition}[Strichartz estimates on $\R^d$]
\label{prop:Strichartz}
Let $d\geq 2$, $\gamma>0$, let $(q,r)$ be a $\dot{H}^\gamma(\R^d)$-wave admissible pair, 
and $(\tilde{q},\tilde{r})$ be a $\dot{H}^{1-\gamma}(\R^d)$-wave admissible pair. 
If $u$ solves
\[\pa_{t}^2 u-\Delta u+F=0  \quad \text{ with } \quad  u(0)=u_0, \quad \pa_t u(0)=u_1\]
on $I\times\R^d$, where $0\in I$, then
\begin{align}\label{Strichartz}
\|u\|_{L^\infty_t(I;\dot{H}^\gamma_x)}+\|\pa_t u\|_{L^\infty_t(I; \dot{H}^{\gamma-1}_x)}+\|u\|_{L^q_t(I; L^r_x)}\lesssim
\|u_0\|_{\dot{H}^{\gamma}}+\|u_1\|_{\dot{H}^{\gamma-1}}+\|F\|_{L^{\tilde{q}'}_t(I; L^{\tilde{r}'}_x)}.
\end{align}
\end{proposition}

\noi
For convenience, in the following, we often denote the space $L^q_t(I; L^r_x)$  by $L^q_IL^r_x$,
or by $L^q_TL^r_x$ if $I=[0,T]$.

The wave admissible pairs for $d=4$ and $5$ that appear most often in this paper are:
\begin{itemize}

\item $\dot{H}^1$-wave admissible: $\big(\frac{d+2}{d-2}, \frac{2(d+2)}{d-2}\big)$ if $d=4$ or $5$, and $(2,8)$ if $d=4$

\item $\dot{H}^{0}$-wave admissible: $(\infty,2)$.

\end{itemize}

\noindent
In particular, the Strichartz space $L^{\frac{d+2}{d-2}}_t\big(I;L^{\frac{2(d+2)}{d-2}}_x(\R^d)\big)$ will appear very often in our analysis, and therefore we fix the notation:
\begin{align*}
\|f\|_{X(I\times\R^d)}:&=\|f\|_{L^{\frac{d+2}{d-2}}_t\big(I;L^{\frac{2(d+2)}{d-2}}_x(\R^d)\big)}, 
\end{align*}
where $I\subset\R$ denotes a time interval.

\medskip

This paper is organized as follows. 
In Section 2, we prove probabilistic Strichartz estimates.
Section 3 is dedicated to the proof of Theorems \ref{ASLWP}
and \ref{scattering}. In Section 4, we design
a ``good'' deterministic local well-posedness theory
using a perturbation lemma.
Section 5 contains the
proof of Theorem \ref{main} (reformulated as Theorem \ref{GWP}, Corollary \ref{GWPcor}, and Proposition  \ref{invariant_Sigma}).
Finally, in Section 6 we prove Theorem \ref{proba_cont}.


\section{Probabilistic estimates}

In this section we recall some basic properties of randomized functions
and then present some improved Strichartz estimates under
randomization.

First recall the following probabilistic estimate.
\begin{lemma}\label{basic_lemma}
Let $\{g_n\}_{n\in \Z^d}$ be a sequence of complex-valued,
zero mean, random variables
such that $g_{-n}=\overline{g_n}$ for all 
$n\in\Z^d$.
Assume that $g_0$, ${\rm Re}\, g_n$, ${\rm Im}\, g_n$, 
$n\in \mathcal I:=\bigcup_{k=0}^{d-1} \Z^k \times \Z_+^\ast\times \{0\}^{d-k-1}$,
are independent and have associated distributions
$\mu_0$, $\mu_n^{(1)}$, respectively $\mu_n^{(2)}$.
Assume that there exists $c>0$ such that
\begin{equation}\label{distrib}
\int_{\R}e^{\gamma x}d\mu_0\leq e^{c\gamma^2} 
\text{  and }\int_{\R}e^{\gamma x}d\mu_n^{(k)}\leq e^{c\gamma^2}
\end{equation}
for all $\gamma\in\R$, $n\in\Z^d$, $k=1,2$.

Then, there exists $C>0$ such that the following holds:
\begin{equation}\label{proba}
\Big\|\sum_{n\in\Z^d}g_n(\omega)c_n\Big\|_{L^p(\Omega)}\leq C\sqrt{p} \|c_n\|_{\ell^2_n(\Z^d)}
\end{equation}
for any $p\geq 2$
and any sequence $\{c_n\}_{n\in\Z^d}\in \ell^2 (\mathbb{Z}^d)$ satisfying $c_{-n}=\overline{c_n}$
for all $n\in\Z^d$.
\end{lemma}

\begin{proof}
As in \eqref{real_randomization}, we notice that
\[\sum_{n\in\Z^d}g_n c_n=g_0c_0+2{\rm Re}\, \sum_{n\in \mathcal I} g_n c_n
=g_0c_0+2\sum_{n\in \mathcal I} \left({\rm Re}\, g_n {\rm Re}\, c_n - {\rm Im}\, g_n {\rm Im}\, c_n\right).\]
Since  $g_0={\rm Re}\, g_0$, $\Re  g_n$, $\Im  g_n$, with $n\in \mathcal I$,
are mean zero, independent real random variables and their distributions satisfy \eqref{distrib}, it is sufficient to apply Lemma 3.1 from \cite{BTI}
to obtain the conclusion. 
\end{proof}


Secondly, recall that if $\phi\in H^s$,
then the randomization $\phi^\omega$
is in $H^s$ almost surely as long as \eqref{distrib} is satisfied.
More precisely, we have the following lemma, 
whose proof is based on Lemma \ref{basic_lemma}.
See \cite{BOP1} for the details.
\begin{lemma}\label{lemma:Hs}
Let $\{g_n\}_{n\in\Z^d}$
be a sequence of random variables satisfying the hypotheses of Lemma 
\ref{basic_lemma}.
Let
$\phi\in H^s(\R^d)$ be a real-valued function and let $\phi^\omega$ be its real-valued randomization
defined by
\[\phi^\omega:=\sum_{n\in\Z^d}g_n(\omega)\psi (D-n) \phi.\]

\noi
Then, we have that
\begin{align*}
P\big(\|\phi^\omega\|_{H^s_x(\R^d)}>\l\big)\leq C\exp \bigg(-\frac{c\l^2}{\|\phi\|_{H^s}^2}\bigg).
\end{align*}
\end{lemma}


Before continuing further, we briefly recall the definitions of the smooth projections from the Littlewood-Paley theory.
Let $\varphi$ be a smooth real-valued bump function supported on $\{\xi\in \R^d: |\xi|\leq 2\}$ and $\varphi\equiv 1$ on $\{\xi: |\xi|\leq 1\}$. 
If $N>1$ is a dyadic number, we define the smooth projection 
$\pP_{\leq N}$ onto frequencies
$\{|\xi| \leq N\}$ by
\[\widehat{\pP_{\leq N}f}(\xi):=\varphi\big(\tfrac{\xi}N\big)\widehat f(\xi).\]

\noindent
Similarly, we can define the smooth projection $\pP_{N}$ onto frequencies $\{|\xi|\sim  N\}$ by
\[\widehat{\pP_{N}f}(\xi):=\Big(\varphi\big(\tfrac{\xi}N\big)-\varphi\big(\tfrac{2\xi}N\big)\Big)\widehat f(\xi).\]
We make the convention that 
$\pP_{\leq 1}=\pP_1$. 
Bernstein's inequality states that
\begin{equation}\label{R3}
\|\pP_{\leq N} f\|_{L^q(\R^d)} \lesssim N^{\frac{d}{p}-\frac{d}{q}}\|\pP_{\leq N} f\|_{L^p(\R^d)}, \quad
1\leq p \leq q \leq \infty.
\end{equation}

\noindent
The same inequality holds if we replace 
$\pP_{\leq N}$ by $\pP_N$. As an immediate corollary, we have
\begin{equation}\label{R4}
\|\psi(D -n) \phi\|_{L^q(\R^d)} \lesssim \|\psi(D-n)  \phi \|_{L^p(\R^d)}, \qquad
1\leq p \leq q \leq \infty,
\end{equation}

\noindent
for all $n \in \Z^d$.
This follows by applying \eqref{R3}
to $\phi_n(x) := e^{2\pi i n\cdot x} \psi(D-n) \phi(x)$
and noting that $\supp \widehat{\phi}_n \subset [-1, 1]^d$ .
The point of \eqref{R4} is that once a function is (roughly) restricted to a cube
in the Fourier space, we do not lose any regularity to go from the $L^q$-norm  to the $L^p$-norm, $q \geq p$.


\begin{proposition}[Improved local-in-time Strichartz estimates]\label{proba_S}
Let $d\geq 1$, $u_0$, $u_1$ two real-valued functions defined on $\R^d$,
and let $(u_0^{\omega}, u_1^\omega)$ be their randomization defined in \eqref{R1}, satisfying \eqref{cond}. Let $I=[a,b]\subset \R$ be a compact time interval.

\noindent
{\rm(i)} If $u_0\in L^2(\R^d)$ and $u_1\in\dot{H}^{-1}(\R^d)$,
then given $1\leq q<\infty$ and $2\leq r<\infty$, there exist $C,c>0$ such that
\begin{align*}
P\left(\left\|S(t)(u_0^\omega,u_1^\omega)\right\|_{L^q_t(I; L^r_x)}>\l\right)\leq C\exp\left(-c\frac{\l^2}{|I|^{\frac2q} \left(\|u_0\|_{L^2}+\|u_1\|_{\dot{H}^{-1}}\right)^2}\right).
\end{align*}

\noindent
{\rm (ii)} If $u_0\in L^2(\R^d)$ and $u_1\in H^{-1}(\R^d)$,
 then given $1\leq q< \infty$, $2\leq r<\infty$,
there exist $C,c>0$ such that
\begin{align*}
P\left(\left\|S(t)(u_0^\omega,u_1^\omega)\right\|_{L^q_t(I; L^r_x)}>\l\right)\leq C\exp\left(-c\frac{\l^2}{\max\left(1, |a|^2, |b|^2\right)|I|^{\frac2q}\left(\|u_0\|_{L^2}+\|u_1\|_{H^{-1}}\right)^2}\right).
\end{align*}

\noindent
{\rm (iii)}
If $u_0\in H^s(\R^d)$ and $u_1\in H^{s-1}(\R^d)$ for some $0<s\leq 1$,
then given $1\leq q<\infty$, there exist $C,c>0$ such that
\begin{align*}
P\left(\left\|S(t)(u_0^\omega,u_1^\omega)\right\|_{L^q_t(I; L^\infty_x)}>\l\right)\leq C\exp\left(-c\frac{\l^2}{\max\left(1, |a|^2, |b|^2\right)|I|^{\frac 2q} (\|u_0\|_{H^s}+\|u_1\|_{H^{s-1}})^2}\right).
\end{align*}
\end{proposition}


\begin{proof}
(i). Let $1\leq q<\infty$, $2\leq r<\infty$, and $p\geq\max (q,r)$. Then, 
using Minkowski's integral inequality, \eqref{proba}, and Bernstein's inequality \eqref{R4}, we have that
\begin{align*}
\Big(\E \big\|S(t)(&u_0^\omega,u_1^\omega)\big\|^p_{L^q_t(I; L^r_x)}\Big)^{1/p}
\leq \left\|\left\|S(t)(u_0^\omega,u_1^\omega)\right\|_{L^p(\Omega)}\right\|_{L^q_IL^r_x}\\
&\lesssim \sqrt{p}  \left\|\left\|\psi (D-n)\cos (t|\nabla|)u_0\right\|_{\ell^2_n}\right\|_{L^q_IL^r_x}
+\sqrt{p}  \left\|\left\|\psi (D-n)\frac{\sin (t|\nabla|)}{|\nabla|}u_1\right\|_{\ell^2_n}\right\|_{L^q_IL^r_x}\\
&\lesssim \sqrt{p}  \left\|\left\|\psi (D-n)\cos (t|\nabla|)u_0\right\|_{L^r_x}\right\|_{L^q_I\ell^2_n}
+\sqrt{p}  \left\|\left\|\psi (D-n)\frac{\sin (t|\nabla|)}{|\nabla |}u_1\right\|_{L^{r}_x}\right\|_{L^q_I\ell^2_n}\\
&\lesssim \sqrt{p}  \left\|\left\|\psi (D-n)\cos (t|\nabla|)u_0\right\|_{L^2_x}\right\|_{L^q_I\ell^2_n}
+\sqrt{p}  \left\|\left\|\psi (D-n)\frac{\sin (t|\nabla|)}{|\nabla |}u_1\right\|_{L^2_x}\right\|_{L^q_I\ell^2_n}\\
&\lesssim \sqrt{p}  \left\|\left\|\psi (D-n)u_0\right\|_{L^2_x}\right\|_{L^q_I\ell^2_n}
+\sqrt{p}  \left\|\left\|\psi (D-n)|\nabla|^{-1}u_1\right\|_{L^2_x}\right\|_{L^q_I\ell^2_n}\\
&\lesssim \sqrt{p}|I|^{\frac 1q}\left(\|u_0\|_{L^2}+\|u_1\|_{\dot{H}^{-1}}\right).
\end{align*}

\noindent
Then, by Chebyshev's inequality we have that
\[P\left( \left\|S(t)(u_0^\omega,u_1^\omega)\right\|_{L^q_t(I; L^r_x)}>\l\right)
<\left(\frac{C|I|^{\frac 1q}p^{\frac 12}\left(\|u_0\|_{L^2}+\|u_1\|_{\dot{H}^{-1}}\right)}{\l}\right)^p\]
for $p\geq \max (q,r)$.

Let $p:=\left(\frac{\l}{C|I|^\frac 1q e\left(\|u_0\|_{L^2}+\|u_1\|_{\dot{H}^{-1}}\right)}\right)^2$.
If $p\geq \max (q,r)$, we have
\begin{align*}
P\left( \left\|S(t)(u_0^\omega,u_1^\omega)\right\|_{L^q_t(I; L^r_x)}>\l\right)
&<\left(\frac{C|I|^{\frac 1q}p^{\frac 12}\left(\|u_0\|_{L^2}+\|u_1\|_{\dot{H}^{-1}}\right)}{\l}\right)^p\\
&=e^{-p}=\exp\left(-c\frac{\l^2}{|I|^{\frac2q} \left(\|u_0\|_{L^2}+\|u_1\|_{\dot{H}^{-1}}\right)^2}\right).
\end{align*}
Otherwise, if $p=\left(\frac{\l}{C|I|^\frac 1q e\left(\|u_0\|_{L^2}+\|u_1\|_{\dot{H}^{-1}}\right)}\right)^2\leq \max (q,r)$, we choose $C$ such that $C e^{-\max (q,r)}\geq 1$.
We then have
\begin{align*}
P\left(\left\|S(t)\left(u_0^\omega,u_1^\omega\right)\right\|_{L^q_t(I; L^r_x)}>\l\right)
&\leq 1\leq Ce^{-\max (q,r)}\leq Ce^{-p}\\
&=C\exp\left(-c\frac{\l^2}{|I|^{\frac2q} \left(\|u_0\|_{L^2}+\|u_1\|_{\dot{H}^{-1}}\right)^2}\right).
\end{align*}


\noindent
(ii). We pay a particular attention to low frequencies since, for $u_1\in H^{-1}(\R^d)$ and $n=0$, $\psi(D-n)|\nabla|^{-1}u_1$ is not in $L^2(\R^d)$. 
For $|n|\geq 1$, we argue as in part (i).
Using  Minkowski's integral inequality, \eqref{proba}, Bernstein's inequality \eqref{R4}, and the fact that $\left|\frac{\sin (t|\xi|)}{t|\xi|}\right|\leq 1$ for all $\xi\neq 0$, we obtain
for $p\geq \max (q,r)$ that
\begin{align*}
\bigg(\E \bigg\|\frac{\sin (t|\nabla|)}{|\nabla|}u_1^\omega&\bigg\|^p_{L^q_t(I; L^r_x)} \bigg) ^{1/p}
\leq\left\| \left\|\frac{\sin (t|\nabla|)}{|\nabla|}u_1^\omega\right\|_{L^p(\Omega)}\right\|_{L^q_IL^r_x}\\
&\lesssim \sqrt{p} \left\|\left\|\psi(D-n)\frac{\sin (t|\nabla|)}{|\nabla|}u_1\right\|_{\ell^2_n}\right\|_{L^q_IL^r_x}\\
&\lesssim \sqrt{p} \left\|\left\|\psi(D-n)\frac{\sin (t|\nabla|)}{|\nabla|}u_1\right\|_{L^r_x}\right\|_{L^q_I\ell^2_n}\\
&\lesssim \sqrt{p} \left\|\left\|\psi(D-n)\frac{\sin (t|\nabla|)}{|\nabla|}u_1\right\|_{L^2_x}\right\|_{L^q_I\ell^2_n}\\
&\lesssim \sqrt{p}|I|^{\frac 1q}\left(\sum_{|n|\geq 1}\left\|\psi(D-n)|\nabla|^{-1}u_1\right\|_{L^2_x}^2+ \sup_{t\in I}\left\|\psi (D)\frac{\sin (t|\nabla|)}{t|\nabla|}t u_1\right\|_{L^{2}_x}^2\right)^{\frac 12}\\
&\lesssim \sqrt{p}|I|^{\frac 1q}\left(\sum_{|n|\geq 1}\left\|\psi(D-n)|\nabla|^{-1}u_1\right\|_{L^2_x}^2+ \max\left(|a|^2, |b|^2\right)
\left\|\psi (D)u_1\right\|_{L^{2}_x}^2\right)^{\frac 12}\\
&\lesssim \sqrt{p}\max\left(1, |a|, |b|\right) |I|^{\frac 1q}\|u_1\|_{H^{-1}}.
\end{align*}
Arguing as in part (i)
for $\cos (t|\nabla|)u_0^\omega$,
the conclusion of (ii) then follows.


\noindent
(iii). Take $r\gg 1$ such that $s r>d$. Then, $W^{s, r}(\R^d)\subset L^{\infty}(\R^d)$ and thus
\[\left\|S(t)\left(u_0^\omega,u_1^\omega\right)\right\|_{L^q_t(I; L^\infty_x)}\lesssim \left\|S(t)\left(\jb{\nabla}^s u_0^\omega,\jb{\nabla}^s u_1^\omega\right)\right\|_{L^q_t(I; L^r_x)}.\]
Applying (ii) with $(q,r)$, the conclusion of (iii) follows.
\end{proof}


\begin{corollary}\label{cor:proba}
Let $d\geq 1$, $u_0\in L^2(\R^d)$, $u_1\in H^{-1}(\R^d)$, and $(u_0^{\omega}, u_1^\omega)$ their randomization defined in \eqref{R1}, satisfying \eqref{cond}. Then, given $1\leq q<\infty$, $2\leq r<\infty$, $0<\gamma<\frac1q$, and $I\subset [a,b]$
a compact time interval, the following holds
\begin{align}\label{improved_Strichartz}
\left\|S(t)\left(u_0^\omega,u_1^\omega\right)\right\|_{L^q_t(I; L^r_x)}\leq |I|^\gamma \left(\|u_0\|_{L^2}+\|u_1\|_{H^{-1}}\right)
\end{align}
outside a set of probability at most $C\exp\left(-\frac{c|I|^{2(\gamma-\frac1q)}}{\max(1,|a|^2, |b|^2)}\right)$.
\end{corollary}


\begin{proof}
The conclusion is obtained by taking $\l=|I|^\gamma \left(\|u_0\|_{L^2}+\|u_1\|_{H^{-1}}\right)$ in Proposition \ref{proba_S} (ii).
\end{proof}


We conclude this section with some improved
global-in-time Strichartz estimates.

\begin{proposition}[Improved global-in-time Strichartz estimates]
\label{globalS}
Let $d=4$ or $5$ and $\frac{(d+1)(d-2)}{(d-1)(d+2)}\leq s\leq 1$. 
Let $u_0\in \dot{H}^s(\R^d)$ and $u_1\in \dot{H}^{s-1}(\R^d)$,
and let $(u_0^{\omega}, u_1^\omega)$ be their randomization defined in \eqref{R1}, satisfying \eqref{cond}.

Then, given $\frac{2(d+2)}{d-2}\leq r<\infty$, there exist $C,c>0$ such that 
\begin{align*}
P\Bigg(\big\|S(t)(u_0^\omega,u_1^\omega)\big\|_{L^{\frac{d+2}{d-2}}_t\left(\R:L^r_x(\R^d)\right)}>\l\Bigg)
\leq C\exp\Bigg(-c\frac{\l^2}{(\|u_0\|_{\dot{H}^s}+\|u_1\|_{\dot{H}^{s-1}})^2}\Bigg).
\end{align*}

\end{proposition}


\begin{proof}
As in the proof of Proposition \ref{proba_S} (i),
we have for $p\geq r$ that
\begin{align*}
\Bigg(\E \big\|S(t)(&u_0^\omega,u_1^\omega)\big\|^p_{L^{\frac{d+2}{d-2}}_t\left(\R;L^r_x(\R^d)\right)}\Bigg)^{1/p}
\leq \Big\|\big\|S(t)(u_0^\omega,u_1^\omega)\big\|_{L^p(\Omega)}\Big\|_{L^{\frac{d+2}{d-2}}_tL^r_x}\\
&\lesssim \sqrt{p} \Big\|\|\psi(D-n)\cos (t|\nabla|)u_0\|_{\ell^2_n}\Big\|_{L^{\frac{d+2}{d-2}}_tL^r_x}
+\Big\|\Big\|\psi(D-n)\frac{\sin (t|\nabla|)}{|\nabla|}u_1\big\|_{\ell^2_n}\Big\|_{L^{\frac{d+2}{d-2}}_tL^r_x} \\
&\lesssim \sqrt{p} \Big\|\|\psi(D-n)\cos (t|\nabla|)u_0\|_{L^{\frac{d+2}{d-2}}_tL^r_x}\Big\|_{\ell^2_n}
+\sqrt{p} \Big\|\Big\|\psi(D-n)\frac{\sin (t|\nabla|)}{|\nabla|}u_1\Big\|_{L^{\frac{d+2}{d-2}}_tL^r_x}\Big\|_{\ell^2_n}.
\end{align*}
Then, we choose $\tilde{r}\leq \frac{2(d+2)}{d-2}$
such that $(\frac{d+2}{d-2},\tilde{r})$ is $\dot{H}^s(\R^d)$-wave admissible.
Note that the condition $\frac{d-2}{d+2}+\frac{d}{\tilde{r}}=\frac d2-s$ together with 
$\frac{(d+1)(d-2)}{(d-1)(d+2)}\leq s\leq 1$
yields that 
\[\frac{2(d-1)(d+2)}{d^2-3d+6}\leq\tilde r \leq \frac{2(d+2)}{d-2}.\]
At its turn, this shows that 
$\frac{d-2}{d+2}+\frac{d-1}{2\tilde{r}}\leq \frac{d-1}{4}$, and thus it is indeed possible to choose $\big(\frac{d+2}{d-2},\tilde r\big)$, with $\tilde r\leq \frac{2(d+2)}{d-2}$, to be $\dot{H}^s(\R^d)$-wave admissible.
Then, by Bernstein's inequality \eqref{R4} and Strichartz estimates 
\eqref{Strichartz},
we have that
\begin{align*}
\bigg(\E \big\|S(t)(&u_0^\omega,u_1^\omega)\big\|^p_{L^{\frac{d+2}{d-2}}_t\left(\R; L^r_x(\R^d)\right)}\bigg)^{1/p}
\lesssim \sqrt{p} \left\|\left\|\psi(D-n)\cos(t|\nabla|)u_0\right\|_{L^{\frac{d+2}{d-2}}_tL^{\tilde{r}}_x}\right\|_{\ell^2_n}\\
&+ \left\|\left\|\psi(D-n)\frac{\sin(t|\nabla|)}{|\nabla|}u_1\right\|_{L^{\frac{d+2}{d-2}}_tL^{\tilde{r}}_x}\right\|_{\ell^2_n}\\
&\lesssim \sqrt{p} \left\|\left\|\psi(D-n)u_0\right\|_{\dot{H}^s}\right\|_{\ell^2_n}+\left\|\left\|\psi(D-n)u_1\right\|_{\dot{H}^{s-1}}\right\|_{\ell^2_n}\\
&\lesssim \sqrt{p}\big( \|u_0\|_{\dot{H}^s}+\|u_1\|_{\dot{H}^{s-1}}\big).
\end{align*}
The conclusion then follows
as in the proof of Proposition \ref{proba_S} (i).
\end{proof}


\section{Almost sure local well-posedness.
Probabilistic small data global theory}

In this section we prove Theorem \ref{ASLWP}
concerning the local well-posedness of the energy-critical defocusing NLW \eqref{NLW} on $\R^d$, $d=4$ and $5$,
and Theorem \ref{scattering} concerning the small data global theory.

\begin{proof}[Proof of Theorem \ref{ASLWP}]
Since the linear part of the solution $z^\omega:=S(t)(u_0^\omega, u_1^\omega)$
is well defined for all times,
it suffices to prove almost sure 
local existence and uniqueness
for the equation \eqref{v} satisfied by the nonlinear part
$v^\omega:=u^\omega-z^\omega$. 
Also, by the time reversibility of \eqref{v},
it is sufficient to work with $t\geq 0$.

Let $0<T\leq 1$ to be fixed later and $0<\theta<\frac{d-2}{d+2}$. By the improved local-in-time Strichartz estimates 
\eqref{improved_Strichartz}, there exists a set $\Omega_T$
such that for all $\omega\in \Omega_T$
we have
\begin{equation}\label{z_Omega_T}
\|z^\omega\|_{X([-T,T]\times\R^d)}\leq
(2T)^{\theta}\left(\|u_0\|_{L^2}+\|u_1\|_{H^{-1}}\right)
\end{equation}

\noindent
and
\[P(\Omega_T^c)\leq C\exp\left(-cT^{2\left(\theta-\frac{d-2}{d+2}\right)}\right)=C\exp \left(-\frac{c}{T^\gamma}\right),\]
with $\gamma=2\big(\frac{d-2}{d+2}-\theta\big)>0$.

By Duhamel's formula \eqref{Duhamel},
we have that $v^\omega$ is a solution of \eqref{v} if and only if it satisfies
\begin{equation}\label{Duhamel_v}
v^\omega(t)=-\int_0^t \frac{\sin ((t-t')|\nabla|)}{|\nabla|}F(v^\omega+z^\omega)(t')dt'.
\end{equation}

\noindent
We define $\Gamma^\omega$ by
\begin{align}\label{Gammav}
\Gamma ^\omega v (t):=-\int_0^t \frac{\sin((t-t')|\nabla|)}{|\nabla|}F(v+z^\omega)(t')dt'.
\end{align}
We prove that, for all $\omega\in\Omega_T$, $\Gamma^\omega$ is a contraction on the ball
\begin{align*}
B_a:=\left\{v \in C\big([0,T];\dot{H}^1(\R^d)\big)\cap X\big([0,T]\times\R^d\big): \|v\|_{L^{\infty}_t([0,T]; \dot{H}^1_x(\R^d))}+\|v\|_{X([0,T]\times\R^d)}\leq a\right\},
\end{align*}
where $a$ is to be chosen later. 
By the Banach fixed point theorem,
this shows that the equation
$\Gamma^\omega v=v$, and therefore equation \eqref{v},
has a unique solution $v^\omega$ in $B_a$.

We first prove that $\Gamma^\omega$
maps $B_a$ into itself.
Using the Strichartz estimates \eqref{Strichartz} and \eqref{z_Omega_T}, 
we obtain for $v\in B_a$ and $\omega\in \Omega_T$ that
\begin{align}\label{Gamma}
\|\Gamma^\omega v\|_{L^{\infty}_t([0,T]; \dot{H}^1_x(\R^d))}&
+\|\Gamma^\omega v\|_{X([0,T]\times \R^d)}\leq C\|F(v+z^\omega)\|_{L^1_TL^2_x}\notag\\
&\leq C_1 \|v\|_{X([0,T]\times\R^d)}^{\frac{d+2}{d-2}}+C_1\|z^\omega\|_{X([0,T]\times\R^d)}^{\frac{d+2}{d-2}}\notag\\
&\leq C_1a^{\frac{d+2}{d-2}} +C_1(2T)^{\frac{(d+2)\theta}{d-2}}(\|u_0\|_{L^2}+\|u_1\|_{H^{-1}})^{\frac{d+2}{d-2}}.
\end{align}

\noindent
Taking $a$ such that $C_1a^{\frac{4}{d-2}}<\frac{1}{2}$ and $T$ sufficiently small, we obtain that
\[\|\Gamma^\omega v\|_{L^{\infty}_t([0,T]; \dot{H}^1_x(\R^d))}+\|\Gamma^\omega v\|_{X([0,T]\times\R^d)}\leq a,\]
and thus $\Gamma^\omega$ maps the ball $B_a$ into itself for all $\omega\in\Omega_T$.
Similarly, we have for $v_1, v_2 \in B_a$ and $\omega\in\Omega_T$ that
\begin{align}\label{diff_Gamma}
\|\Gamma^\omega v_1-\Gamma^\omega v_2&\|_{L^\infty_t([0,T],\dot{H}^1_x(\R^d))}+\|\Gamma^\omega v_1-\Gamma^\omega v_2\|_{X([0,T]\times\R^d)}\notag\\
&\leq C_2 \left(\|v_1\|_{X([0,T]\times\R^d)}^{\frac{4}{d-2}}+\|v_2\|_{X([0,T]\times\R^d)}^{\frac{4}{d-2}}+(2T)^{\frac{4\theta}{d-2}}(\|u_0\|_{L^2}+\|u_1\|_{\dot{H}^{-1}})^{\frac{4}{d-2}}\right)\notag\\
&\hphantom{XXXXXXXXXXX}\times \|v_1-v_2\|_{X([0,T]\times\R^d)}\notag\\
&\leq C_2 \left(2a^{\frac{4}{d-2}}+(2T)^{\frac{4\theta}{d-2}}(\|u_0\|_{L^2}+\|u_1\|_{H^{-1}})^{\frac{4}{d-2}}\right) \|v_1-v_2\|_{X([0,T]\times\R^d)}.
\end{align}

\noindent
Making $a$ and $T$ smaller if needed, we obtain that  $\Gamma^\omega$ is a contraction on $B_a$
for all $\omega\in\Omega_T$.
This proves the existence of a unique solution
$v^\omega$ of \eqref{v} in $B_a$.

Next, we show that $v^\omega\in L^{\infty}_t([0,T]; L^2_x(\R^d))$. 
By \eqref{Duhamel_v} and  \eqref{improved_Strichartz},
we have that 
\begin{align*}
\|v^\omega(t)\|_{L^2_x(\R^d)}&\leq \int_0^t \Big\| \frac{\sin ((t-t')|\nabla|)}{|\nabla|}F(v^\omega+z^\omega)(t')\Big\|_{L^2_x}dt'
\leq \int_0^t (t-t')\|F(v^\omega+z^\omega)(t')\|_{L^2_x}dt'\\
&\leq C_3T\Big(\|v^\omega\|_{X([0,T]\times\R^d)}^{\frac{d+2}{d-2}}+\|z^\omega\|_{X([0,T]\times\R^d)}^{\frac{d+2}{d-2}}\Big)\\
&\leq C_3T\Big(a^{\frac{d+2}{d-2}}+ (2T)^{\frac{(d+2)\theta}{d-2}}(\|u_0\|_{L^2}+\|u_1\|_{H^{-1}})^{\frac{d+2}{d-2}}\Big)
\leq Ta,
\end{align*}

\noindent
for all $\omega\in \Omega_T$, if $C_3a^{\frac{4}{d-2}}\leq \frac 12$ and $T$ is sufficiently small.

Finally, we prove that $\pa_t v^\omega \in L^{\infty}_t([0,T]; L^2_x(\R^d))$.
By \eqref{Strichartz}, the bound on \\$\|F(v+z^\omega)\|_{L^1_TL^2_x}$ obtained in \eqref{Gamma},
and the choice we made for $a$ and $T$, we have for $\omega\in\Omega_T$ that
\begin{align*}
\|\pa_t v^\omega\|_{L^{\infty}_t([0,T]; L^2_x(\R^d))}&\leq C\|F(v^\omega+z^\omega)\|_{L^1_TL^2_x}\\
&\leq C_1a^{\frac{d+2}{d-2}} +C_1(2T)^{\frac{(d+2)\theta}{d-2}}(\|u_0\|_{L^2}+\|u_1\|_{H^{-1}})^{\frac{d+2}{d-2}}\leq a.
\end{align*}
This concludes the proof of (a).


To prove (b), we decompose the solution $u^\omega$
into its linear and nonlinear parts $u^\omega= \tilde z^\omega + \tilde v^\omega$, where
\begin{align*}
\tilde z^\omega (t)&=S(t-t_\ast)\bigg(\Big(S(t_\ast)(u_0^\omega, u_1^\omega), \pa_t S(t_\ast)(u_0^\omega, u_1^\omega)\Big)+(w_0,w_1)\bigg)\\
&=S(t)(u_0^\omega, u_1^\omega)+ S(t-t_\ast)(w_0, w_1)
\end{align*}

\noindent
and the nonlinear part $\tilde v^\omega$ satisfies
\begin{align*}
\tilde v^\omega (t)= -\int_{t_\ast}^t \frac{\sin((t-t')|\nabla|)}{|\nabla|}F(\tilde v^\omega + \tilde z^\omega)(t')dt'.
\end{align*}

\noindent
As in (a), it is sufficient to design 
a fixed point argument to prove the local existence and uniqueness of $\tilde v^\omega$ on $[t_\ast - T', t_\ast+T']$ for all $\omega \in \Omega_{T}$.
The key observation is that, for all $\omega\in\Omega_T$, the improved local-in-time Strichartz estimates
hold uniformly on subintervals $[t_\ast-T', t_\ast + T']\subset [-T,T]$. Therefore,
we have
\begin{align*}
\|\tilde z^\omega\|_{X([t_\ast-T', t_\ast + T']\times\R^d)}
&\leq \|S(t)(u_0^\omega, u_1^\omega)\|_{X([t_\ast-T', t_\ast + T']\times\R^d)}\\
&+\|S(t-t_\ast)(w_0,w_1)\|_{X([t_\ast-T', t_\ast + T']\times\R^d)}\\
&\leq (2T)^{\theta}(\|u_0\|_{L^2}+\|u_1\|_{H^{-1}})
+\|S(t)(w_0,w_1)\|_{X([-T', T']\times\R^d)}. 
\end{align*}

\noindent
We then choose $T'<T$ sufficiently small depending on $w_0$ and $w_1$ such that 
$\|S(t)(w_0,w_1)\|_{X([-T', T']\times\R^d)}$ is small.
The rest of the proof 
follows exactly as that of (a).

\end{proof}


We conclude this section with the proof of the probabilistic small data global theory in Theorem \ref{scattering}.

\begin{proof}[Proof of Theorem \ref{scattering}]
To prove Theorem \ref{scattering},
it suffices to prove almost sure global existence and uniqueness
for the equation \eqref{v} satisfied by the nonlinear part of the solution $v^\omega:=u^\omega-z^\omega$.

Let $\eta>0$ sufficiently small such that
\[2C_1\eta^{\frac{4}{d-2}}\leq 1, \quad \quad 3C_2\eta^{\frac{4}{d-2}}\leq \frac 12,\]
where $C_1, C_2$ are the constants appearing in \eqref{Gamma}
and \eqref{diff_Gamma} above.
Then, by Proposition \ref{globalS},
there exists a set $\Omega_\eps$
with 
\[P(\Omega_\eps^c)\leq C\exp \left(-c\frac{\eta^2}{\eps^2(\|u_0\|_{\dot{H}^s}+\|u_1\|_{\dot{H}^{s-1}})^2}\right)\]
such that, for all $\omega\in \Omega_\eps$, we have
\[\|z^\omega\|_{X(\R\times\R^d)}
=\|S(t)(\eps u_0^\omega, \eps u_1^\omega)\|_{L_t^{\frac{d+2}{d-2}}\big(\R;L_x^{\frac{2(d+2)}{d-2}}(\R^d)\big)}\leq \eta.\]

In the following, we prove that $\Gamma^\omega$ defined in \eqref{Gammav} is a contraction on the ball
\begin{align*}
B_\eta:=\left\{v \in C\big(\R; \dot{H}^1(\R^d)\big)\cap X\big(\R\times\R^d\big): 
\|v\|_{L^{\infty}_t(\R; \dot{H}^1_x(\R^d))}+\|v\|_{X(\R\times\R^d)}\leq \eta\right\}.
\end{align*}
Indeed, similarly to \eqref{Gamma}
and \eqref{diff_Gamma},
we have for $\omega\in\Omega_\eps$ and $v, v_1, v_2\in B_\eta$ that
\begin{align*}
\|\Gamma^\omega v\|_{L^{\infty}_t(\R; \dot{H}^1_x(\R^d))}&
+\|\Gamma^\omega v\|_{X(\R\times\R^d)}
\leq C_1 \|v\|_{X(\R^\times\R^d)}^{\frac{d+2}{d-2}}+C_1\|z^\omega\|_{X(\R^\times\R^d)}^{\frac{d+2}{d-2}}\leq 2C_1\eta^{\frac{d+2}{d-2}}\leq \eta
\end{align*}
and
\begin{align*}
\|\Gamma^\omega v_1-\Gamma^\omega v_2&\|_{L^\infty_t(\R; \dot{H}^1_x(\R^d))}+\|\Gamma^\omega v_1-\Gamma^\omega v_2\|_{X(\R^\times\R^d)}\notag\\
&\leq C_2 \left(\|v_1\|_{X(\R^\times\R^d)}^{\frac{4}{d-2}}+\|v_2\|_{X(\R^\times\R^d)}^{\frac{4}{d-2}}+\|z^\omega\|_{X(\R^\times\R^d)}^{\frac{4}{d-2}}\right) \|v_1-v_2\|_{X(\R^\times\R^d)}\notag\\
&\leq 3C_2\eta^{\frac{4}{d-2}}\|v_1-v_2\|_{X(\R^\times\R^d)}\leq \frac 12 \|v_1-v_2\|_{X(\R^\times\R^d)}.
 \end{align*}
Thus, $\Gamma^\omega$ is indeed a contraction on $B_\eta$
as long as $\omega\in \Omega_\eps$.
Therefore, by the Banach fixed point theorem, for all $\omega\in \Omega_\eps$,
there exists a unique global solution $v^\omega\in B_\eta$
of equation \eqref{v}.
By the Strichartz estimates \eqref{Strichartz},
we also have that
\[\|\pa_t v^\omega\|_{L^\infty_t(\R; L^2_x(\R^d))}\lesssim \|v^\omega\|_{X(\R^\times\R^d)}^{\frac{d+2}{d-2}}+\|z^\omega\|_{X(\R^\times\R^d)}^{\frac{d+2}{d-2}}
\lesssim \eta^{\frac{d+2}{d-2}}.\]
Since the following global space-time bound holds
\[\|v^\omega\|_{X(\R^\times\R^d)}\leq \eta<\infty,\]
a standard argument then shows that, for all $\omega\in \Omega_\eps$, $(v^\omega,\pa_t v^\omega)$ scatters in $\dot{H}^1(\R^d)\times L^2(\R^d)$ to 
a linear solution, both forward and backward in time.

\end{proof}


\section{Deterministic local well-posedness}

This section is dedicated to the local well-posedness of
the energy-critical defocusing nonlinear wave equation with a deterministic 
perturbation on $\R^d$, $d=4$ and $5$. We start with a standard local well-posedness 
result and a blowup criterion. 
We then upgrade these to a ``good''
local well-posedness result in which
the time of existence 
depends only on 
the $\dot{H}^1(\R^d)\times L^2(\R^d)$-norm of the initial data and on the perturbation.
This upgraded local well-posedness
is one of the two main ingredients in proving Theorem \ref{main}, the other ingredient being a probabilistic energy bound.

\begin{proposition}[Deterministic local well-posedness]
\label{deterministic_LWP}
Let $d=4$ or $5$ and $(v_0,v_1)\in\dot{H}^1(\R^d)\times L^2(\R^d)$. 
Let $t_0\in\R$ and let $I$ be an interval containing $t_0$.
Then, there exists $\delta>0$ sufficiently small
such that if 
\[\|f\|_{X(I)}\leq \delta^{\frac{d-2}{d+2}}\]
and 
\begin{equation}\label{profile}
\|S(t-t_0)(v_0,v_1)\|_{X(I)}\leq \delta,
\end{equation}
the Cauchy problem
\begin{align}\label{v_data}
\begin{cases}
\pa_{t}^2v-\Delta v+F(v+f)=0\\
(v, \dt v)\big|_{t = t_0} = (v_0, v_1), 
\end{cases}
\end{align}
admits a unique solution $(v,\pa_t v)\in C(I; \dot{H}^1(\R^d)\times L^2(\R^d))$.
Here, $v$
is unique in the ball $B_{a}(I)$ of 
$X(I)$ defined by
\begin{equation}\label{Bab}
B_{a}(I):=\{v\in X(I): \|v\|_{X(I)}\leq a\},
\end{equation}
where $a=C_0\delta$ for some $C_0>0$.
\end{proposition}

The proof of Proposition \ref{deterministic_LWP}
is standard and therefore we omit it.
It consists in using Duhamel's formula \eqref{Duhamel}
to design a fixed point argument in $B_{a}$.
As a consequence of Proposition \ref{deterministic_LWP},
we obtain the following blowup criterion.


\begin{lemma}[Blowup criterion]
\label{blowup}
Let $d=4$ or $5$ and $(v_0,v_1)\in\dot{H}^1(\R^d)\times L^2(\R^d)$. 
Let $T_0<T_1<T_2$ and let $f$ be a function with the property that 
$\|f\|_{X([T_0,T_2])}<\infty$.
If $v$ is a solution on $[T_0,T_1]$ of 
the Cauchy problem
\begin{align}\label{v_dataT0}
\begin{cases}
\partial_t^2 v-\Delta v+F(v+f)=0\\
(v, \dt v)\big|_{t = T_0} = (v_0, v_1), 
\end{cases}
\end{align}
satisfying 
\[\|v\|_{X([T_0,T_1])}<\infty,\]
then there exists $\eps_0>0$ such that the solution $v$
can be uniquely extended to $[T_1,T_1+\eps_0]$.

Equivalently,
if $T_1<\infty$ is the maximal time of existence of the solution $v$ of \eqref{v_dataT0}, then
\[\|v\|_{X([T_0,T_1])}=\infty.\]
\end{lemma}


Condition \eqref{profile} in Proposition 
\ref{deterministic_LWP}
shows that the local time of existence of the solution of the perturbed NLW \eqref{v_data}
depends on the 
profile of the initial data $(v_0,v_1)$. 
In the following, we upgrade the 
local well-posedness result in Proposition 
\ref{deterministic_LWP} to a ``good''
local well-posedness, in which 
the local time of existence depends only on the 
$\dot{H}^1(\R^d)\times L^2(\R^d)$-norm
of $(v_0,v_1)$ and on the perturbation $f$.

\begin{proposition}[``Good'' local well-posedness]
\label{prop:main}
Let $d=4$ or $5$ and $(v_0,v_1) \in \dot{H}^1(\R^d)\times L^2(\R^d)$.
Let $K,\gamma >0$, $t_0<T$, and  
let $f$ be a real-valued function defined on $[t_0,T]$.

Then, there exists $\tau=\tau\Big(\|(v_0,v_1)\|_{\dot{H}^1(\R^d)\times L^2(\R^d)}, K,\gamma\Big)$ sufficiently small 
and non-increasing in the first two arguments
such that, 
if $f$ satisfies the condition
\begin{equation}\label{f}
\|f\|_{X([t_0,t_0+\tau_\ast])}\leq K\tau_\ast^\gamma
\end{equation}
for some $0<\tau_\ast\leq \tau$,
then
\eqref{v_data}
admits a unique solution $(v,\partial_tv)$
in $C([t_0,t_0+\tau_\ast]; \dot{H}^1(\R^d)\times L^2 (\R^d))$.
Moreover,
\begin{align}\label{deterministic_bound}
\|(v,\pa_t v)\|_{L^\infty_t([t_0,t_0+\tau_\ast];\dot{H}^1_x(\R^d)\times L^2_x(\R^d))}
+\|v\|&_{L^q_t([t_0,t_0+\tau_\ast]; L^r_x(\R^d))}\notag\\
&\leq C\left(\|(v_0,v_1)\|_{\dot{H}^1(\R^d)\times L^2(\R^d)}\right),
\end{align}
for all $\dot{H}^1(\R^d)$-wave admissible pairs $(q,r)$, where $C(\cdot)$ is a positive non-decreasing function.

Here uniqueness holds in the following sense.
There exists a family of disjoint intervals $\{\mathcal{I}_n\}_{n\in\N}$
covering $[t_0,t_0+\tau_\ast]$
such that $v$ is unique in each ball $B_{a}(\mathcal I_n)$
of $X(\mathcal I_n)$ for all $n\in\N$,
where $a>0$ is a sufficiently small constant.
 \end{proposition}


The key ingredient in the proof of Proposition \ref{prop:main}
consists in the following perturbation lemmas.
\begin{lemma}[Short-time perturbations]\label{short_perturb}
Let $d=4$ or $5$, $(v_0,v_1)\in \dot{H}^1(\R^d)\times L^2(\R^d)$, $I\subset\R$ be a compact time interval, and $t_0\in I$.
Let $v$ be a solution defined on $I\times\R^d$ of the perturbed equation
\[\pa_{t}^2 v-\Delta v+ F(v)=e\]

\noindent
with initial data $(v,\pa_t v)\big|_{t=t_0}=(v_0,v_1)$.
Let $(w_0,w_1)\in \dot{H}^1(\R^d)\times L^2(\R^d)$ 
and let $w$ be the solution of the energy-critical defocusing nonlinear wave equation on 
$I\times\R^d$ with initial data $(w,\pa_t w)\big|_{t=t_0}=(w_0,w_1)$.

Then, there exist $\delta>0$ and $\eps_0>0$
sufficiently small such that if $0<\eps<\eps_0$ and
\begin{align}
\|v\|_{X(I\times \R^d)}&\leq \delta\label{vdelta}\\
\left\|\left(v_0-w_0,v_1-w_1\right)\right\|_{\dot{H}^1(\R^d)\times L^2(\R^d)}&\leq \eps\label{initial}\\
\left\|e\right\|_{L^1_t(I; L^2_x(\R^d))}&\leq \eps,\label{e}
\end{align}
there exists $C\geq 1$ such that the following holds:
\begin{align*}
\sup_{t\in I}\left\|\left(v(t)-w(t),\pa_tv(t)-\pa_t w(t)\right)\right\|_{\dot{H}^1_x(\R^d)\times L^2_x(\R^d)}
+\left\|v-w\right\|_{L^q_t(I; L^r_x(\R^d))}
\leq C\eps,
\end{align*}
for all $\dot{H}^1(\R^d)$-wave admissible pairs $(q,r)$.
\end{lemma}


\begin{proof}
Without loss of generality,
we can assume that $t_0=\inf I$.
We set $V:=w-v$. Then, $V$ satisfies the equation
\begin{align*}
\pa_t^2 V-\Delta V+F(V+v)-F(v)+e=0.
\end{align*}

\noindent
By Duhamel's formula, Strichartz estimates, \eqref{Fbasic}, H\"older's inequality,
\eqref{vdelta}, \eqref{initial}, and \eqref{e}, we obtain for $0<\eps<\eps_0$ that
\begin{align*}
\|(V,\pa_t V)\|_{L^\infty([t_0,t]; \dot{H}^1(\R^d)\times L^2(\R^d))}
&+\|V\|_{L^q([t_0,t];L^r_x(\R^d))}+\|V\|_{X([t_0,t]\times \R^d)}\\
&\lesssim \left\|\left(V(t_0),\pa_t V(t_0)\right)\right\|_{\dot{H}^1(\R^d)\times L^2(\R^d)}+ \|V\|_{X([t_0,t]\times \R^d)}^{\frac{d+2}{d-2}}\\
&+\|V\|_{X([t_0,t]\times\R^d)}\|v\|_{X([t_0,t]\times \R^d)}^{\frac{4}{d-2}}
+\|e\|_{L^1([t_0,t];L^2_x(\R^d))}\\
&\lesssim \eps+ \delta^{\frac{4}{d-2}} \|V\|_{X([t_0,t]\times \R^d)}+\|V\|_{X([t_0,t]\times \R^d)}^{\frac{d+2}{d-2}},
\end{align*}
for all $\dot{H}^1(\R^d)$-wave admissible pairs $(q,r)$.
If $\eps_0$ and $\delta$ are sufficiently small, then a standard continuity argument yields
$\|V\|_{X([t_0,t]\times\R^d)}\lesssim \eps$
for all $t\in I$. We then obtain that
\begin{equation*}
\sup_{t\in I}\|(V,\pa_t V)\|_{\dot{H}^1_x(\R^d)\times L^2_x(\R^d)}+\|V\|_{L^q_t(I;L^r_x(\R^d))}\lesssim \eps.
\end{equation*}
\end{proof}


\begin{lemma}[Long-time perturbations]\label{lemma:perturbation}
Let $d=4$ or $5$, $(v_0,v_1)\in \dot{H}^1(\R^d)\times L^2(\R^d)$, $I\subset\R$ be a compact time interval, $t_0\in I$, and $M>0$.
Let $v$ be a solution defined on $I\times\R^d$ of the perturbed equation
\[\pa_{t}^2 v-\Delta v+ F(v)=e\]
with initial data $(v,\pa_t v)\big|_{t=t_0}=(v_0,v_1)$, satisfying
\begin{align}
\|v\|_{X(I\times \R^d)}&\leq M.\label{vM}
\end{align}

\noindent
Let $(w_0,w_1)\in \dot{H}^1(\R^d)\times L^2(\R^d)$ 
and let $w$ be the solution of the energy-critical defocusing nonlinear wave equation on $I\times\R^d$ with initial data $(w,\pa_t w)\big|_{t=t_0}=(w_0,w_1)$.

Then, there exists $\tilde\eps (M)>0$
sufficiently small such that if $0<\eps<\tilde{\eps}(M)$ and
\begin{align}
\left\|\left(v_0-w_0,v_1-w_1\right)\right\|_{\dot{H}^1(\R^d)\times L^2(\R^d)}&\leq \eps\label{initial2}\\
\left\|e\right\|_{L^1_t(I; L^2_x(\R^d))}&\leq \eps,\label{e2}
\end{align}
the following holds:
\begin{align*}
\sup_{t\in I}\left\|\left(v(t)-w(t),\pa_tv(t)-\pa_t w(t)\right)\right\|_{\dot{H}^1_x(\R^d)\times L^2_x(\R^d)}
+\left\|v-w\right\|_{L^q_t(I; L^r_x(\R^d))}
&\leq C(M) \eps,
\end{align*}
for all $\dot{H}^1(\R^d)$-wave admissible pairs $(q,r)$.
Here $C(M)\geq 1$ is a non-decreasing function of $M$.
\end{lemma}


\begin{proof}
Let $\delta>0$ be as in Lemma \ref{short_perturb}.
Without loss of generality, we can assume $t_0=\inf I$. 
The bound \eqref{vM} allows us to divide the interval $I$
into $J=J(M,\delta)$ subintervals $I_j=[t_j,t_{j+1}]$, such that
\begin{equation}
\|v\|_{X(I_j\times \R^d)}\sim\delta
\end{equation} 
for all $j=0,\dots,J-1$.
By \eqref{initial2}
and \eqref{e2} with $0<\eps<\eps_0$, Lemma \ref{short_perturb}
yields on the first interval $I_0$ that
\begin{align*}
\sup_{t\in I_0}\left\|\left(v(t)-w(t),\pa_tv(t)-\pa_t w(t)\right)\right\|_{\dot{H}^1_x(\R^d)\times L^2_x(\R^d)}
+\left\|v-w\right\|_{L^q_t(I_0; L^r_x(\R^d))}
&\leq C\eps,
\end{align*}

\noindent
for all $\dot{H}^1(\R^d)$-wave admissible pairs $(q,r)$.
In particular,
\[\left\|\left(v(t_1)-w(t_1),\pa_tv(t_1)-\pa_t w(t_1)\right)\right\|_{\dot{H}^1_x(\R^d)\times L^2_x(\R^d)}
\leq C\eps.\]
If $\eps$ is sufficiently small such that $C\eps< \eps_0$,
we can apply Lemma  \ref{short_perturb} on the interval $I_1$
and obtain
\begin{align*}
\sup_{t\in I_1}\left\|\left(v(t)-w(t),\pa_tv(t)-\pa_t w(t)\right)\right\|_{\dot{H}^1_x(\R^d)\times L^2_x(\R^d)}
+\left\|v-w\right\|_{L^q_t(I_1;L^r_x(\R^d))}
&\leq C^2\eps.
\end{align*}
Arguing recursively, we obtain
\begin{align*}
\sup_{t\in I_j}\left\|\left(v(t)-w(t),\pa_tv(t)-\pa_t w(t)\right)\right\|_{\dot{H}^1_x(\R^d)\times L^2_x(\R^d)}
+\left\|v-w\right\|_{L^q_t(I_j; L^r_x(\R^d))}
&\leq C^{j+1}\eps
\end{align*}
for each $j=0,1,\dots, J-1$,
as long as $\max_{j=0,J-1}C^j\eps<\eps_0$.
Since $J=J(M,\delta)$ is finite,
the conclusion follows with $\tilde{\eps}(M):=\frac{\eps_0}{C^{J(M,\delta)}}$.
\end{proof}


Before proceeding to the proof of Proposition \ref{prop:main},
we recall a global space-time bound for solutions 
of the energy-critical defocusing nonlinear wave equation
on $\R\times\R^d$, $d=4$ and $5$.
\begin{lemma}[Global space-time bound for energy-critical defocusing NLW on $\R^d$, $d=4, 5$]
Let $d=4$ or $5$ and $(v_0,v_1)\in\dot{H}^1(\R^d)\times L^2(\R^d)$.
Let $w$ be the solution of the
energy-critical defocusing nonlinear wave equation on $\R\times \R^d$:
\begin{equation*}
\begin{cases}
\pa_{t}^2w-\Delta w+F(w)=0\\
w(t_0)=v_0, \quad \pa_t w(t_0)=v_1.
\end{cases}
\end{equation*}
Then, the following holds:
\begin{equation}\label{bound_w}
\|w\|_{X(\R\times\R^d)}<C\left(\|(v_0,v_1)\|_{\dot{H}^1(\R^d)\times L^2(\R^d)}\right),
\end{equation} 
where $C(\cdot)$ is a positive non-decreasing function.
\end{lemma}


\begin{proof}
According to the work of Shatah and Struwe \cite{Shatah_Struwe},
it follows that $(w,\partial_t w)\in L^q_{\textup{loc}}(\R,\dot{B}^{\frac 12}_q\times\dot{B}^{-\frac 12}_q)$
with $q=\frac{2(d+1)}{d-1}$.
Then,
by Proposition 4.5 (ii) in \cite{Nakanishi1999},
we have that
\begin{equation}
\lim_{t\to\infty}\int_{\R^d}|w(t,x)|^{\frac{2d}{d-2}}dx=0.
\end{equation}
Proceeding as in the proof of Proposition 2.4
in \cite{Bahouri_Gerard} (see also \cite{Bahouri_Shatah}),
a simple argument then shows that $w\in L^q_t(\R; L^r_x(\R^d))$
for any $\dot{H}^1(\R^d)$-wave admissible pair $(q,r)$. 
In particular,
$w\in X(\R\times\R^d)$.

Next, in order to prove the global space-time bound \eqref{bound_w},
we use 
a concentration-compactness argument adapted to dimensions
$d=4$ and $5$.
More precisely, we use a profile decomposition theorem 
for solutions of the energy-critical defocusing nonlinear wave equation on $\R^d$,
$d=4$ and $5$.
This theorem states that if $\{(\phi_{0,n},\phi_{1,n})\}_{n\in\N}$
is a bounded sequence in $\dot{H}^1(\R^d)\times L^2(\R^d)$,
then for every $\ell\geq 1$, the corresponding
solutions $w_n$ of \eqref{NLW} with initial data 
$(w_n,\partial_t w_n)|_{t=0}=(\phi_{0,n},\phi_{1,n})$
can be decomposed (on a subsequence) into:
\begin{equation}\label{decomposition}
w_n(t,x)=\sum_{j=1}^\ell \frac{1}{(\lambda_n^{(j)})^{\frac{d-2}{2}}}U^{(j)}\left(\frac{t-t_n^{(j)}}{\lambda_n^{(j)}}, \frac{x-x_n^{(j)}}{\lambda_n^{(j)}}\right)+r_n^{(\ell)}(t,x),
\end{equation}
where $U^{(j)}$ are some solutions of \eqref{NLW} with initial data in $\dot{H}^1(\R^d)\times L^2(\R^d)$,
$\lambda_n^{(j)}>0$, $t_n^{(j)}\in\R$, $x_n^{(j)}\in\R^d$, and $\limsup_{n\to\infty}\|r_n^{(\ell)}\|_{X(\R\times\R^d)}\to 0$ as $\ell\to\infty$. 

In the case of dimension $d=3$, 
the analogue of the above profile decomposition theorem was proved by 
Bahouri and G\'erard in \cite{Bahouri_Gerard}.
We point out that the extension to dimensions $d=4$ and $5$ does not pose any difficulty 
and simply consists in changing the numerology, 
as dictated by the dimension-dependent Sobolev embeddings and Strichartz estimates. 
In dimension $d=5$, some additional care is needed since the 
degree of the nonlinearity $|w|^{\frac 43}w$
is not an integer. 
 
The global space-time bound \eqref{bound_w}
then follows by a contradiction argument, as in the proof of Corollary 2 in \cite{Bahouri_Gerard}.
Indeed, let us assume that there exists a sequence of solutions $\{w_n\}_{n\in\N}$
such that 
\[\sup_{n\in\N}E(w_n)<\infty, \quad \quad \lim_{n\to\infty}\|w_n\|_{X(\R\times \R^d)}=\infty.\] 
Since $\sup_{n\in\N}E(w_n)<\infty$, we can apply the above profile decomposition
theorem to the sequence $\{w_n\}_{n\in\N}$.
In particular, it follows by \eqref{decomposition} that $\{w_n\}_{n\in\N}$
is bounded in $X(\R\times\R^d)$. 
This is a contradiction, and hence \eqref{bound_w} holds.

See also \cite[Lemma 4.3, Corollary 4.5]{KM} for slightly different versions of the profile
decomposition theorem and of the global space-time bound \eqref{bound_w} in dimensions $d=3,4,5$.
\end{proof}


We end this section with the proof of Proposition \ref{prop:main}.

\begin{proof}[Proof of Proposition \ref{prop:main}]
We prove that
any solution $v$ of \eqref{v_data} defined on $[t_0,t_0+\tau_\ast]$, if it exists,
satisfies the a priori  bound
\begin{align}
\label{apriori}
\|(v,\pa_t v)\|_{L^\infty_t([t_0,t_0+\tau_\ast]; \dot{H}^1_x(\R^d)\times L^2_x(\R^d))}
+\|v\|&_{L^q_t([t_0,t_0+\tau_\ast]; L^r_x(\R^d))}\notag\\
&\leq C\left(\|(v_0,v_1)\|_{\dot{H}^1(\R^d)\times L^2(\R^d)}\right), 
\end{align}
for all $\dot{H}^1(\R^d)$-wave admissible pairs $(q,r)$
and for $C(\cdot)$ a
positive non-decreasing function,
provided that $0<\tau_\ast\leq \tau\big(\|(v_0,v_1)\|_{\dot{H}^1(\R^d)\times L^2(\R^d)}, K,\gamma\big)$
and $f$ satisfies \eqref{f}.
This, together with Proposition \ref{deterministic_LWP}
and Lemma \ref{blowup}, 
shows that there exists
a unique solution $(v,\pa_t v)\in C([t_0, t_0+\tau_\ast]; \dot{H}^1(\R^d)\times L^2(\R^d))$ satisfying the bound \eqref{apriori}.
Indeed, since
\[\|S(t-t_0)(v_0,v_1)\|_{X(\R\times \R^d)}\lesssim \|v_0\|_{\dot{H}^1(\R^d)}+\|v_1\|_{L^2(\R^d)}<\infty,\]
one can always find a small time interval
$\mathcal I_0=[t_0,t_1)\subset [t_0,t_0+\tau_\ast]$ on which 
the conditions of Proposition \ref{deterministic_LWP}
are satisfied. As a consequence, there exists a solution 
$(v,\pa_t v)\in C(\mathcal I_0;\dot{H}^1(\R^d)\times L^2(\R^d))$ of \eqref{v_data}, unique 
in the ball $B_{a}(\mathcal I_0)$ of $X(\mathcal I_0\times\R^d)$, where $a$ is sufficiently small. 
Furthermore, by the a priori bound \eqref{apriori} 
and Lemma \ref{blowup},
it follows that the solution $v$ can be extended to some interval 
$\mathcal I_1=[t_1,t_2)\subset [t_0,t_0+\tau_\ast]$ and the extension is unique in $B_{a}(\mathcal I_1)$. 
On the interval $\mathcal I_0 \cup \mathcal I_1$, $v$ still satisfies the a priori bound \eqref{apriori}
and hence it can be further extended. Arguing recursively,
it follows that we can extend the solution $v$ as long as \eqref{apriori} is satisfied.
Hence, we can define it on the whole interval $[t_0,t_0+\tau_\ast]$.
Moreover, there exist disjoint intervals $\{\mathcal I_n\}_{n\in\N}$
with $[t_0,t_0+\tau_\ast]=\bigcup_{n\in\N} \mathcal I_n$
such that $v$ is unique in $B_{a}(\mathcal I_n)$ for all $n\in\N$.

In the following, we focus on finding $\tau=\tau \big(\|(v_0,v_1)\|_{\dot{H}^1(\R^d)\times L^2(\R^d)}, K,\gamma\big)$ such that the a priori bound \eqref{apriori} holds for $\tau_\ast = \tau $,
provided \eqref{f}
holds for this value of $\tau_\ast$.
It will be clear from the proof below that \eqref{f} also implies \eqref{apriori}
for all $0<\tau_\ast<\tau$.
By Duhamel's formula \eqref{Duhamel}, Strichartz estimates \eqref{Strichartz}, and \eqref{f}, we have 
\begin{align}\label{vqr}
\|(v,\pa_t v)\|&_{L^\infty_t([t_0,t_0+\tau]; \dot{H}^1_x(\R^d)\times L^2_x(\R^d))}
+\|v\|_{L^q_t([t_0,t_0+\tau]; L^r_x(\R^d))}\notag\\
&\lesssim \|(v_0, v_1)\|_{\dot{H}^1(\R^d)\times L^2(\R^d)}
+\big\|F(v+f)\big\|_{L^{1}_t([t_0,t_0+\tau]; L^{2}_x(\R^d))}\notag\\
&\lesssim \|v_0\|_{\dot{H}^1(\R^d)}+\|v_1\|_{L^2(\R^d)}+ \|v\|^{\frac{d+2}{d-2}}_{X([t_0,t_0+\tau]\times\R^d)}
+ C(K\tau^{\gamma})^{\frac{d+2}{d-2}}
\end{align} 
for all $\dot{H}^1(\R^d)$-wave admissible pairs $(q,r)$.
As a consequence, 
in order to obtain \eqref{apriori},
it is sufficient to show that 
\[\|v\|_{X([t_0,t_0+\tau]\times\R^d)}\leq C\left(\|(v_0,v_1)\|_{\dot{H}^1(\R^d)\times L^2(\R^d)}\right),\]

\noindent
with $C(\cdot)$ a positive non-decreasing function.

Let $w$ be the solution of
the energy-critical defocusing nonlinear wave equation on $\R\times\R^d$
with the same initial conditions as $v$:
\begin{equation*}
\begin{cases}
\pa_{t}^2w-\Delta w+F(w)=0\\
w(t_0)=v_0, \quad \pa_t w(t_0)=v_1.
\end{cases}
\end{equation*}
By \eqref{bound_w},
we have that $\|w\|_{X(\R\times\R^d)}<C\big(\|(v_0,v_1)\|_{\dot{H}^1(\R^d)\times L^2(\R^d)}\big)$. 
Then, we divide 
$\R$ into $J=J\big(\|(v_0,v_1)\|_{\dot{H}^1(\R^d)\times L^2(\R^d)},\eta\big)$ subintervals $I_j=[t_j, t_{j+1}]$ such that
\[\|w\|_{X(I_j\times\R^d)}\sim\eta\]
for some small $\eta>0$ to be chosen later.

Let $\tau>0$ to be chosen later.
We write $[t_0,t_0+\tau]=\cup_{j=0}^{J'-1}([t_0,t_0+\tau]\cap I_j)$
for some $J'\leq J$, where $[t_0,t_0+\tau]\cap I_j\neq \emptyset$ for $0\leq j\leq J'-1$. 

Since the nonlinear evolution of $w$
on each $I_j$ is small, it follows that so is the linear evolution 
$S(t-t_j)(w(t_j),\pa_t w(t_j))$. Indeed, recall first Duhamel's formula
\[w(t)=S(t-t_j)(w(t_j),\pa_t w(t_j))-\int_{t_j}^t\frac{\sin((t-t')|\nabla|)}{|\nabla|}F\big(w(t')\big)dt' \text{ for } t\in I_j.\]
Then, by the Strichartz estimates \eqref{Strichartz}, we have that
\begin{align}
\|S(t-t_j)(w(t_j),\pa_t w(t_j))\|_{X(I_j\times\R^d)}
&\leq \|w\|_{X(I_j\times\R^d)}+C\|F(w)\|_{L^{1}_t(I_j; L^{2}_x(\R^d))}\label{linear_w}\\
&\leq \|w\|_{X(I_j\times\R^d)}+C\|w\|_{X(I_j\times\R^d)}^{\frac{d+2}{d-2}}
\leq \eta+C\eta^{\frac{d+2}{d-2}}\leq 2\eta,\notag 
\end{align} 
for $j=0,1,\dots, J'-1$ and $\eta$ sufficiently small.

In the following, we use the Perturbation Lemma \ref{lemma:perturbation}
to show that, on each interval $I_j$,
$v-w$ is small in the $L^\infty_t\big(I_j; \dot{H}^1_x(\R^d)\times L^2_x(\R^d)\big)$-norm,
as well as in Strichartz norms.
We first estimate $v$ on the interval $I_0$. Arguing as before and using $(v(t_0),\pa_t v(t_0))=(w(t_0),\pa_tw(t_0))$, we obtain that
\begin{align}
\|v\|_{X(I_0\times\R^d)}&\leq \|S(t-t_0)(w(t_0),\pa_t w(t_0))\|_{X(I_0\times\R^d)}+C\|v\|_{X(I_0\times\R^d)}^{\frac{d+2}{d-2}}+C\|f\|_{X(I_0\times\R^d)}^{\frac{d+2}{d-2}}\notag\\
&\leq 2\eta+C\|v\|_{X(I_0\times\R^d)}^{\frac{d+2}{d-2}}+C(K\tau^{\gamma})^{\frac{d+2}{d-2}}\label{vI_0}.
\end{align}
By taking $\eta\ll 1$ sufficiently small and $K\tau^\gamma\ll 1$, it follows by a standard continuity argument that $\|v\|_{X(I_0\times\R^d)}\leq 3\eta+CK\tau^\gamma$.
Furthermore, by taking 
$\tau=\tau (K,\gamma,\eta)$
sufficiently small such that
\begin{equation}
\label{first_tau}
CK\tau^\gamma\leq \eta,
\end{equation} 
we obtain that 
\[\|v\|_{X(I_0\times\R^d)}\leq 4\eta.\]
Thus, condition \eqref{vM} in Lemma \ref{lemma:perturbation}
is satisfied on $I_0$ with $M=4\eta$.
We are thus left with estimating the error
$e:=F(v+f)-F(v)$. 
First, consider $\eps_0$ to be chosen later
such that
$0<\eps_0<\tilde{\eps}(4\eta)$,
where $\tilde{\eps}(4\eta)$ is as in Lemma \ref{lemma:perturbation}.
As above,
we have that
\begin{align*}
\|e\|_{L^{1}_t(I_0; L^{2}_x(\R^d))}&\leq C\|v\|_{X(I_0\times\R^d)}^{\frac{4}{d-2}}\|f\|_{X(I_0\times\R^d)}+C\|f\|_{X(I_0\times\R^d)}^{\frac{d+2}{d-2}}
\leq C\eta^{\frac{4}{d-2}}K\tau^\gamma+C(K\tau^{\gamma})^{\frac{d+2}{d-2}}\\
&\leq CK\tau\leq \eps_0,
\end{align*}
provided we choose $\tau=\tau(K, \gamma, \eta, \eps_0)$ 
sufficiently small such that $CK\tau\leq\eps_0$. 
Thus,  condition \eqref{e2} is satisfied on $I_0$.

Applying the Perturbation Lemma \ref{lemma:perturbation} on the interval $I_0$,
we then deduce that 
\[\sup_{t\in I_0}\big\|(v-w, \pa_t v-\pa_t w)\big\|_{\dot{H}^1_x(\R^d)\times L^2_x(\R^d)}\leq C(4\eta)\eps_0.\]
In particular,
\begin{equation}\label{t_1}
\big\|\big(v(t_1)-w(t_1), \pa_t v(t_1)-\pa_t w(t_1)\big)\big\|_{\dot{H}^1_x(\R^d)\times L^2_x(\R^d)}\leq C(4\eta)\eps_0=:\eps_1.
\end{equation}
Then, proceeding as in \eqref{vI_0} and using \eqref{linear_w} and \eqref{first_tau}, we have that
\begin{align*}
\|v\|_{X(I_1\times\R^d)}&\leq \|S(t-t_1)(w(t_1),\pa_t w(t_1))\|_{X(I_1\times\R^d)}\\
&+\|S(t-t_1)((v-w)(t_1),(\pa_t v-\pa_t w)(t_1))\|_{X(I_1\times\R^d)}\\
&+C\|v\|_{X(I_1\times\R^d)}^{\frac{d+2}{d-2}}+C\|f\|_{X(I_1\times\R^d)}^{\frac{d+2}{d-2}}\leq 2\eta+\eps_1+C\|v\|_{X(I_1\times\R^d)}^{\frac{d+2}{d-2}}+C(K\tau^{\gamma})^{\frac{d+2}{d-2}}\\
&\leq 2\eta+\eps_1+C\|v\|_{X(I_1\times\R^d)}^{\frac{d+2}{d-2}}+\eta^{\frac{d+2}{d-2}}.
\end{align*}
Therefore, imposing that $\eps_1=\eps_1(\eta)<\eta$ 
and fixing $\eta\ll 1$
sufficiently small,
we have by a standard continuity argument that 
\[\|v\|_{X(I_1\times\R^d)}\leq 4\eta.\]
Thus, condition \eqref{vM} in Lemma \ref{lemma:perturbation} is satisfied on $I_1$ with $M=4\eta$. 
Then, by \eqref{t_1}, condition \eqref{initial2} is satisfied on the interval $I_1$ if we choose $\eps_0$ sufficiently small such that 
$\eps_1= C(4\eta)\eps_0< \tilde{\eps}( 4\eta)$. 
In what concerns the error, we have as above
\begin{align*}
\|e\|_{L^{1}_t(I_1; L^2_x(\R^d))}&\leq C\|v\|_{X(I_1\times\R^d)}^{\frac{4}{d-2}}\|f\|_{X(I_1\times\R^d)}
+C\|f\|_{X(I_1\times\R^d)}^{\frac{d+2}{d-2}}
\leq C\eta^{\frac{4}{d-2}}K\tau^\gamma+C(K\tau^{\gamma})^{\frac{d+2}{d-2}}\\
&\leq CK\tau\leq\eps_1,
\end{align*}
for $\tau=\tau(K, \gamma, \eta,\eps_0,\eps_1)$ sufficiently small. Therefore, condition \eqref{e2} is also satisfied
on $I_1$ and we can apply the Perturbation Lemma \ref{lemma:perturbation}
on this interval
 to obtain that
\[\sup_{t\in I_1}\big\|(v-w, \pa_t v-\pa_t w)\big\|_{\dot{H}^1_x(\R^d)\times L^2_x(\R^d)}\leq C(4\eta)^2\eps_0.\]

We proceed similarly for the intervals $I_2, \dots, I_{J'-1}$. On each $I_j$, $j=1,\dots, J'-1$, we impose that 
\[\eps_j:=C(4\eta)^{j}\eps_0\]
satisfies $\eps_j<\tilde{\eps}(4\eta)$
and $\e_j<\eta$.
In order to satisfy all these conditions, 
it is enough 
to fix 
\[\eps_0\left(\|(v_0,v_1)\|_{\dot{H}^1(\R^d)\times L^2(\R^d)}, \eta\right)=\frac 12\frac{\min (\eta, \tilde{\eps}(4\eta))}{C(4\eta)^{J\big(\|(v_0,v_1)\|_{\dot{H}^1(\R^d)\times L^2(\R^d)},\eta\big)}}.\] 
Furthermore, for all $j=0,1,\dots, J'-1$, we impose the condition
\begin{align*}
\|e\|_{L^1_t(I_j; L^2_x(\R^d))}\leq CK\tau\leq\eps_j.
\end{align*}
To satisfy this, we fix $\tau=\tau\big(\|(v_0,v_1)\|_{\dot{H}^1(\R^d)\times L^2(\R^d)}, K,\gamma\big)$ such that 
\[CK\tau\leq\min\{\eta,\eps_0,\eps_1,\dots,\eps_{J-1}\}=\min \left\{\eta, \eps_0\left(\|(v_0,v_1)\|_{\dot{H}^1(\R^d)\times L^2(\R^d)},\eta\right)\right\}.\]

\noindent
Since $J\big(\|(v_0,v_1)\|_{\dot{H}^1(\R^d)\times L^2(\R^d)},\eta\big)$
is non-decreasing in $\|(v_0,v_1)\|_{\dot{H}^1(\R^d)\times L^2(\R^d)}$,
we notice easily that $\tau$ can be chosen to be non-increasing in both
$\|(v_0,v_1)\|_{\dot{H}^1(\R^d)\times L^2(\R^d)}$ and $K$.

Applying the Perturbation Lemma \ref{lemma:perturbation} recursively on the intervals $I_j$, we conclude that any solution $v$ defined on $[t_0,t_0+\tau]\times\R^d$ satisfies the following a priori estimates:
\begin{align*}
\|v\|_{X([t_0,t_0+\tau]\times\R^d)}\leq 4\eta J'\left(\|(v_0,v_1)\|_{\dot{H}^1(\R^d)\times L^2(\R^d)}\right)\leq C\big(\|(v_0,v_1)\|_{\dot{H}^1(\R^d)\times L^2(\R^d)}\big),\\
\sup_{t\in [t_0,t_0+\tau]}\left\|\left((v-w)(t),(\pa_t v-\pa_t w)(t)\right)\right\|_{\dot{H}^1_x(\R^d)\times L^2_x(\R^d)}
+\|v-w\|_{L^q_t([t_0,t_0+\tau]; L^r_x(\R^d))}
\leq\tilde{\eps}(4\eta),
\end{align*}

\noindent
where $C(\cdot)$ is a positive non-decreasing function.
Combining this with \eqref{vqr} and \eqref{first_tau} yields the estimate \eqref{apriori}.

\end{proof}


\section{Almost sure global existence and uniqueness}

The goal of this section is to prove the main result of the paper, namely Theorem \ref{main}.
We start by stating and proving a probabilistic energy bound.
The conclusion of Theorem \ref{main} then follows from 
Theorem \ref{GWP}, Corollary \ref{GWPcor}, and Proposition \ref{invariant_Sigma} below.

We first recall a nonlinear Gronwall's inequality that will be useful in proving
the probabilistic energy bound for $d=5$. See, for example, \cite[Theorem 1, p.~360]{MPK}
for more details.
\begin{lemma}[A nonlinear Gronwall's inequality]
\label{Gronwall}
Fix $T>0$, $c\geq 0$, and $0\leq \alpha <1$.  Let $u$ and $b$ be two non negative continuous functions defined on $[0,T]$ 
and satisfying the estimate 
\[u(t)\leq c+ \int_0^t b(s)u^\alpha(s)ds,\]
for all $t\in [0,T]$. Then,
\begin{align*}
u(t)\leq \left(c^{1-\alpha}+(1-\alpha)\int_0^t b(s)ds\right)^{\frac{1}{1-\alpha}}
\end{align*}
for all $t\in [0,T]$.
\end{lemma}
%

Following the same lines as in 
\cite[Proposition 2.2]{BT3}, we show that the energy $E(v^\omega)$ of the nonlinear part $v^\omega$ of the solution $u^\omega$
 is almost surely bounded for solutions $v^\omega$
of the equation \eqref{v}.


\begin{proposition}[Probabilistic energy bound]\label{prop:energy}
Let $d=4$ or $5$ and $0<\eps\ll 1$.
Let $(u_0, u_1)\in H^{s}(\R^d)\times H^{s-1}(\R^d)$,
with $0<s\leq 1$ if $d=4$, and $0\leq s\leq 1$ if $d=5$,
and let 
$(u_0^\omega, u_1^\omega)$ be the randomization
defined in \eqref{R1}, satisfying \eqref{cond}. 
Given $1\leq T<\infty$, 
let $v^\omega$ be a solution of the 
Cauchy problem \eqref{v} on $[0,T]$. 
Then, there exists a set $\tilde{\Omega}_{T,\eps}\subset\Omega$
with $P(\tilde{\Omega}_{T,\eps}^c)<\frac{\eps}{2}$, such that for all $t\in [0,T]$ and all $\omega\in\tilde{\Omega}_{T,\eps}$,
we have that 
\[E(v^\omega(t))\leq C\left(T,\eps,\|(u_0,u_1)\|_{H^s(\R^d)\times H^{s-1}(\R^d)}\right),\]
and thus also
\begin{equation}\label{L2_apriori}
\big\|(v^\omega, \pa_t v^\omega)\big\|_{L^\infty_t([0,T]; H^1_x(\R^d)\times L^2_x(\R^d))}\leq C\left(T,\eps,\|(u_0,u_1)\|_{H^s(\R^d)\times H^{s-1}(\R^d)}\right),
\end{equation}
where $C\left(T,\eps,\|(u_0,u_1)\|_{H^s(\R^d)\times H^{s-1}(\R^d)}\right)$ is a constant depending only on $T$, $\eps$, and $\|(u_0, u_1)\|_{H^s(\R^d)\times H^{s-1}(\R^d)}$.
\end{proposition}


\begin{proof}
Taking the time derivative of the energy, we obtain that
\begin{align*}
\frac{d}{dt}E\big(v^\omega(t)\big)&=\int_{\R^d}\pa_tv^\omega\pa_{t}^2v^\omega+\nabla\pa_t v^\omega\cdot \nabla v^\omega+F\big(v^\omega)\pa_t v^\omega dx\\
&=\int_{\R^d}\pa_t v^\omega\Big(\pa_{t}^2v^\omega-\Delta v^\omega+F\big(v^\omega\big)\Big)dx=\int_{\R^d}\pa_t v^\omega\Big[F\big(v^\omega)-F\big(z^\omega+v^\omega\big)\Big]dx.
\end{align*}

\noindent
Using the Cauchy-Schwarz and H\"older's inequalities, it then follows that
\begin{align}\label{Egen}
\left|\frac{d}{dt}E\big(v^\omega(t)\big)\right|&\leq C \Big(E\big(v^\omega(t)\big)\Big)^{\frac12}\Big\|F\big(v^\omega\big)-F\big(z^\omega+v^\omega\big)\Big\|_{L^2_x(\R^d)}\notag\\
&\leq C\Big(E\big(v^\omega(t)\big)\Big)^{\frac12}\left(\big\|z^\omega\big\|_{L^{\frac{2(d+2)}{d-2}}_x(\R^d)}^{\frac{d+2}{d-2}}+\big\|\left|z^\omega\right|\left|v^\omega\right|^{\frac{4}{d-2}}\big\|_{L^2_x(\R^d)}\right).
\end{align}

We first consider the case $d=4$.
By \eqref{Egen}, we have that
\begin{align*}
\left|\frac{d}{dt}E\big(v^\omega(t)\big)\right|&\leq C\Big(E\big(v^\omega(t)\big)\Big)^{\frac12}\Big(\big\|z^\omega\big\|_{L^6_x(\R^4)}^3+\big\|z^\omega\big\|_{L^\infty_x(\R^4)}\big\|v^\omega\big\|_{L^4_x(\R^4)}^2\Big)
\end{align*}
and, therefore
\begin{align}\label{basic_energy}
\left|\frac{d}{dt}\Big(E\big(v^\omega(t)\big)\Big)^{\frac12}\right|&\leq C \big\|z^\omega\big\|_{L^6_x(\R^4)}^3+C\big\|z^\omega\big\|_{L^\infty_x(\R^4)}\Big(E\big(v^\omega(t)\big)\Big)^{\frac12}.
\end{align}

\noindent
Integrating this from $t=0$ to $t\leq T$, we then obtain that
\begin{align*}
\Big(E\big(v^\omega(t)\big)\Big)^{\frac12}&\leq C \big\|z^\omega\big\|_{X([0,T]\times \R^4)}^3+C\int_0^t\big\|z^\omega(s)\big\|_{L^\infty_x(\R^4)}\Big(E\big(v^\omega(s)\big)\Big)^{\frac12}ds.
\end{align*}
Then, by Gronwall's inequality,
it follows that
\begin{align}\label{energy_estim}
\Big(E\big(v^\omega(t)\big)\Big)^{\frac12}&\leq C \big\|z^\omega\big\|_{X([0,T]\times\R^4)}^3e^{C\|z^\omega\|_{L^1_t([0,T];L_x^\infty(\R^4))}}.
\end{align}

Next, we consider $K_1, K_2>0$ such that $C\exp(-cK_1^2)+C\exp(-cK_2^2)<\frac{\eps}{2}$,
where $C,c>0$ are such that both the estimates in Proposition \ref{proba_S} (ii) for $(q,r)=(3,6)$ and in Proposition \ref{proba_S} (iii) for $(1,\infty)$ hold with $C,c$.
Then, for $0<s\leq 1$, by Proposition \ref{proba_S} (ii) and (iii)
with $\l=K_1 T^{\frac43} (\|u_0\|_{L^2}+\|u_1\|_{H^{-1}})$ and $\l=K_2 T^2 (\|u_0\|_{H^{s}}+\|u_1\|_{H^{s-1}})$ respectively, we obtain that
there exists $\tilde{\Omega}_{T,\eps}(\R^4)\subset \Omega$
with $P(\tilde{\Omega}_{T,\eps}^c(\R^4))<\frac{\eps}{2}$
such that
\begin{align*}
\|z^\omega\|_{X([0,T]\times \R^4)}&\leq K_1 T^{\frac 43} \left(\|u_0\|_{L^2(\R^4)}+\|u_1\|_{H^{-1}(\R^4)}\right),\\
\|z^\omega\|_{L^1_t([0,T];L^\infty_x(\R^4))}&\leq K_2  T^2 \left(\|u_0\|_{H^{s}(\R^4)}+\|u_1\|_{H^{s-1}(\R^4)}\right)
\end{align*} 
for any $\omega\in \tilde{\Omega}_{T,\eps}(\R^4)$.
Therefore, by \eqref{energy_estim}, we conclude for $d=4$ that
\begin{align*}
E(v^\omega(t))&\leq \left(K_1^3 T^4 \left(\|u_0\|_{L^2(\R^4)}+\|u_1\|_{H^{-1}(\R^4)}\right)^3e^{CK_2 T^2 \left(\|u_0\|_{H^{s}(\R^4)}+\|u_1\|_{H^{s-1}(\R^4)}\right)}\right)^2\\
&\leq C\left(T,\eps,\|(u_0,u_1)\|_{H^s(\R^4)\times H^{s-1}(\R^4)}\right), 
\end{align*}
for all $t\in [0,T]$ and all $\omega\in \tilde{\Omega}_{T,\eps}(\R^4)$.
Notice that, for $d=4$, we used Proposition \ref{proba_S} (iii), which requires $s>0$.

We now turn to the case $d=5$.
Using \eqref{Egen}, we have that
\begin{align*}
\left|\frac{d}{dt}E\big(v^\omega(t)\big)\right|&\leq C\Big(E\big(v^\omega(t)\big)\Big)^{\frac12}\left(\big\|z^\omega\big\|_{L^{\frac{14}{3}}_x(\R^5)}^{\frac 73}+\big\|z^\omega\big\|_{L^{10}_x(\R^5)}\big\|v^\omega\big\|_{L^{\frac{10}{3}}_x(\R^5)}^{\frac 43}\right)
\end{align*}
and, therefore
\begin{align}\label{basic_energy5}
\left|\frac{d}{dt}\Big(E\big(v^\omega(t)\big)\Big)^{\frac12}\right|&\leq C \big\|z^\omega\big\|_{L^{\frac{14}{3}}_x(\R^5)}^{\frac 73}+C\big\|z^\omega\big\|_{L^{10}_x(\R^5)}\left(\Big(E\big(v^\omega(t)\big)\Big)^{\frac12}\right)^{\frac 45}.
\end{align}
Integrating from $t=0$ to $t\leq T$, we obtain
\begin{align*}
\Big(E\big(v^\omega(t)\big)\Big)^{\frac12}&\leq C \big\|z^\omega\big\|
_{X([0,T]\times \R^5)}^{\frac 73}+C\int_0^t\big\|z^\omega(s)\big\|_{L^{10}_x(\R^5)}\left(\Big(E\big(v^\omega(s)\big)\Big)^{\frac12}\right)^{\frac 45}ds.
\end{align*}
Then, by the nonlinear Gronwall's inequality in Lemma \ref{Gronwall} with $\alpha=\frac 45$, 
it follows that:
\begin{align*}
\Big(E\big(v^\omega(t)\big)\Big)^{\frac12}&\leq C \big\|z^\omega\big\|
_{X([0,T]\times \R^5)}^{\frac 73}+C\|z^\omega\|_{L^1_t([0,T];L^{10}_x(\R^5))}^{5}.
\end{align*}
Applying Proposition \ref{proba_S} (ii) as above, it follows that there exists a set $\tilde\Omega_{T,\eps}(\R^5)\subset \Omega$ with $P(\tilde\Omega_{T,\eps}^c(\R^5))=0$
such that
\begin{align*}
\|z^\omega\|_{X([0,T]\times \R^5)}&\leq K_1T^{\frac{10}{7}} \left(\|u_0\|_{L^2(\R^5)}+\|u_1\|_{H^{-1}(\R^5)}\right),\\
\|z^\omega\|_{L^1_t([0,T],L^{10}_x(\R^5))}&\leq K_2 T^2 \left(\|u_0\|_{L^2(\R^5)}+\|u_1\|_{H^{-1}(\R^5)}\right) 
\end{align*} 
for any $\omega\in \tilde{\Omega}_{T,\eps}(\R^5)$ and for some $K_1$ and $K_2$ depending on $\eps$. Therefore, for $d=5$,
we obtain that
\begin{align*}
E\big(v(t)\big)&\leq \left(K_1T^{\frac{10}{7}}\left(\|u_0\|_{L^2(\R^5)}+\|u_1\|_{H^{-1}(\R^5)}\right)\right)^{\frac{14}{3}}
+\left(K_2T^2\left(\|u_0\|_{L^2(\R^5)}+\|u_1\|_{H^{-1}(\R^5)}\right)\right)^{10}\\
&\leq C\Big(T,\eps,\|(u_0,u_1)\|_{L^2(\R^5)\times H^{-1}(\R^5)}\Big)
\end{align*}
for all $t\in [0,T]$ and all $\omega\in \tilde{\Omega}_{T,\eps}(\R^5)$.
Notice that, for $d=5$,
we only applied Proposition \ref{proba_S} (ii), which allows us to consider $(u_0,u_1)\in H^s(\R^5)\times H^{s-1}(\R^5)$ with $0\leq s\leq 1$, thus also including $s=0$.

Finally, \eqref{L2_apriori} follows by noticing that
\begin{equation*}
\big\|\big(v^\omega(t),\pa_tv^\omega(t)\big)\big\|_{\dot{H}^1_x(\R^d)\times L^2_x(\R^d)}^2
\leq 2 E(v^\omega(t))
 \leq C\left(T,\eps,\|(u_0,u_1)\|_{H^s(\R^d)\times H^{s-1}(\R^d)}\right)
\end{equation*}
and by using
\begin{align*}
\|v^\omega (t)\|_{L^2_x(\R^d)}&=\Big\|\int_0^t\pa_t v^\omega(s)ds\Big\|_{L^2_x(\R^d)} 
\leq \int_0^t \big\|\pa_t v^\omega (s)\big\|_{L^2_x(\R^d)}ds
\leq T\big\|\pa_t v^\omega\big\|_{L^\infty_t([0,T]; L^2_x(\R^d))}\\
&\leq C\left(T,\eps,\|(u_0,u_1)\|_{H^s(\R^d)\times H^{s-1}(\R^d)}\right).
\end{align*}
This concludes the proof.
\end{proof}


\begin{remark}\label{rem3D}
{\rm
The argument in the proof of Proposition \ref{prop:energy}
cannot be used to prove a probabilistic energy bound 
for the analogous problem on $\R^3$.
Indeed, from \eqref{Egen}
we see that in the proof of Proposition \ref{prop:energy}
we require the control of
$\|z^\omega |v^\omega|^{\frac{4}{d-2}}\|_{L^2(\R^d)}$
in terms of the energy $E(v^\omega)$.
We notice that:
\begin{equation}\label{3D}
\|z^\omega |v^\omega|^{\frac{4}{d-2}}\|_{L^2(\R^d)}\leq \|z^\omega\|_{L^\infty(\R^d)}\|v^\omega\|_{L^{\frac{8}{d-2}}(\R^d)}^{\frac{4}{d-2}}.
\end{equation}
The energy $E(v^\omega)$
only controls the $L^p(\R^d)$-norms of $v^\omega$
for $2\leq p\leq \frac{2d}{d-2}$. 
Therefore, in order to control the right hand-side of \eqref{3D} in terms of $E(v^\omega)$,
one needs $\frac{8}{d-2}\leq \frac{2d}{d-2}$,
that is $d\geq 4$. 
This shows that for $d=3$
a more intricate analysis is needed
 to prove 
a probabilistic energy bound.

}
\end{remark}

The main result of this section is the following theorem concerning the 
almost sure global existence and uniqueness of the energy-critical defocusing NLW in $H^{s}(\R^d)\times H^{s-1}(\R^d)$,
with $0<s\leq 1$ if $d=4$, and $0\leq s\leq 1$ if $d=5$.
\begin{theorem}[Almost sure global existence and uniqueness]\label{GWP}
Let $d=4$ or $5$.
Let $(u_0, u_1)\in H^{s}(\R^d)\times H^{s-1}(\R^d)$,
with $0<s\leq 1$ if $d=4$, and $0\leq s\leq 1$ if $d=5$,
and let 
$(u_0^\omega, u_1^\omega)$ be the randomization
defined in \eqref{R1}, satisfying \eqref{cond}. 

Then, the defocusing energy-critical NLW \eqref{NLW} on $\R^d$ admits almost surely
a unique global solution with initial data $(u_0^\omega, u_1^\omega)$ at $t=0$. More precisely,
there exists a set $\tilde{\Omega}\subset\Omega$ with $P(\tilde{\Omega})=1$
such that, for each $\omega \in \tilde{\Omega}$, there exists a unique global solution $u^\omega$ of equation \eqref{NLW} with $(u^\omega,\pa_t u^\omega)\big|_{t=0}=(u_0^\omega, u_1^\omega)$, in the class
\begin{align*}
(u^\omega,\pa_tu^\omega)&\in \Big(S(t)(u_0^\omega,u_1^\omega), \pa_t S(t)(u_0^\omega,u_1^\omega)\Big)+C\left(\R; H^1(\R^d)\times L^2(\R^d)\right)\\
&\subset C\left(\R; H^{s}(\R^d)\times H^{s-1}(\R^d)\right).
\end{align*}

Uniqueness here holds in the following sense.
The set $\tilde \Omega$
can be written as 
$\tilde \Omega=\cup_{\eps>0}\Omega_\eps$
with $P(\Omega_\eps^c)<\eps$
and
for any $\eps>0$, $\omega\in\Omega_\eps$, and $0<T<\infty$,
there exists a family of disjoint intervals $\{\mathcal{I}_n\}_{n\in\N}$
covering $[-T,T]$
such that the nonlinear part of the solution
$v^\omega=u^\omega - S(\cdot) (u_0^\omega, u_1^\omega)$ is unique in
the ball $B_{a}(\mathcal I_n)$
of $X(\mathcal I_n \times \R^d)$ for all $n\in\N$,
where $a>0$ is a small constant.
\end{theorem}


As in \cite{Bourgain94, Colliand_Oh}, the following proposition stating ``almost'' almost sure global existence and uniqueness for \eqref{NLW} readily implies Theorem \ref{GWP}.

\begin{proposition}[``Almost'' almost sure global existence and uniqueness]
\label{almost_almost}
Let $d=4$ or $5$.
Let $(u_0, u_1)\in H^{s}(\R^d)\times H^{s-1}(\R^d)$,
with $0<s\leq 1$ if $d=4$, and $0\leq s\leq 1$ if $d=5$,
and let 
$(u_0^\omega, u_1^\omega)$ be the randomization
defined in \eqref{R1}, satisfying \eqref{cond}. 

Then,
for any $0<\eps\ll 1$ and $T\geq 1$, there exists $\Omega_{T,\eps}\subset \Omega$
with $P(\Omega_{T,\eps}^c)<\eps$
such that for any $\omega\in \Omega_{T,\eps}$ there exists a unique solution $u^\omega$ of equation \eqref{NLW} on $[0,T]\times \R^d$,
with initial data $\left(u^\omega, \pa_t u^\omega\right)\big|_{t=0}=\left(u_0^\omega, u_1^\omega\right)$ in the class
\begin{align*}
(u^\omega,\pa_tu^\omega)&\in \Big(S(t)(u_0^\omega,u_1^\omega), \pa_t S(t)(u_0^\omega,u_1^\omega)\Big)+C\left([0,T]; H^1(\R^d)\times L^2(\R^d)\right)\\
&\subset C\left([0,T]; H^{s}(\R^d)\times H^{s-1}(\R^d)\right).
\end{align*}
Moreover, the nonlinear part of the solution $v^\omega =u^\omega -S(\cdot)(u_0^\omega, u_1^\omega)$ satisfies the bounds
\begin{align}\label{estim_strich}
\|v^\omega&\|_{L^q_t([0,T]; L^r_x(\R^d))}\notag\\
&\leq \tilde{F}\left(\sup_{t\in [0,T]}\left\|\left(v^\omega(t),\pa_t v^\omega(t)\right)\right\|_{\dot{H}^1_x(\R^d)\times L^2_x(\R^d)},
\|(u_0,u_1)\|_{H^s(\R^d)\times H^{s-1}(\R^d)}\right)T^7\Big(\log \frac{T}{\eps}\Big)^3\notag\\
&\leq C\big(T,\eps, \|(u_0,u_1)\|_{H^{s}(\R^d)\times H^{s-1}(\R^d)}\big),
\end{align}
for all $\dot{H}^1(\R^d)$- wave admissible pairs $(q,r)$.
Here $\tilde{F}:[0,\infty)\times [0,\infty)\to [0,\infty)$
is a non-decreasing function in both variables satisfying $F(0,\cdot)\equiv 0$.

Uniqueness here holds in the following sense.
There exists a family of disjoint intervals $\{\mathcal{I}_n\}_{n\in\N}$
covering $[0,T]$
such that the nonlinear part of the solution
$v^\omega$ is unique in each ball $B_{a}(\mathcal I_n)$
of $X(\mathcal I_n\times \R^d)$ for all $n\in\N$,
where $a>0$ is a small constant.
\end{proposition}


Before proving Proposition \ref{almost_almost},
we first show how Theorem \ref{GWP} follows from it.
\begin{proof}[Proof of Theorem \ref{GWP}]
By the time reversibility of the equation,
it is sufficient to prove the almost sure existence of unique solutions
defined on the time interval $[0,\infty)$.
For fixed $\eps>0$, we consider $T_j=2^j$ and $\e_j=2^{-j}\eps$,
and using Proposition \ref{almost_almost}, we obtain $\Omega_{T_j,\eps_j}$.
Considering now $\Omega_\eps:=\cap_{j=1}^\infty \Omega_{T_j,\eps_j}$,
we have that $P(\Omega_{\eps}^c)<\eps$ and 
\eqref{NLW} admits a unique solution on $[0,\infty)$ for all $\omega\in \Omega_\eps$.  
Finally, defining $\tilde{\Omega}:=\cup_{\eps>0}\Omega_\eps$,
we have that $P(\tilde{\Omega}^c)=0$ 
and \eqref{NLW} admits a unique solution on $[0,\infty)$ for all $\omega\in \tilde{\Omega}$.
\end{proof}

We continue
with the proof of Proposition \ref{almost_almost}.
The main ingredients of the proof are Proposition \ref{prop:energy}
giving the probabilistic energy bound
and Proposition \ref{prop:main} 
containing the ``good'' deterministic local
well-posedness.


\begin{proof}[Proof of Proposition \ref{almost_almost}]
By Proposition \ref{prop:energy}, there
exists $\tilde{\Omega}_{T,\eps}$, with $P(\tilde{\Omega}^{c}_{T,\eps})<\frac{\eps}{2}$
such that for any $\omega\in\tilde\Omega_{T,\eps}$, the solution of  \eqref{v} satisfies the a priori bound 
\begin{equation}\label{A1}
A:=\sup_{t\in [0,T]}\|(v^\omega(t),\pa_tv^\omega(t))\|_{\dot{H}^1_x(\R^d)\times L^2_x(\R^d)}
\leq C\big(T,\eps, \|(u_0,u_1)\|_{H^s(\R^d)\times H^{s-1}(\R^d)}\big).
\end{equation}
Then, we set
\begin{equation}\label{K}
K:=C\big(\|u_0\|_{L^2(\R^d)}+\|u_1\|_{H^{-1}(\R^d)}\big) \text{ and } \gamma<\frac{d-2}{d+2}.
\end{equation}
Consider $\tau=\tau(A, K,\gamma)$ defined in Proposition 
\ref{prop:main}. Let $\tau_\ast\leq \tau$ to be chosen later.
We cover the interval $[0,T]$ by $\Big[\frac{T}{\tau_\ast}\Big]$ or $\Big[\frac{T}{\tau_\ast}\Big]+1$ subintervals $\mathfrak{I}_k:=[k\tau_\ast,(k+1)\tau_\ast]\cap [0,T]$, $k=0,1,\dots$.
By Corollary \ref{cor:proba},
for each $k$, there exists $\Omega_k\subset\Omega$ with
\[P(\Omega_k^c)\leq C\exp\left(-\frac{c}{T^2\tau_\ast^{2(\frac{d-2}{d+2}-\gamma)}}\right)\]
such that for any $\omega\in \Omega_k$ we have
\begin{align*}
\|z^\omega\|_{X(\mathfrak{I}_k\times\R^d)}\leq K\tau_\ast^\gamma.
\end{align*}

Consider $\widehat \Omega_{T,\eps}:=\bigcap_{k} \Omega_k\subset \Omega$. 
Then, for any $\omega\in \widehat \Omega_{T,\eps}$ 
and for all $k=0,1,\dots$, we have 
\begin{align*}
\|z^\omega\|_{X(\mathfrak{I}_k\times\R^d)}\leq K\tau_\ast^\gamma
\end{align*}
and
\begin{align*}
P(\widehat \Omega^c_{T,\eps})&\leq \left(\Big[\frac{T}{\tau_\ast}\Big]+1\right) C\exp\left(-\frac{c}{T^2\tau_\ast^{2(\frac{d-2}{d+2}-\gamma)}}\right)\\
&\leq \frac{T}{\tau_\ast}\tau_\ast \exp\left(-\frac{c}{2T^2\tau_\ast^{2(\frac{d-2}{d+2}-\gamma)}}\right)
=T \exp\left(-\frac{c}{2T^2\tau_\ast^{2(\frac{d-2}{d+2}-\gamma)}}\right)
\end{align*}
if $\tau_\ast$ is sufficiently small.
Fixing 
\begin{equation}\label{tau_ast}
\tau_\ast (A, K,\gamma, T,\eps) = \min \left\{ \tau(A, K,\gamma), \frac{1}{2}\left(\frac{c}{2T^2\log\frac{2T}{\eps}}\right)^{\frac{d+2}{2(d-2-\gamma(d+2))}}\right\},
\end{equation}

\noindent
we obtain that $P( \widehat \Omega^c_{T,\eps})<\frac{\eps}{2}$.

We now define $\Omega_{T,\eps}:=\tilde\Omega_{T,\eps}\cap \widehat{\Omega}_{T,\eps}$.
Notice that $P(\Omega_{T,\eps}^c)<\eps$. 
Then, for any $\omega\in \Omega_{T,\eps}$,
we have that the conditions of Proposition
\ref{prop:main} are satisfied on each subinterval $\mathfrak{I}_k$ 
with $f=z^\omega$, $K$ and $\gamma$ defined above
in \eqref{K}, and 
$\|(v(k\tau_\ast), (k+1)\tau_\ast)\|_{\dot{H}^1_x(\R^d)\times L^2_x(\R^d)}\leq A$ with $A$ defined above in \eqref{A1}.
Applying successively Proposition \ref{prop:main} on each $\mathfrak{I}_k$, $k=0,1,\dots$, we obtain for all $\omega\in\Omega_{T,\eps}$
a unique solution 
$v^\omega$ of \eqref{v} such that $(v^\omega, \pa_t v^\omega)$
is in the class $C([0,T]; \dot{H}^1(\R^d)\times L^2(\R^d))$. 
By \eqref{L2_apriori}, it follows that
$(v^\omega,\pa_t v^\omega)$ also belongs to the class
$C([0,T]; H^1(\R^d)\times L^2(\R^d))$.
The uniqueness is in the sense of Proposition \ref{prop:main}.

Moreover, with the choice of $\tau_\ast$ in \eqref{tau_ast}
and fixing $\gamma=\frac{5d-14}{6(d+2)}<\frac{d-2}{d+2}$,
it follows by \eqref{deterministic_bound}
that
this solution satisfies
\begin{align*}
\|v^\omega\|_{L^q_t([0,T]; L^r_x(\R^d))}&
\leq C(A)\frac{T}{\tau_\ast}\leq C(A)T\max\left\{\frac{1}{\tau (A,K)}, \Big(T^2\log\frac{T}{\eps}\Big)^3\right\}\notag\\
&\leq \tilde{F}(A,K)T^7 \Big(\log\frac{T}{\eps}\Big)^3,\notag
\end{align*}
for all $\dot{H}^1(\R^d)$-wave admissible pairs $(q,r)$ and $A>0$.
Since $C(A)$ is a non-decreasing function of $A$
and $\tau (A,K)$ is non-increasing in both $A$ and $K$,
it follows that $\tilde{F}$ can be chosen to be a non-decreasing function 
in both variables.
Then, by \eqref{A1}, 
we conclude that
\[\|v^\omega\|_{L^q_t([0,T]; L^r_x(\R^d))}\leq C\big(T,\eps, 
\|(u_0,u_1)\|_{H^{s}(\R^d)\times H^{s-1}(\R^d)}\big).\]
\end{proof}

The following corollary shows that almost sure global existence and uniqueness of 
the energy-critical defocusing NLW on $\R^d$, $d=4$ and $5$,
can also be proved for more general initial data than the ones considered in Theorem \ref{GWP}.


\begin{corollary}[Enhanced almost sure global existence and uniqueness]\label{GWPcor}
Let $d=4$ or $5$.
Let $(u_0, u_1)\in H^{s}(\R^d)\times H^{s-1}(\R^d)$,
with $0<s\leq 1$ if $d=4$, and $0\leq s\leq 1$ if $d=5$,
and let 
$(u_0^\omega, u_1^\omega)$ be the randomization
defined in \eqref{R1}, satisfying \eqref{cond}. 
Let $t_\ast\in\R$ and $(v_0,v_1)\in H^1(\R^d)\times L^2(\R^d)$.

Then, there exists a set $\Omega'\subset\Omega$ with $P(\Omega')=1$
such that, for each $\omega \in \Omega'$, 
the energy-critical defocusing NLW on $\R^d$ 
with initial data 
\[(u^\omega,\pa_t u^\omega)\Big|_{t=t_\ast}
=\left(S(t_\ast)(u_0^\omega, u_1^\omega), \pa_t S(t_\ast)(u_0^\omega, u_1^\omega)\right)+(v_0,v_1)\] 
admits a unique global solution $u^\omega$  in the class
\begin{align*}
(u^\omega,\pa_tu^\omega)&\in \Big(S(t)(u_0^\omega,u_1^\omega), \pa_t S(t)(u_0^\omega,u_1^\omega)\Big)+C\left(\R; H^1(\R^d)\times L^2(\R^d)\right)\\
&\subset C\left(\R; H^{s}(\R^d)\times H^{s-1}(\R^d)\right).
\end{align*}

Uniqueness here holds in the following sense.
The set $ \Omega'$
can be written as 
$\Omega'=\cup_{\eps>0}\Omega_\eps'$
with $P((\Omega_\eps')^c)<\eps$
and
for any $\eps>0$, $\omega\in\Omega_\eps'$, and $0<T<\infty$,
there exists a family of disjoint intervals $\{\mathcal{I}_n\}_{n\in\N}$
covering $[t_\ast-T,t_\ast+T]$
such that the nonlinear part of the solution
$v^\omega=u^\omega - S(\cdot) (u_0^\omega, u_1^\omega)$ is unique in
the ball $B_{a}(\mathcal I_n)$
of $X(\mathcal I_n\times \R^d)$ for all $n\in\N$,
where $a>0$ is a small constant.
\end{corollary}


\begin{proof}
Let $0<\eps\ll 1$ and $2\leq T<\infty$. Then, $\max\left(|t_\ast|, |t_\ast+T|\right)\geq 1$.
We look for a solution of the energy-critical defocusing NLW on $\R^d$ of the form
\[u^\omega (t)=S(t)\big(u_0^\omega, u_1^\omega\big)
+v^\omega(t),\]

\noindent
where the nonlinear part $v^\omega$ satisfies 
\begin{align}\label{vw_01}
\begin{cases}
\pa_t^2 v^\omega-\Delta v^\omega+ F\big(v^\omega+z^\omega\big)=0\\
\big(v^\omega,\pa_t v^\omega\big)\Big|_{t=t_\ast}=(v_0,v_1)
\end{cases}
\end{align}

\noindent
with $z^\omega= S(t)(u_0^\omega, u_1^\omega)$.

In the following, we first prove a probabilistic energy bound for $v^\omega$.
Notice that, by the Sobolev embedding $\dot{H}^1(\R^d)\subset L^{\frac{2d}{d-2}}(\R^d)$,
we have
\begin{equation}\label{energyH1}
E(v^\omega(t))\leq \frac 12\big\|\big(v^\omega(t),\pa_tv^\omega(t)\big)\big\|_{\dot{H}^1(\R^d)\times L^2(\R^d)}^2
+C\|v^\omega(t)\|_{\dot{H}^1(\R^d)}^{\frac{2d}{d-2}}.
\end{equation}

We first consider the case $d=4$.
The same computations as in Proposition \ref{prop:energy}
show that \eqref{basic_energy} holds. Integrating \eqref{basic_energy} from $t_\ast$ to $t$, where $t_\ast\leq t\leq t_\ast+T$,
and using $\big(v^\omega,\pa_t v^\omega\big)\Big|_{t=t_\ast}=(v_0,v_1)$ and 
\eqref{energyH1},
it follows that
\begin{align*}
\Big(E\big(v^\omega(t)\big)\Big)^{\frac12}&
\leq C\big(\|v_0\|_{\dot{H}^1(\R^4)}+
 \|v_1\|_{L^2(\R^4)}+ \|v_0\|_{\dot{H}^1(\R^4)}^2\big)
 +C \big\|z^\omega\big\|_{X([t_\ast, t_\ast+T]\times \R^4)}^3\\
 &+C\int_{t_\ast}^{t}
 \big\|z^\omega(s)\big\|_{L^\infty_x(\R^4)}\Big(E\big(v^\omega(s)\big)\Big)^{\frac12}ds.\notag
\end{align*}
Then, by Gronwall's inequality, we obtain for all $t\in [t_\ast, t_\ast+T]$ that
\begin{align}\label{Ecor}
\Big(E\big(v^\omega(t)\big)\Big)^{\frac12}&
\leq C\left(\|v_0\|_{\dot{H}^1(\R^4)}+
 \|v_1\|_{L^2(\R^4)}+ \|v_0\|_{\dot{H}^1(\R^4)}^2
 + \big\|z^\omega\big\|_{X([t_\ast,t_\ast+T]\times \R^4)}^3\right)\notag\\
 &\hphantom{XXXXXX}\times  e^{C\|z^\omega\|_{L^1_t([t_\ast,t_\ast+T];L_x^\infty(\R^4))}}.
\end{align}

\noindent
Next, we consider $K_1, K_2>0$ such that $C\exp(-cK_1^2)+C\exp(-cK_2^2)<\frac{\eps}{2}$,
where $C,c>0$ are such that both the estimates in Proposition \ref{proba_S} (ii) for $(q,r)=(3,6)$ and in Proposition \ref{proba_S} (iii) for $(1,\infty)$ hold with $C,c$.
Then, by Proposition \ref{proba_S} (ii) and (iii)
with 
\[\l=K_1T^{\frac13}\max\left(|t_\ast|, |t_\ast+T|\right)\left(\|u_0\|_{L^2(\R^4)}+\|u_1\|_{H^{-1}(\R^4)}\right) \] 
and 
\[\l=K_2 T \max\left(|t_\ast|, |t_\ast+T|\right)\left(\|u_0\|_{H^{s}(\R^4)}+\|u_1\|_{H^{s-1}(\R^4)}\right)\] 
respectively, we obtain that
there exists $\tilde{\Omega}_{t_\ast, T,\eps}(\R^4)\subset \Omega$
with $P(\tilde{\Omega}_{t_\ast,T,\eps}^c(\R^4))<\frac{\eps}{2}$
such that
\begin{align*}
\|z^\omega\|_{X([t_\ast, t_\ast+T]\times \R^4)}&\leq K_1T^{\frac 13} \max\left(|t_\ast|, |t_\ast+T|\right) \left(\|u_0\|_{L^2(\R^4)}+\|u_1\|_{H^{-1}(\R^4)}\right) ,\\
\|z^\omega\|_{L^1_t([t_\ast, t_\ast+T]; L^\infty_x(\R^4))}&\leq K_2 T \max\left(|t_\ast|, |t_\ast+T|\right)\left(\|u_0\|_{H^{s}(\R^4)}+\|u_1\|_{H^{s-1}(\R^4)}\right) 
\end{align*} 
for any $\omega\in \tilde{\Omega}_{t_\ast, T,\eps}(\R^4)$ and $0<s\leq 1$.
Combining these with \eqref{Ecor} yields for all
 $\omega\in \tilde \Omega_{t_\ast, T,\eps}$ that
\begin{align}\label{tildeA}
&\|(v^\omega,\pa_t v^\omega)\|_{L^\infty_t([t_\ast, t_\ast+T]; \dot{H}^1_x(\R^4)\times L^2_x(\R^4))}
\leq C\Big(\|v_0\|_{\dot{H}^1(\R^4)}+
 \|v_1\|_{L^2(\R^4)}+ \|v_0\|_{\dot{H}^1(\R^4)}^2\notag\\
& \hphantom{XXXXXXXXXXXX}+K_1^3 T\max\left(|t_\ast|, |t_\ast+T|\right)^3
\left(\|u_0\|_{L^2(\R^4)}+\|u_1\|_{H^{-1}(\R^4)}\right)^3\Big)\notag\\
&\hphantom{XXXXXXXXXXXXXXXXXX} \times e^{CK_2T\max\left(|t_\ast|, |t_\ast+T|\right)\left(\|u_0\|_{H^{s}(\R^4)}+\|u_1\|_{H^{s-1}(\R^4)}\right) }\notag\\
&\hphantom{XXX}\leq C\big(t_\ast, T,\eps, \|(v_0,v_1)\|_{\dot{H}^1(\R^4)\times L^2(\R^4)}, 
\|(u_0,u_1)\|_{H^s(\R^4)\times H^{s-1}(\R^4)}\big)=:\tilde A(\R^4).
\end{align}

We now turn to the case $d=5$. 
The same computations as in Proposition \ref{prop:energy}
show that \eqref{basic_energy5} holds. Integrating \eqref{basic_energy5} from $t_\ast$ to $t$, where $t_\ast\leq t\leq t_\ast+T$,
and using $\big(v^\omega,\pa_t v^\omega\big)\Big|_{t=t_\ast}=(v_0,v_1)$ and 
\eqref{energyH1}, we obtain
\begin{align*}
\Big(E\big(v^\omega(t)\big)\Big)^{\frac12}&
\leq 
C\left(\|v_0\|_{\dot{H}^1(\R^5)}+
 \|v_1\|_{L^2(\R^5)}+ \|v_0\|_{\dot{H}^1(\R^5)}^{\frac 53}\right)
+C \big\|z^\omega\big\|
_{X([t_\ast,t_\ast+T]\times \R^5)}^{\frac 73}\\
&+C\int_{t_\ast}^t\big\|z^\omega(s)\big\|_{L^{10}_x(\R^5)}\left(\Big(E\big(v^\omega(s)\big)\Big)^{\frac12}\right)^{\frac 45}ds.
\end{align*}
Then, by the nonlinear Gronwall's inequality in Lemma \ref{Gronwall} with $\alpha=\frac 45$, 
it follows that:
\begin{align*}
\Big(E\big(v^\omega(t)\big)\Big)^{\frac12}
&\leq C\left(\|v_0\|_{\dot{H}^1(\R^5)}+
 \|v_1\|_{L^2(\R^5)}+ \|v_0\|_{\dot{H}^1(\R^5)}^{\frac 53}\right) 
 +C \big\|z^\omega\big\|
_{X([t_\ast,t_\ast+T]\times \R^5)}^{\frac 73}\\
&+C\|z^\omega\|_{L^1_t([t_\ast,t_\ast+T];L^{10}_x(\R^5))}^{5}
\end{align*}
for all $t\in [t_\ast, t_\ast+T]$.
Applying Proposition \ref{proba_S} (ii) as above,
we obtain the existence of a set $\tilde\Omega_{t_\ast,T,\eps}(\R^5)\subset \Omega$
such that $P(\tilde\Omega^c_{t_\ast,T,\eps}(\R^5))<\frac{\eps}{2}$ and
such that for all $\omega\in\tilde\Omega_{t_\ast,T,\eps}(\R^5)$ we have 
\begin{align}\label{tildeA5}
\|(v^\omega,\pa_t &v^\omega)\|_{L^\infty_t([t_\ast, t_\ast+T]; \dot{H}^1_x(\R^5)\times L^2_x(\R^5))}
\leq C\left(\|v_0\|_{\dot{H}^1(\R^5)}+
 \|v_1\|_{L^2(\R^5)}+ \|v_0\|_{\dot{H}^1(\R^5)}^{\frac 53}\right) \notag\\
& +\left(K_1T^{\frac{3}{7}}\max\left(|t_\ast|, |t_\ast+T|\right)\left(\|u_0\|_{L^2(\R^5)}+\|u_1\|_{H^{-1}(\R^5)}\right)\right)^{\frac{7}{3}}\notag\\
&+\left(K_2T\max\left(|t_\ast|, |t_\ast+T|\right)\left(\|u_0\|_{L^2(\R^5)}+\|u_1\|_{H^{-1}(\R^5)}\right)\right)^{5}\notag\\
&\leq C\Big(t_\ast, T,\eps,\|(v_0,v_1)\|_{\dot{H}^1(\R^5)\times L^2(\R^5)},\|(u_0,u_1)\|_{L^2(\R^5)\times H^{-1}(\R^5)}\Big)=:\tilde A(\R^5) .
\end{align}

We now return to general dimension $d=4$ or $5$.
For $\tilde A$ defined in \eqref{tildeA} if $d=4$, and in \eqref{tildeA5} if $d=5$, and for $K$ and $\gamma$ as in \eqref{K}, we consider $\tau(\tilde A, K, \gamma)$ defined in Proposition \ref{prop:main}.
Following exactly the same lines as in the proof of Proposition \ref{almost_almost}
with $[0,T]$ replaced by $[t_\ast, t_\ast+T]$, 
 $\tau(A, K, \gamma)$ replaced by  $\tau(\tilde{A}, K, \gamma)$, 
\begin{align*}
\tau_{\ast}\left(\tilde A, K, \gamma, t_\ast, T,\eps\right):=
\min\left\{\tau(\tilde A, K, \gamma), \frac 12\left(\frac{c}{2\max\left(|t_\ast|^2, |t_\ast+T|^2\right)\log \frac{2T}{\eps}}\right)^{\frac{d+2}{2(d-2-\gamma(d+2))}}\right\},
\end{align*}
and initial data $(v_0,v_1)$ at $t=t_\ast$ for the first interval $[t_\ast, t_\ast+\tau_\ast]$ on which we apply Proposition \ref{prop:main},
we obtain the existence of a set $\Omega_{t_\ast, T,\eps}'\subset \Omega$ with $P\big((\Omega_{t_\ast, T,\eps}')^c\big)<\eps$ such that, for any $\omega\in\Omega_{t_\ast, T,\eps}'$, \eqref{vw_01} admits a unique solution on $[t_\ast,t_\ast+T]$.
Corollary \ref{GWPcor} then follows exactly the same way Theorem \ref{GWP}
follows from  Proposition \ref{almost_almost}.
\end{proof}


Theorem \ref{GWP} above essentially states that there
exists a set of initial data $\Theta \subset H^s(\R^d)\times H^{s-1}(\R^d)$ with
$0< s\leq 1$ if $d=4$, and $0\leq s\leq 1$ if $d=5$, of full measure, such that
for each $(\phi_0, \phi_1)\in\Theta$, equation \eqref{NLW} admits a unique global solution
with $(\phi_0, \phi_1)$ as initial data at $t=0$.
We recall the notation for the solution map of equation \eqref{NLW}, $\Phi(t): (\phi_0,\phi_1)\mapsto u(t)$.
Even though at time $t=0$ we have $\Theta$ a set of initial data of full measure, 
this does not a priori prevent $\Phi (t)(\Theta)$ to be a set of small measure for $t\neq 0$. 
Proposition \ref{invariant_Sigma} below ensures
that there exists a set $\Sigma$  
on which the flow is globally defined and with the property that
$\Phi(t)(\Sigma)$ is of full measure for all $t\in\R$. 
See \cite[Proposition 3.1]{BT3} and \cite[Theorem 1.2]{OQ} for related results 
concerning cubic NLW on $\T^3$ and cubic NLS on $\T$, respectively. 


\begin{proposition}[Existence of invariant sets of full measure]
\label{invariant_Sigma}
Let $d=4$ or $5$, $0 < s\leq 1$ if $d=4$, and $0\leq s\leq 1$ if $d=5$.
Then, for any countable subgroup $\mathcal T$ of $(\R,+)$, there exists $\Sigma_\mathcal T \subset H^s(\R^d)\times H^{s-1}(\R^d)$ of full measure, 
i.e. $\mu_{(w_0,w_1)}(\Sigma_\mathcal T )=1$ for all $\mu_{(w_0,w_1)}\in \mathcal{M}_s$,
such that for every $(\phi_0,\phi_1)\in \Sigma_\mathcal T$, equation \eqref{NLW} admits a unique global solution $u$ with initial data $(u,\pa_t u)\big|_{t=0}=(\phi_0,\phi_1)$, and $\Phi(t)(\Sigma_\mathcal T) =\Sigma_\mathcal T$ for all $t\in\mathcal T$.

As a consequence, there exists a set $\Sigma \subset H^s(\R^d)\times H^{s-1}(\R^d)$ such that the flow of the energy-critical defocusing NLW \eqref{NLW} is globally defined on $\Sigma$ and $\Phi(t)(\Sigma)$ is of full measure for all $t\in\R$.
\end{proposition}


\begin{proof}
Let $t_\ast\in\R$. 
We first find full measure sets $\Theta$ and $\Theta_n(t_\ast)$ of initial data 
in $H^s(\R^d)\times H^{s-1}(\R^d)$ with $0<  s\leq 1$ if $d=4$, and $0\leq s\leq 1$ if $d=5$,
that give rise to unique global solutions of the energy-critical defocusing NLW on $\R^d$. 
This is merely a reformulation of Theorem \ref{GWP} and
Corollary \ref{GWPcor}.
These sets will then be used to construct 
a set of full measure $\Sigma$ 
such that the flow of equation \eqref{NLW} is globally defined on $\Sigma$
and $\Phi(t)(\Sigma)$ is of full measure for all $t\in\R$.

In the following, we present the proof for the case $d=4$.
Let $0<\gamma<\frac 13$. Given $A, K>0$, we recall  
$\tau (A, K,\gamma)$ defined in Proposition \ref{prop:main}.
For a finite time time interval $I=[a,b]$ and $\tau_\ast>0$,
we denote by $\mathfrak{I}_k(I)=[a+k\tau_\ast, a+(k+1)\tau_\ast]$, $k=0,1,\dots,$
the $\Big[\frac{|I|}{\tau_\ast}\Big]$ or $\Big[\frac{|I|}{\tau_\ast}\Big]+1$
subintervals of length $\tau_\ast$ covering $I$.
For $0<s\leq 1$, we then consider the set 
\begin{align*}
\Theta :=\bigg\{&(\phi_0,\phi_1)\in H^s(\R^4)\times H^{s-1}(\R^4):
S(t)(\phi_0,\phi_1)\in L^3_{t,{\rm loc}}L^6_x\cap L^1_{t,{\rm loc}}L^\infty_x, \text{ there exist } \\ &C,K>0 \text{ such that } \|S(t)(\phi_0,\phi_1)\|_{X(\mathfrak{I}_k(I)\times\R^4)}\leq K\tau_\ast^\gamma \text{ for all } I=[-T,T], 0< T<\infty,\\
&k=0,1,\dots, \text{ and some } 0<\tau_\ast (I) \leq 
\tau \bigg(C\|S(t)(\phi_0,\phi_1)\|_{X(I\times\R^4)}^3e^{C\|S(t)(\phi_0,\phi_1)\|_{L^1_IL^\infty_x}}, K,\gamma\bigg)
\bigg\}.
\end{align*}

\noindent
We have seen in the proofs of Proposition \ref{prop:energy}, Proposition \ref{almost_almost}, and Theorem \ref{GWP} that $\mu_{(u_0,u_1)}(\Theta)=1$
for any $\mu_{(u_0,u_1)}\in\mathcal M_s$.
Moreover, for any $(\phi_0,\phi_1)\in \Theta$,
there exists a unique global solution of equation \eqref{NLW}
with initial data $(u(0),\pa_t u(0))=(\phi_0,\phi_1)$.

Similarly, for $0<s\leq 1$, $t_\ast\in \R$, and $n\in\N$, we define
\begin{align*}
\Theta_n (t_\ast) : =\bigg\{&(\phi_0,\phi_1)\in H^s(\R^4)\times H^{s-1}(\R^4):
S(t)(\phi_0,\phi_1)\in L^3_{t,{\rm loc}}L^6_x \cap L^1_{t,{\rm loc}}L^\infty_x, \text{ there exist }\\ &C,K>0 \text{ such that } \|S(t)(\phi_0,\phi_1)\|_{X(\mathfrak{I}_k(I)\times\R^4)}\leq K\tau_\ast^\gamma \text{ for all } I=[t_\ast-T, t_\ast+T], \\
&0<T<\infty, k=0,1,\dots, \text{ and some } \\
&0<\tau_\ast (I) \leq 
\tau \left(C\left(n+n^2+\|S(t)(\phi_0,\phi_1)\|_{X(I\times\R^4)}^3\right)e^{C\|S(t)(\phi_0,\phi_1)\|_{L^1_IL^\infty_x}}, K,\gamma\right)
\bigg\}.
\end{align*}

\noindent
Notice that $\Theta=\Theta_0 (0)$. 
By Proposition \ref{prop:main}, we have that $\tau(A,K,\gamma)$
is non-increasing in $A$. As a consequence, $\Theta_n(t_\ast)\subset \Theta_m(t_\ast)$
for all $n\geq m$ and $t_\ast\in\R$.
By Corollary \ref{GWPcor} and its proof based on 
Proposition \ref{prop:energy} and Proposition \ref{almost_almost},
we have that 
$\mu_{(u_0,u_1)}(\Theta_n(t_\ast))=1$
for all $\mu_{(u_0,u_1)}\in\mathcal M_s$, $n\in\N$, and $t_\ast\in\R$.
Moreover, for any $t_\ast\in\R$, $(v_0,v_1)\in H^1(\R^4)\times L^2 (\R^4)$ with 
$\|(v_0,v_1)\|_{H^1\times L^2}\leq n$, and any
$(\phi_0,\phi_1)\in\Theta_n (t_\ast)$,
the defocusing cubic NLW on $\R^4$ with initial data
\begin{equation}\label{dataTheta_n}
(u, \pa_t u)\Big|_{t=t_\ast}=S(t_\ast)(\phi_0,\phi_1)+(v_0,v_1)
\end{equation}
admits a unique global solution in the class 
\begin{align}\label{uTheta_n}
(u,\pa_tu)&\in \Big(S(t)(\phi_0,\phi_1), \pa_t S(t)(\phi_0,\phi_1)\Big)+C\left(\R; H^1(\R^4)\times L^2(\R^4)\right).
\end{align}

Next, we show that 
for any $t_0\in\R$, $n\in\N$, and $(\phi_0,\phi_1)\in\Theta_n (t_\ast)$,
we have
\begin{equation}\label{STheta_n}
(\psi_0,\psi_1):=\big(S(t_0)(\phi_0,\phi_1), \pa_t S(t_0)(\phi_0,\phi_1)\big) \in \Theta_n (t_\ast-t_0).
\end{equation}

\noindent
It follows easily that
\[S(t)(\psi_0,\psi_1)=S(t+t_0)(\phi_0,\phi_1)\in L^3_{t,{\rm loc}}L^6_x \cap L^1_{t,{\rm loc}}L^\infty_x.\]
Let $I= [t_\ast-T, t_\ast + T]$ for some $0<T<\infty$.
By the second condition in the definition of $\Theta_n (t_\ast)$
applied to $(\phi_0,\phi_1)$ on the interval $[t_\ast-T, t_\ast + T]$,
it follows that 
\[\|S(t-t_0)(\psi_0,\psi_1)\|_{X(\mathfrak{I}_k([t_\ast-T, t_\ast + T])\times\R^4)}=\|S(t)(\phi_0,\phi_1)\|_{X(\mathfrak{I}_k([t_\ast-T, t_\ast + T])\times\R^4)}\leq K\tau_\ast^\gamma\]
for all $k$, and for some positive $\tau_\ast$ satisfying
\[\tau_\ast\leq \tau \left(C\left(n+n^2+\|S(t-t_0)(\psi_0,\psi_1)\|_{X([t_\ast-T, t_\ast + T]\times\R^4)}^3\right)e^{C\|S(t-t_0)(\psi_0,\psi_1)\|_{L^1_{[t_\ast-T, t_\ast + T]}L^\infty_x}}, K,\gamma\right).\]
Therefore, a simple change of variables shows that
$(\psi_0,\psi_1)$ satisfies the second condition in the definition of $\Theta_n (t_\ast-t_0)$ on the interval $[t_\ast-t_0-T, t_\ast -t_0+ T]$. Hence, $(\psi_0,\psi_1)\in \Theta_n (t_\ast -t_0)$
and \eqref{STheta_n} is proved. 

Let $\mathcal T$ be a countable subgroup of $(\R,+)$.
We then define
\[\tilde \Theta_\mathcal T:=\bigcap_{t_\ast\in\mathcal T}\bigcap_{n\in\N}\Theta_n (t_\ast).\]
Notice that $\mu_{(u_0,u_1)}(\tilde \Theta_\mathcal T)=1$
for all $\mu_{(u_0,u_1)}\in\mathcal M_s$,
since 
$\tilde\Theta_\mathcal T$ is a countable intersection of full measure sets
$\Theta_n (t_\ast)$.
By \eqref{STheta_n}
and since $\mathcal T-t=\mathcal T$ for any $t\in\mathcal T$,
it follows that for all $t\in \mathcal T$ and $(\phi_0,\phi_1)\in\tilde\Theta_\mathcal T$, we have
\begin{equation}\label{STheta}
\big(S(t)(\phi_0,\phi_1), \pa_t S(t)(\phi_0,\phi_1)\big) \in \tilde\Theta_\mathcal T.
\end{equation}

Using the fact that $0\in\mathcal T$ and thus
$\tilde \Theta_\mathcal T \subset \bigcap_{n\in\N} \Theta_n (0)$,
it follows 
from the above discussion regarding the properties of $\Theta_n(t_\ast)$
that equation \eqref{NLW} admits a unique global solution 
with initial data at $t=0$ of the form
$(\phi_0,\phi_1)+(v_0,v_1)$,
where
$(\phi_0,\phi_1)\in\tilde\Theta_\mathcal T$ and $(v_0,v_1)\in H^1(\R^4)\times L^2(\R^4)$.
Moreover, by \eqref{uTheta_n} and \eqref{STheta}, we have
\[\left(u(t),\pa_t u(t)\right)\in \tilde\Theta_\mathcal T + H^1(\R^4)\times L^2(\R^4) \text{ for all } t\in \mathcal T.\]

\noindent
In other words, setting $\Sigma_\mathcal T:=\tilde\Theta_\mathcal T + H^1(\R^4)\times L^2(\R^4)$,
the flow of equation \eqref{NLW} is defined globally in time on $\Sigma_\mathcal T$ and has the property that
\[\Phi(t)\left(\Sigma_\mathcal T\right)\subset \Sigma_\mathcal T \text{ for all } t\in \mathcal T.\]

\noindent
Using the time reversibility of the equation
and the fact that $-t\in\mathcal T$ for all $t\in\mathcal T$, it then follows that
$\Phi(t)\left(\Sigma_\mathcal T\right)= \Sigma_\mathcal T$ for all $t\in \mathcal T$. 
Moreover, $\mu_{(u_0,u_1)}(\Sigma_\mathcal T)=1$
for all $\mu_{(u_0,u_1)}\in\mathcal M_s$,
since $\tilde\Theta_\mathcal T$ is of full measure.

Lastly, noticing that given $t\in\R$ arbitrary, 
$t\Z$ is a countable subgroup of $(\R,+)$ containing $t$, we define 
\[\Sigma :=\bigcup_{t\in\R} \Sigma_{t\Z}.\]
It easily follows that $\Sigma$ is of full measure and the flow of equation \eqref{NLW} is globally 
defined on $\Sigma$. 
Moreover,
\begin{align*}
\Sigma_{t\Z}=\Phi(t) (\Sigma_{t\Z})\subset \Phi(t)(\Sigma).
\end{align*}
Therefore,
\[\mu_{(u_0,u_1)}\left(\Phi(t)(\Sigma)\right)=1\]
for all $t\in\R$ and $\mu_{(u_0,u_1)}\in\mathcal M_s$.
This completes the proof of Proposition \ref{invariant_Sigma} in the case $d=4$.

The proof for $d=5$ is completely analogous. 
The only difference is the definition of $\Theta_n(t_\ast)$
(and thus also the definition of $\Theta=\Theta_0(0)$) which, for $d=5$, becomes
\begin{align*}
\Theta_n \big(t_\ast, &\R^5\big) : =\bigg\{(\phi_0,\phi_1)\in H^s(\R^5)\times H^{s-1}(\R^5):
S(t)(\phi_0,\phi_1)\in L^{\frac 73}_{t,{\rm loc}}L^{\frac{14}{3}}_x \cap L^1_{t,{\rm loc}}L^{10}_x, \\ 
&\text{there exist }C,K>0 \text{ such that } \|S(t)(\phi_0,\phi_1)\|_{X(\mathfrak{I}_k(I)\times\R^5)}\leq K\tau_\ast^\gamma \\
&\text{for all } I=[t_\ast-T, t_\ast+T], 0<T<\infty, k=0,1,\dots, \text{ and some } \\
&0<\tau_\ast (I) \leq 
\tau \left(C\Big(n+n^{\frac 53}+\|S(t)(\phi_0,\phi_1)\|_{X(I\times\R^5)}^{\frac 73}+\|S(t)(\phi_0,\phi_1)\|_{L^1_IL^{10}_x}^{5}\Big), K,\gamma\right)
\bigg\},
\end{align*}
where $0\leq s\leq 1$ and $0<\gamma<\frac{3}{7}$.

\end{proof}


\section{Probabilistic continuous dependence of the flow}

In this section
we prove the probabilistic continuity of the flow of 
the energy-critical defocusing cubic NLW \eqref{NLW} on $\R^4$ in $H^s(\R^4)\times H^{s-1}(\R^4)$,
$0<s\leq 1$, with respect to the initial data.
The notion of probabilistic continuity that we use here
was first proposed by Burq and Tzvetkov in \cite{BT3}.
For the readers' convenience, we first recall two lemmas
from Appendix A.2 of \cite{BT3}, that will be used in the proof
of Proposition \ref{cond_proba} below.

\begin{lemma}[Lemma A.9 in \cite{BT3}]\label{lemma:A9}
If $h$ is a Bernoulli random variable independent of 
a real random variable $g$ with symmetric distribution $\theta$,
then $h g$ has the same distribution $\theta$ as $g$. 
\end{lemma}
%

\begin{lemma}[Lemma A.8 in \cite{BT3}]\label{lemma:A8}
Let $Y_j$, $j=1,2$, be two Banach spaces endowed with measures $\mu_j$.
Let $f:Y_1\times Y_2 \to \C$ and $g_1,g_2:Y_2\to\C$
be three measurable functions.
Then
\begin{align*}
\mu_1\otimes \mu_2& \left((x_1,x_2)\in Y_1\times Y_2:
|f(x_1,x_2)|>\lambda \,
\Big|\, 
|g_1(x_2)|\leq \eps, \, |g_2(x_2)|\leq R\right)\\
&\leq \sup_{x_2\in Y_2, |g_1(x_2)|\leq \eps, |g_2(x_2)|\leq R}
\mu_1\left(x_1\in Y_1: |f(x_1,x_2)|>\lambda\right).
\end{align*}

\end{lemma}

Building upon the improved local-in-time Strichartz 
estimates in Proposition \ref{proba_S},
one can use the strategy developed by
Burq and Tzvetkov in Appendix A.2 of \cite{BT3}
to obtain the following result.


\begin{proposition}\label{cond_proba}
Let $d=4$. Assume that in the definition of the randomization
\eqref{R1}, the smooth cutoff $\psi$
is replaced
by the characteristic function $\chi_{Q_0}$ of the unit cube $Q_0$ centered at the origin.
Assume also that the probability distributions $\mu_{0,j}$, $\mu_{n,j}^{(1)}$, $\mu_{n,j}^{(2)}$,
$n\in\mathcal I$, $j=0,1$ are symmetric. 

Let $0<s\leq 1$, $T> 0$, and $\mu\in \mathcal{M}_{s}$.
Then, given $1\leq q<\infty$ and $2\leq r\leq\infty$, 
there exist constants $C,c>0$
such that for every $\eps, \lambda, \Lambda, R>0$, we have
\begin{align}\label{mu_times_mu}
\mu \otimes \mu
&\Big( \left((w_0,w_1), (w_0',w_1')\right)\in \left(H^{s}(\R^4)\times H^{s-1}(\R^4)\right)^2:\notag\\
&
\left\|S(t)(w_0-w_0', w_1-w_1')\right\|_{L^q_t([0,T]; L^r_x)}>\lambda \text{ or }
\left\|S(t)(w_0+w_0', w_1+w_1')\right\|_{L^q_t([0,T]; L^r_x)}>\Lambda
\notag\\
& \Big| \quad
\left\|(w_0-w_0', w_1-w_1')\right\|_{H^{s}\times H^{s-1}}\leq \eps
\text{  and } \left\|(w_0+w_0', w_1+w_1')\right\|_{H^{s}\times H^{s-1}}
\leq R\Big)\notag\\
&\leq C\left(e^{-c\frac{\lambda^2}{\max(1,T^2)T^{\frac 2q}\eps^2}}+e^{-c\frac{\Lambda^2}{\max(1,T^2)T^{\frac 2q}R^2}}\right).
\end{align}
\end{proposition}


\begin{proof}
For simplicity of notations, in the following we assume $T\geq 1$. 
The case $0<T<1$ is completely analogous, with the only change coming from 
Proposition \ref{proba_S} (ii) and (iii).

We consider $Y:=\R\times \R^{\mathcal I}$
equipped with the Banach structure induced by the $\ell^\infty$-norm,
where $\mathcal I$ is the index set in \eqref{index} for $d=4$.
If $\mathcal{B}(\R)$ denotes the Borel $\sigma$-algebra of $\R$, 
we endow $Y$
with the $\sigma$-algebra 
$\mathcal{B}(\R)^{\otimes \{0\}\cup \mathcal I}$
generated by 
\[\left\{\underset{n\in\{0\}\cup \mathcal I}{\times} A_n:
A_n \in \mathcal{B}(\R) \text{ and } A_n=\R
\text{ except for finitely many } n\right\}.\]
Let $\{k_n\}_{n\in  \{0\} \cup \mathcal I}$ be a 
system of independent Bernoulli variables.
Then, the map
\[\omega\mapsto \{k_n(\omega)\}_{n\in  \{0\} \cup \mathcal I}\]
is measurable and we endow $Y$ with the probability measure induced by this map:
\[\mu_0(A):=P\left(\left\{k_n(\omega)\right\}_{n\in  \{0\} \cup \mathcal I}\in A\right) \quad \text{for any } \quad
A\in \mathcal B(\R)^{\otimes \{0\}\cup \mathcal I}.\]

For $\{h_n\}_{n\in \{0\}\cup\mathcal I}$, we set $h_{-n}=h_n$,
which defines $h_n$ for all $n\in\Z^4$. Then, for 
$\phi\in L^2(\R^4)$ real-valued, we set
\begin{equation}\label{h}
h\odot \phi: =\sum_{n\in\Z^4} h_n\chi_{Q_0} (D-n)\phi = h_0\chi_{Q_0} (D)\phi +2\text{Re} \sum_{n\in\mathcal I} h_n \chi_{Q_0}(D-n) \phi,
\end{equation}

\noindent
where $\chi_{Q_0}$ is 
the characteristic function of the unit cube centered at the origin
$Q_0=\Big[-\frac 12,\frac 12\Big)^4$.
Since 
\begin{equation}\label{support}
(n+Q_0)\cap (m+Q_0)=\emptyset
\text{ for any } n,m\in\Z^4,  \quad n\neq m,
\end{equation} 

\noindent
we have
\begin{align}
h\odot \phi^\omega&=\sum_{n\in \Z^4} h_n \chi_{Q_0} (D-n) \phi^\omega
=\sum_{n\in\Z^4} h_n\chi_{Q_0}(D-n)\sum_{m\in\Z^4}g_m(\omega)\chi_{Q_0}(D-m)\phi\\
&=\sum_{n\in\Z^4} h_n g_n (\omega) \chi_{Q_0}(D-n) \phi
=h_0g_0(\omega)\chi_{Q_0}(D)\phi
+\text{Re}\,\sum_{n\in\mathcal I} h_ng_n (\omega) \chi_{Q_0}(D-n) \phi\notag\\
&=h_0g_0(\omega)\chi_{Q_0}(D)\phi
+\sum_{n\in\mathcal I} \Big(h_n \text{Re}\, g_n (\omega) \text{Re}\,\chi_{Q_0}(D-n) \phi
-h_n \text{Im}\, g_n (\omega) \text{Im}\,\chi_{Q_0}(D-n) \phi\Big).\notag
\label{hg}
\end{align}

Let $\{h_{n,j}\}_{n\in\{0\}\cup\mathcal I}$, $j=0,1$
be two systems of independent Bernoulli random variables
such that $\left\{h_{0,j}, g_{0,j}, h_{n,j}, \text{Re}\, g_{n,j}, \text{Im}\, g_{n,j}\right\}_{n\in\mathcal I, j=0,1}$ are independent.
For $j=0,1$, we then 
consider the following random variables
with values in  
$\R^{\Z^4}$
endowed with the $\sigma$-algebra 
$\mathcal{B}(\R)^{\otimes \Z^4}$:
\begin{align*}
g_j:=\left(
\begin{matrix}
g_{0,j}\\
\text{Re}\, g_{n,j}\\
\text{Im}\, g_{n,j}
\end{matrix}
\right)_{n\in\mathcal I}, 
\quad 
h_jg_j:=\left(
\begin{matrix}
h_{0,j}g_{0,j}\\
h_{n,j} \text{Re}\, g_{n,j}\\
h_{n,j} \text{Im}\, g_{n,j}
\end{matrix}
\right)_{n\in\mathcal I}
: \Omega
\to \R\times \R^{\mathcal I}\times\R^{\mathcal I}=\R^{\Z^4}.
\end{align*}

\noindent
By the independence 
of $\left\{h_{0,j}, g_{0,j}, h_{n,j}, \text{Re}\, g_{n,j}, \text{Im}\, g_{n,j}\right\}_{n\in\mathcal I, j=0,1}$ and the fact that $g_{0,j}$, $\text{Re}\, g_{n,j}$, and $\text{Im}\, g_{n,j}$ are symmetric, it then follows using Lemma \ref{lemma:A9} that the probability distributions $\theta_{h_jg_j}$ and $\theta_{g_j}$ of $h_jg_j$ and $g_j$ respectively, coincide.
As a consequence, if $\mu= \mu_{(u_0,u_1)}$ for some $(u_0,u_1)\in H^s(\R^4)\times H^{s-1}(\R^4)$,
we have for all $\lambda>0$ that
\begin{align*}
\mu&\otimes\mu_0\otimes \mu_0
\left( (w_0,w_1)\in H^{s}\times H^{s-1}, (h_0,h_1)\in Y\times Y: 
\big\| S(t)\big(h_0\odot w_0, h_1\odot w_1\big)\big\|_{L^q_tL^r_x}>\lambda \right)\\
&=P\left(
\left\|S(t)\left(h_0(\omega)g_0(\omega)\odot u_0, h_1(\omega)g_1(\omega)\odot u_1\right)\right\|_{L^q_tL^r_x}>\lambda \right)\\
&=\int_{\R^{\Z^4}\times\R^{\Z^4}} \pmb{1}_{\left\{\left\|S(t)\left(y\odot u_0, y'\odot u_1\right)\right\|_{L^q_tL^r_x}>\lambda \right\}}d\theta_{h_0g_0}(y)d\theta_{h_1g_1}(y')\\
&=\int_{\R^{\Z^4}\times\R^{\Z^4}} \pmb{1}_{\left\{\left\|S(t)\left(y\odot u_0, y'\odot u_1\right)\right\|_{L^q_tL^r_x} >\lambda \right\}}d\theta_{g_0}(y)d\theta_{g_1}(y')\\
&=P\left(\left\|S(t)\left(g_0(\omega)\odot u_0, g_1(\omega)\odot u_1\right)\right\|_{L^q_tL^r_x} >\lambda \right)\\
&=\mu
\left( (w_0,w_1)\in H^{s}(\R^4)\times H^{s-1}(\R^4): \|S(t)(w_0, w_1)\|_{L^q_tL^r_x} >\lambda \right).
\end{align*}

\noindent
Arguing analogously, we obtain that the following two quantities are equal:
\begin{align*}
M_1:=\mu \otimes \mu\Big(&
\left((w_0,w_1), (w_0',w_1')\right)\in \left(H^{s}(\R^4)\times H^{s-1}(\R^4)\right)^2:\\
&
\left\|S(t)(w_0-w_0', w_1-w_1')\right\|_{L^q_t([0,T]; L^r_x)}>\lambda
\quad \Big| \quad\\
&\left\|(w_0-w_0', w_1-w_1')\right\|_{H^{s}\times H^{s-1}}\leq \eps,\,
\left\|(w_0+w_0', w_1+w_1')\right\|_{H^{s}\times H^{s-1}}\leq R\Big)\\
=\mu \otimes \mu\otimes \mu_0\otimes \mu_0\Big(
&\left((w_0,w_1), (w_0',w_1')\right)\in \big(H^{s}(\R^4)\times H^{s-1}(\R^4)\big)^2, (h_0,h_1)\in Y\times Y:\\
&\left\|S(t)\left(h_0\odot(w_0-w_0'), h_1\odot(w_1-w_1')\right)\right\|_{L^q_t([0,T]; L^r_x)}>\lambda
\quad \Big| \quad\\
&\left\|\left(h_0\odot(w_0-w_0'), h_1\odot(w_1-w_1')\right)\right\|_{H^{s}\times H^{s-1}}\leq \eps\notag\\
&
\text{  and } \left\|\left(h_0\odot(w_0+w_0'), h_1\odot(w_1+w_1')\right)\right\|_{H^{s}\times H^{s-1}}\leq R\Big).
\end{align*}

\noindent
Noticing by \eqref{h} and \eqref{support} that $\|h\odot \phi\|_{H^\sigma}=\|\phi\|_{H^\sigma}$ for any $\phi\in H^\sigma$, $\sigma\in\R$, and any Bernoulli random variable $h$,
it follows that
\begin{align*}
M_1=\mu \otimes \mu\otimes \mu_0\otimes & \mu_0\Big(
\left((w_0,w_1), (w_0',w_1')\right)\in \big(H^{s}(\R^4)\times H^{s-1}(\R^4)\big)^2, (h_0,h_1)\in Y\times Y:\\
&\left\|S(t)\left(h_0\odot(w_0-w_0'), h_1\odot(w_1-w_1')\right)\right\|_{L^q_t([0,T]; L^r_x)}>\lambda
\quad \Big| \quad\\
&\left\|(w_0-w_0', w_1-w_1')\right\|_{H^{s}\times H^{s-1}}\leq \eps,\,
\left\|(w_0+w_0', w_1+w_1')\right\|_{H^{s}\times H^{s-1}}\leq R\Big).
\end{align*}

\noindent
Then, by Lemma \ref{lemma:A8},
we obtain that
\begin{align*}
M_1\leq \sup_{\left\|(w_0-w_0', w_1-w_1')\right\|_{H^{s}\times H^{s-1}}\leq \eps}
&\mu_0\otimes \mu_0
\Big((h_0,h_1)\in Y\times Y: \\
&\left\|S(t)\left(h_0\odot(w_0-w_0'), h_1\odot(w_1-w_1')\right)\right\|_{L^q_t([0,T]; L^r_x)}>\lambda\Big).
\end{align*}

\noindent
By the improved local-in-time Strichartz estimates in Proposition \ref{proba_S} (ii) and (iii)
with Bernoulli random variables,
it then follows that
\[M_1\leq Ce^{-c\frac{\lambda^2}{T^{2+\frac 2q}\eps^2}}.\]

\noindent
Similarly, we obtain
\begin{align*}
M_2:&=\mu \otimes \mu \Big(
\left((w_0,w_1), (w_0',w_1')\right)\in \big(H^{s}(\R^4)\times H^{s-1}(\R^4)\big)^2:\\
&\hphantom{XXXx}\left\|S(t)(w_0+w_0', w_1+w_1')\right\|_{L^q_t([0,T]; L^r_x)}>\Lambda
\quad \Big| \quad\\
&\hphantom{XXXx}\left\|(w_0-w_0', w_1-w_1')\right\|_{H^{s}\times H^{s-1}}\leq \eps
\text{  and } \left\|(w_0+w_0', w_1+w_1')\right\|_{H^{s}\times H^{s-1}}\leq R\Big)\\
&\leq \sup_{\left\|(w_0+w_0', w_1+w_1')\right\|_{H^{s}\times H^{s-1}}\leq R}
\mu_0\otimes \mu_0
\Big((h_0,h_1)\in Y\times Y: \\
&\hphantom{XXXXXXXXXXXXX}\left\|S(t)\big(h_0\odot(w_0+w_0'), h_1\odot(w_1+w_1')\big)\right\|_{L^q_t([0,T]; L^r_x)}>\Lambda\Big)\\
&\leq Ce^{-c\frac{\Lambda^2}{T^{2+\frac 2q}R^2}}.
\end{align*}

\noindent
Therefore, \eqref{mu_times_mu} follows.
\end{proof}


In the remaining of this section, we prove the probabilistic continuity of the flow
with respect to the initial data in Theorem \ref{proba_cont}.

\begin{proof}[Proof of Theorem \ref{proba_cont}]
For simplicity of notations, in the following we assume $T\geq 1$.
The case $0<T<1$ is completely analogous.

\noindent 
{\bf Step 1} (Control of the linear parts of the solutions and of their difference).

Let $\{\eta_k\}_{k\in\N}$ be such that $\eta_k \in (0,1)$ and $\eta_k\searrow 0$ as $k\to\infty$.
By Proposition \ref{cond_proba}, for any 
$(q,r)\in \{ (3,6), (3,\infty), (2,8)\}$,
$\mu\in \mathcal{M}_{s}$,
$\alpha\in (0,\frac 12)$, 
and $L(\eta_k, T, R,\alpha)=L(\eta_k)$ to be chosen later such that $L(\eta_k)\to\infty$ as $k\to\infty$, we have that
\begin{align*}
\mu \otimes \mu\Big(
&((w_0,w_1), (w_0',w_1'))\in \big(H^{s}(\R^4)\times H^{s-1}(\R^4)\big)^2:
\|S(t)(w_0-w_0', w_1-w_1')\|_{L^q_TL^r_x}>\eta_k^{1-\alpha} \\
&\text{ or }
\|S(t)(w_0, w_1)\|_{L^q_TL^r_x}>L(\eta_k)
 \text{ or } \|S(t)(w_0', w_1')\|_{L^q_TL^r_x}>L(\eta_k)\\
&\quad \Big| \quad
\|(w_0-w_0', w_1-w_1')\|_{H^{s}(\R^4)\times H^{s-1}(\R^4)}\leq \eta_k,
(w_0, w_1), (w_0', w_1')\in B_R\Big)\\
& \leq C\left(e^{-\frac{c}{T^3\eta_k^{2\alpha}}}+ e^{-c\frac{L(\eta_k)^2}{T^3R^2}}\right)(1+ o(1))
\to 0 \text{ as } k\to\infty.
\end{align*}

\noindent
Therefore, when we estimate the conditional probability in Theorem \ref{proba_cont}, we can assume that 
\begin{align}
&\|z-z'\|_{L^3_TL^r_x}\leq \eta_k^{1-\alpha}, \,\,
\|z\|_{L^3_TL^r_x}\leq L(\eta_k), \, \text{ and } \,
\|z'\|_{L^3_TL^r_x}\leq L(\eta_k) \, \text{ for } \, r=6,\infty,\label{zz'}\\
&\|z-z'\|_{L^2_TL^8_x}\leq \eta_k^{1-\alpha}, \quad
\|z\|_{L^2_TL^8_x}\leq L(\eta_k), \quad \text{and} \quad
\|z'\|_{L^2_TL^8_x}\leq L(\eta_k),\label{zz'28}
\end{align}
where we set $z(t): = S(t)(w_0,w_1)$ and $z'(t):=S(t)(w_0', w_1')$.
We also denote by $v(t):=\Phi(t)(w_0,w_1)-z(t)$ and $v'(t):=\Phi(t)(w_0',w_1')-z'(t)$
the nonlinear parts of the solutions $\Phi(t)(w_0,w_1)$ and $\Phi(t)(w'_0,w'_1)$, respectively.


\medskip

\noindent
{\bf Step 2} (Control of Strichartz norms of the nonlinear parts of the solutions).

In the following we prove that
there exists $G(\eta_k)>0$ such that
\begin{align}
M(\eta_k):=\mu \otimes \mu\Big(
&((w_0,w_1), (w_0',w_1'))\in \big(H^{s}(\R^4)\times H^{s-1}(\R^4)\big)^2:\notag\\
&\|v\|_{L^3_{T}L^6_x}+\|v\|_{L^2_{T}L^8_x}>  G(\eta_k) \text{ or }
\|v'\|_{L^3_{ T}L^6_x}+\|v'\|_{L^2_{ T}L^8_x}>  G(\eta_k);\notag \\
&\|z\|_{L^3_TL^r_x}\leq L(\eta_k) \text{ and }
\|z'\|_{L^3_TL^r_x}\leq L(\eta_k) \text{ for }  r=6,\infty \notag\\
&\quad \Big| \quad
\|(w_0-w_0', w_1-w_1')\|_{H^{s}(\R^4)\times H^{s-1}(\R^4)}\leq \eta_k\notag\\
&
\hphantom{XXX}\text{  and } (w_0, w_1)\in B_R \text{ and }
 (w_0', w_1')\in B_R\Big)\to 0 \text{ as } k\to\infty.\label{M(eta)}
\end{align}

\noindent
As a consequence, in addition to \eqref{zz'} and \eqref{zz'28}, we can assume that 
\begin{align}
\|v\|_{L^3_TL^6_x}
&+\|v\|_{L^2_TL^8_x}
\leq G(\eta_k)\label{vv'Strich}\\
\|v'\|_{L^3_TL^6_x}
&+\|v'\|_{L^2_TL^8_x}
\leq G(\eta_k),\label{vv'Strich28}
\end{align}

\noindent
when we estimate the conditional probability in
Theorem \ref{proba_cont}.

Arguing as in the proof of Proposition \ref{cond_proba}, 
we have that
\begin{align*}
M(\eta_k)\leq \sup\Big\{&\mu_0\otimes \mu_0\Big((h_0,h_1)\in Y\times Y:
\|\wt v\|_{L^3_{ T}L^6_x}+\|\wt v\|_{L^2_{ T}L^8_x}> G(\eta_k) \text{ or }\\
&\quad \quad \quad \quad \|\wt v'\|_{L^3_{ T}L^6_x}+\|\wt v'\|_{L^2_{ T}L^8_x}>  
G(\eta_k); \\ 
&\quad \quad \quad \quad \|\wt z\|_{L^3_TL^r_x}\leq L(\eta_k)
\text{ and } \|\wt z'\|_{L^3_TL^r_x}\leq L(\eta_k) \text{ for } r=6,\infty\Big):\\
&
\hphantom{XX}\|(w_0-w_0', w_1-w_1')\|_{H^{s}(\R^4)\times H^{s-1}(\R^4)}\leq \eta_k, \quad (w_0, w_1), \, (w_0', w_1')\in B_R\Big\},
\end{align*} 

\noindent
where we denoted by $\wt z$ and $\wt v$ the linear and nonlinear parts of the solution 
with initial data $(h_0\odot w_0, h_1\odot w_1)$,
and by $\wt z '$ and $\wt v '$ the linear and nonlinear parts of the solution 
with initial data $(h_0\odot w_0', h_1\odot w_1')$.
Then, we can upper bound $M (\eta_k)$ by
\begin{align}
M(\eta_k)&\leq \sup_{ (w_0, w_1)\in B_R}
\mu_0\otimes \mu_0 \Big((h_0,h_1)\in Y\times Y:
\|\wt v\|_{L^3_{ T}L^6_x}+\|\wt v\|_{L^2_{ T}L^8_x}>  G(\eta_k),\notag\\
&\hphantom{XXXXXXXXXXXXXXXXXXXi}
 \|\wt z\|_{L^3_TL^r_x}\leq L(\eta_k) \text{ for } r=6, \infty\Big)\notag\\
&+\sup_{ (w_0', w_1')\in B_R}
\mu_0\otimes \mu_0 \Big((h_0,h_1)\in Y\times Y:
\|\wt v'\|_{L^3_{ T}L^6_x}+\|\wt v'\|_{L^2_{ T}L^8_x}>  G(\eta_k),\notag\\
&\hphantom{XXXXXXXXXXXXXXXXXXXi}
 \|\wt z'\|_{L^3_TL^r_x}\leq L(\eta_k) \text{ for } r=6, \infty\Big)\notag\\
 &=2\sup_{ (w_0, w_1)\in B_R}
P\Big(
\omega\in \tilde{\Omega} (w_0,w_1): \|\wt v^\omega\|_{L^3_{ T}L^6_x}+\|\wt v^\omega\|_{L^2_{ T}L^8_x}>  G(\eta_k),\notag\\
&\hphantom{XXXXXXXXXXXXXXXXXXi}
 \|\wt z^\omega\|_{L^3_TL^r_x}\leq L(\eta_k) \text{ for } r=6,\infty\Big)\notag\\
 &=:2\sup_{ (w_0, w_1)\in B_R}M_{(w_0,w_1)}(\eta_k).
 \label{Mm}
\end{align} 

\noindent
Here, the set 
$\tilde{\Omega}(w_0,w_1)\subset \Omega$ with $P(\wt \Omega(w_0,w_1))=1$ was defined in Theorem \ref{GWP}
and has the property that for all $\omega\in \tilde \Omega (w_0,w_1)$,
\eqref{NLW} admits a unique global solution $\tilde u^\omega$ with initial data $\left(\tilde u^\omega (0), \pa_t \tilde u^\omega (0)\right)=\left(h_0(\omega)\odot w_0, h_1(\omega)\odot w_1\right)$. 
We also denoted by
$\wt z^\omega$ and $\wt v^\omega$ 
the linear and nonlinear parts of the solution $\tilde u^\omega$.

We choose $\eps_k=\frac{T}{\log\frac{1}{\eta_k}}$ and
consider the set
\begin{align*}
\tilde \Omega_1(w_0,w_1, \eta_k):=\Big\{\omega\in \tilde\Omega_1 (w_0,w_1)\cap \Omega_{T,\eps_k}(w_0,w_1): \|\wt v^\omega &\|_{L^3_{T}L^6_x}+\|\wt v^\omega\|_{L^2_{T}L^8_x}>  G(\eta_k)
\text{ and }\\
&\|\wt z^\omega\|_{L^3_TL^r_x}\leq L(\eta_k) \text{ for } r=6,\infty \Big\},
\end{align*}

\noindent
where $\Omega_{T,\eps_k}(w_0,w_1)$
was defined in Proposition \ref{almost_almost}
and is such that $P\big(\Omega_{T,\eps_k}^c(w_0,w_1)\big)<\eps_k$.
We then have that
\begin{align*}
M_{(w_0,w_1)}(\eta_k)\leq 
 P\left(\tilde \Omega_1(w_0,w_1, \eta_k)\right)+\eps_k.
\end{align*}

\noindent
Next, we show that
there exists $G (\eta_k)=G\left(\eta_k, L(\eta_k), T, R\right)$
such that, if $ (w_0, w_1)\in B_R$,
 then 
$\tilde \Omega_1(w_0,w_1, \eta_k)=\emptyset$.
In particular, this shows that
\begin{equation}\label{M_u_0_u_1}
M_{(w_0,w_1)}(\eta_k)\leq \frac{T}{\log\frac{1}{\eta_k}} \to 0 \quad \text{ as } \quad k\to\infty.
\end{equation}

First,
for all $\omega \in \tilde\Omega_1 (w_0,w_1,\eta_k)$, we have
$ \|\wt z^\omega\|_{L^3_t(\R,L^r_x)}\leq L(\eta_k)$ for $r=6, \infty$,
and, as in \eqref{energy_estim}, we obtain
\begin{align*}
E(\wt v^\omega(t))&\leq C\|\wt z^\omega\|_{L^3_TL^6_x}^6e^{C\|\wt z^\omega\|_{L^1_TL^{\infty}_x}}
\leq C\|\wt z^\omega\|_{L^3_TL^6_x}^6e^{CT^{\frac 23}\|\wt z^\omega\|_{L^3_TL^{\infty}_x}}
\leq CL(\eta_k)^6e^{CT^{\frac 23}L(\eta_k)}\\
&\leq e^{CT^{\frac 23}L(\eta_k)}.
\end{align*}
Then,  for $ (w_0, w_1)\in B_R$,
we have by \eqref{estim_strich}
that
\begin{align}\label{G}
\|\wt v^\omega\|_{L^3_TL^6_x}
+\|\wt v^\omega\|_{L^2_TL^8_x}
&\leq \tilde{F}\left(e^{C_1T^{\frac 23}L(\eta_k)}, R\right)T^7\left(\log \frac{T}{\eps}\right)^3\\
&=F_R\left(e^{C_1T^{\frac 23}L(\eta_k)}\right)T^7\left(\log\log\frac{1}{\eta_k}\right)^3
=:G(\eta_k),\notag
\end{align}

\noindent
where $C_1\geq 1$ is an absolute constant.
In the second to the last equality we used the definition $F_R(x):=\wt F(x,R)$
with $\wt F$ as in \eqref{estim_strich}.
This fixes $G(\eta_k)$ and shows that, with this choice,  
$\tilde \Omega_1 (w_0,w_1,\eta_k)=\emptyset$. Thus, \eqref{M_u_0_u_1}
holds for all $(w_0, w_1)\in B_R$. 
Combining \eqref{M_u_0_u_1} with \eqref{Mm}, 
we obtain that $M(\eta_k)\to 0 $ as $k\to\infty$, which proves \eqref{M(eta)}. 

Recall that $\wt F$ defined in \eqref{estim_strich} is a non-decreasing function in both its variables.  
By increasing $F_R=F(\cdot, R)$ if needed, we can choose it to be strictly increasing 
and, moreover, to satisfy
$F_R(x)\geq x$ for all $x>0$. In particular, we have
\begin{equation}\label{estim_F}
F_R\left(e^{C_1T^{\frac 23}L(\eta_k)}\right)\geq e^{C_1T^{\frac 23}L(\eta_k)}
\geq C_1 T^{\frac 23}L(\eta_k).
\end{equation}

\noindent
Then, by \eqref{estim_F}, \eqref{G}, and $T\geq 1$, we have that
\begin{equation}\label{LG}
C_1L(\eta_k)< G(\eta_k).
\end{equation}


\medskip

\noindent
{\bf Step 3} (Control of the difference of the nonlinear parts of the solutions).

For the remaining of the proof, we assume that the bounds 
on the linear and nonlinear parts of solutions given in \eqref{zz'}, \eqref{zz'28}, \eqref{vv'Strich}, and \eqref{vv'Strich28} hold.
 
Using the equations satisfied by $v$ and $v'$,
we have
\begin{align*}
\frac{d}{dt}&\left\|\left(v(t)-v'(t),\pa_tv(t)-\pa_t v'(t)\right)\right\|^2_{\dot{H}^1_x\times L^2_x}
\leq 2\left|\int_{\R^4}\pa_t \left(v(t)-v'(t)\right)\left(\pa_{t}^2-\Delta\right)\left(v(t)-v'(t)\right)dx\right|\\
&\leq 2 \left\|\left(v(t)-v'(t),\pa_tv(t)-\pa_t v'(t)\right)\right\|_{\dot{H}^1_x\times L^2_x}
\left\|(v+z)^3-(v'+z')^3\right\|_{L^2_x}.
\end{align*}

\noindent
Then, using the Sobolev embedding $\dot{H}^1(\R^4)\subset L^4(\R^4)$ and H\"older's inequality, it follows that
\begin{align*}
\frac{d}{dt}&\left\|\left(v(t)-v'(t),\pa_tv(t)-\pa_t v'(t)\right)\right\|_{\dot{H}^1_x(\R^4)\times L^2_x(\R^4)}\\
&\leq C\left\|\left(v(t)-v'(t),\pa_t v(t)-\pa_t v'(t)\right)\right\|_{\dot{H}^1_x\times L^2_x}
\left(\|v(t)\|_{L^8_x}^2+\|v'(t)\|_{L^8_x}^2+\|z(t)\|_{L^8_x}^2+\|z'(t)\|_{L^8_x}^2\right)\\
&\hphantom{XXx}+C\|z(t)-z'(t)\|_{L^6_x}
\left(\|v(t)\|_{L^6_x}^2+\|v'(t)\|_{L^6_x}^2+\|z(t)\|_{L^6_x}^2+\|z'(t)\|_{L^6_x}^2\right).
\end{align*}

\noindent
Integrating in time from 0 to $t\leq T$ and using H\"older's inequality,
we then have that
\begin{align*}
\big\|\big(v(t)&-v'(t),\pa_tv(t)-\pa_t v'(t)\big)\big\|_{\dot{H}^1_x(\R^4)\times L^2_x(\R^4)}\\
&\leq C\|z-z'\|_{L^3([0,t]; L^6_x)}
\left(\|v\|_{L^3([0,t]; L^6_x)}^2+\|v'\|_{L^3([0,t]; L^6_x)}^2+\|z\|_{L^3([0,t]; L^6_x)}^2+\|z'\|_{L^3([0,t]; L^6_x)}^2\right)\\
&\hphantom{XXx}+C\int_0^t \left\|\left(v(t')-v'(t'),\pa_t v(t')-\pa_t v'(t')\right)\right\|_{\dot{H}^1_x(\R^4)\times L^2_x(\R^4)}\\
&\hphantom{XXXXXXXXXX}\times \left(\|v(t')\|_{L^8_x}^2+\|v'(t')\|_{L^8_x}^2+\|z(t')\|_{L^8_x}^2+\|z'(t')\|_{L^8_x}^2\right)dt'.
\end{align*}

\noindent
Then, Gronwall's inequality yields for any $t\in [0,T]$ that
\begin{align*}
\big\|\big(v(t)&-v'(t),\pa_tv(t)-\pa_t v'(t)\big)\big\|_{\dot{H}^1_x(\R^4)\times L^2_x(\R^4)}\\
&\leq C\|z-z'\|_{L^3([0,t]; L^6_x)}
\Big(\|v\|_{L^3([0,t],L^6_x)}^2+\|v'\|_{L^3([0,t]; L^6_x)}^2+\|z\|_{L^3([0,t]; L^6_x)}^2+\|z'\|_{L^3([0,t]; L^6_x)}^2\Big)\\
& \hphantom{XXXXXXXXXX} \times e^{C\left(\|v\|_{L^2([0,t]; L^8_x)}^2+\|v'\|_{L^2([0,t]; L^8_x)}^2+\|z\|_{L^2([0,t]; L^8_x)}^2+\|z'\|_{L^2([0,t]; L^8_x)}^2\right)}.
\end{align*}
By \eqref{zz'}, \eqref{zz'28}, \eqref{vv'Strich}, and \eqref{vv'Strich28},
it then follows for $t\in [0,T]$ that
\begin{align*}
\left\|\left(v(t)-v'(t),\pa_tv(t)-\pa_t v'(t)\right)\right\|_{\dot{H}^1_x(\R^4)\times L^2_x(\R^4)}
\leq C\eta_k^{1-\alpha} \Big(L(\eta_k)^2+G(\eta_k)^2\Big)e^{C(G(\eta_k)^2+L(\eta_k)^2)}.
\end{align*}
Then, by \eqref{LG}, it follows easily for $t\in [0,T]$ that
\begin{align*}
\left\|\left(v(t)-v'(t),\pa_tv(t)-\pa_t v'(t)\right)\right\|_{\dot{H}^1_x(\R^4)\times L^2_x(\R^4)}
\leq C_2\eta_k^{1-\alpha} e^{C_3G(\eta_k)^2},
\end{align*}

\noindent
where $C_2$ and $C_3$ are absolute constants. 
Furthermore,
\begin{align*}
\|v(t)-v'(t)\|_{L^2_x(\R^4)}
&\leq \int_0^t \left\|\pa_t v(t')-\pa_t v'(t')\right\|_{L^2_x}dt'\leq t \|\pa_t v(t')-\pa_t v'(t')\|_{L^\infty_t([0,T]; L^2_x)}\\
&\leq C_2T\eta_k^{1-\alpha} e^{C_3G(\eta_k)^2}.
\end{align*}

\noindent
Hence, we obtained
\begin{align}\label{diff_v}
\left\|\left(v-v',\pa_tv-\pa_t v'\right)\right\|_{L^{\infty}_t([0,T]; H^1_x(\R^4)\times L^2_x(\R^4))}
\leq 2C_2T\eta_k^{1-\alpha} e^{C_3G(\eta_k)^2}.
\end{align}

In the following, we discuss the choice of $L(\eta_k)$.
The two conditions that we need to impose on $L(\eta_k)$
are $L(\eta_k)\to\infty$ as $k\to\infty$, which is crucial in Step 1, and that the right hand side of \eqref{diff_v}
tends to zero as $k\to\infty$.

Recall that $F_R:[0,\infty)\to [0,\infty)$ is a strictly increasing function
satisfying $F_R(0)=0$ and $\lim_{A\to\infty}F_R(A)=\infty$.
In particular, $F_R$ has at most countably many 
discontinuities. These are jump discontinuities that we denote by
\[0\leq x_1< x_2 <\dots <x_n<\dots .\]
We claim that, given a sequence $\{y_k\}_{k\in\N}\subset (0,\infty)$, $y_k\nearrow \infty$, 
there exists another sequence $\{y'_k\}_{k\in\N}\subset [0,\infty)$ such that $y'_k\in \text{Ran } F_R$, 
$y'_k\leq y_k$ for all $k$, and
$y'_k \to\infty$ as $k\to\infty$.

Indeed, if $y_k\in \text{Ran } F_R$, then we choose $y'_k:=y_k$. 
Otherwise, if $y_k\notin \text{Ran } F_R$,
it follows that $y_k\in [F_R(x_{n_k}-), F_R(x_{n_k}+)]$ for some $n_k\in\N$.
In this case, if $n_k\geq 2$, 
we choose $y'_k:=F_R(x_{n_{k-1}})$, 
otherwise we choose $y'_k=0$. 
Clearly, we have $y'_k\leq y_k$
and $y'_k\in \text{Ran } F_A $
for all $k$. 
We then denote by $\{y'_{k_1}\}_{k_1\in \N}$ and $\{y'_{k_2}\}_{k_2\in\N}$
the subsequence of $\{y'_k\}_{k\in\N}$
corresponding to $y_k$ in $\text{Ran } F_R$, 
respectively corresponding to $y_k$ in $\left(\text{Ran } F_R\right)^c$.
One of these two subsequences is necessarily infinite.
Clearly, we have that either $y'_{k_1}\nearrow \infty$ or $\{y'_{k_1}\}_{k_1\in\N}$ is a finite set. Also, $\{y'_{k_2}\}_{k_2\in\N}$ is either a non-decreasing sequence
converging to infinity or a finite set. 
This shows that $y'_k\to\infty$ as $k\to\infty$.

We apply the above reasoning to the sequence
\[y_k:= \sqrt{\frac{1-2\alpha}{2C_3T^{2}}}\Big(\log\frac{1}{\eta_k}\Big)^{\frac 14} \nearrow \infty \text{  as } k\to\infty.\]
As a consequence, 
there exists $\eta'_k\in (0,1)$, $\eta'_k\geq \eta_k$
such that 
\begin{align*}
\sqrt{\frac{1-2\alpha}{2C_3T^{2}}}\Big(\log\frac{1}{\eta'_k}\Big)^{\frac 14}\in \text{Ran }F_R
\quad \text{for all} \quad k\in\N
\quad \text{ and } \quad \eta'_k\to 0 \text{ as } k\to \infty.
\end{align*}
 
 \noindent
The function $F_R$ being invertible on its range, we choose $L(\eta_k)$ as
\begin{align}\label{L}
L(\eta_k)
&:=\frac{1}{C_1 T^{\frac 23}}\log\Bigg[F_R^{-1}\Bigg(\sqrt{\frac{1-2\alpha}{2C_3T^{14}}}\Big(\log \frac {1}{\eta'_k}\Big)^{\frac 14}\Bigg)\Bigg].
\end{align}
This guarantees that
$L(\eta_k)\to\infty$ as $k\to\infty$, since $\eta'_k\to 0$ and 
$\lim_{\substack{y\in\text{Ran }F_R\\ y\to\infty} }F_R^{-1}(y)=\infty$.
Moreover, by \eqref{diff_v},
the choice of $G(\eta_k)$ and $L(\eta_k)$
in \eqref{G} and \eqref{L}, and $\eta'_k>\eta_k$, we have
\begin{align}
\big\|\big(v-v',\pa_tv-\pa_t v\big)\big\|_{L^\infty_t([0,T]; H^1_x(\R^4)\times L^2_x(\R^4))}
&\leq 2C_2T\eta_k^{1-\alpha} e^{\left(\frac 12-\alpha\right)\left(\log\frac{1}{\eta'_k}\right)^{\frac 12}\left(\log\log\frac{1}{\eta_k}\right)^6}\notag\\
&\leq 2C_2T\eta_k^{1-\alpha} e^{\left(\frac 12-\alpha\right)\log\frac{1}{\eta_k}}
\leq 2C_2T\eta_k^{\frac 12},\label{vv'}
\end{align}
for $k$ sufficiently large. 

On the other hand, notice that the condition
$ \|(w_0,w_1)-(w_0',w_1')\|_{H^{s}\times H^{s-1}}<\eta_k$
immediately implies 
the control on the difference of the linear parts of solutions:
\[\left\|\left(z-z', \pa_tz-\pa_t z'\right)\right\|_{L^{\infty}_t([0,T]; H^{s}_x(\R^4)\times H^{s-1}_x(\R^4))}\leq T\eta_k.\]
Hence, we obtained that
\begin{align*}
&\left\|\left(\Phi(t)(w_0,w_1)-\Phi(t)(w_0',w_1'),
\pa_t \Phi(t)(w_0,w_1)-\pa_t \Phi(t)(w_0',w_1')\right)\right\|_{L^\infty_t\left([0,T]; H^{s}_x(\R^4)\times H^{s-1}_x(\R^4)\right)}\\
&\hphantom{XXXXXXXXX}\leq  3C_2 T \eta_k^{\frac 12}.
\end{align*}
Therefore, for a fixed $\delta>0$,
the $\mu \otimes \mu$-measure
of $\left((w_0,w_1), (w_0',w_1')\right)\in B_R\times B_R$
such that 
\[\left\|\left(\Phi(t)(w_0,w_1)-\Phi(t)(w_0',w_1'),
\pa_t \Phi(t)(w_0,w_1)-\pa_t \Phi(t)(w_0',w_1')\right)\right\|_{L^\infty_t\left([0,T]; H^{s}_x\times H^{s-1}_x\right)}>\delta\]
under the constraints $\|(w_0,w_1)-(w_0', w_1')\|_{H^{s}(\R^4)\times H^{s-1}(\R^4)}<\eta_k$, \eqref{zz'}, 
\eqref{zz'28}, \eqref{vv'Strich}, and \eqref{vv'Strich28}, is zero if $\eta_k$ is sufficiently small. 
This shows that the right hand-side of \eqref{prob_cont_eq}
converges indeed to zero as $\eta\to 0$.\end{proof}


\begin{ackno}
{\rm 
The author would like to thank Tadahiro Oh
for suggesting this problem and for his support during the preparation of this paper.
She would like to express her gratitude to Prof.\,Nicolas Burq for a suggestion that lead to an improvement of the main result.
She is also grateful to Prof.\,Monica Vi\c{s}an, Prof.\,Patrick G\'erard, and Prof.\,\'Arpad B\'enyi
for their availability in answering her questions.
Finally, she would like to thank Prof.\,Andrea Nahmod and 
Prof.\,Jeremy Quastel for helpful discussions. 
}
\end{ackno}


\end{document}